\documentclass{article}
\usepackage{graphicx} %

\usepackage{amsmath,amssymb,amsthm,amsfonts,latexsym,bbm,xspace,graphicx,float,mathtools,mathdots}
\usepackage{braket,caption,subcaption,ellipsis,xcolor,textcomp}
\usepackage{combelow} %
\usepackage[citecolor=blue]{hyperref}
\usepackage[nameinlink]{cleveref}
\crefname{ineq}{inequality}{inequalities}
\creflabelformat{ineq}{#2{\upshape(#1)}#3}

\usepackage[letterpaper,margin=1in]{geometry}
\usepackage{enumitem} 

\newtheorem{theorem}{Theorem}

\usepackage{chngcntr}
\counterwithin{theorem}{section}

\usepackage{tikz}
\usepackage{pgfplots}
\pgfplotsset{compat=1.9}

\newtheorem{lemma}[theorem]{Lemma}

\newtheorem{proposition}[theorem]{Proposition}
\newtheorem{fact}[theorem]{Fact}

\newtheorem{assumption}[theorem]{Assumption}

\theoremstyle{definition}
\newtheorem{definition}[theorem]{Definition}
\newtheorem{remark}[theorem]{Remark}

\DeclareMathOperator*{\argmax}{arg\,max}

\newcommand{\ignore}[1]{}

\newcommand{\defined}{\coloneqq}
\usepackage{natbib}
\usepackage{autonum} 

\newcommand{\lp}{\left(}
\newcommand{\rp}{\right)}
\newcommand{\lb}{\left[}
\newcommand{\lbr}{\left\{}
\newcommand{\rbr}{\right\}}
\newcommand{\rb}{\right]}
\newcommand{\lv}{\left\lvert}
\newcommand{\rv}{\right \rvert}

\newcommand{\Xsample}{\mathcal{X}\xspace}

\newcommand{\muhat}{\widehat{\mu}}
\newcommand{\mutilde}{\widetilde{\mu}}
\newcommand{\sigmatilde}{\widetilde{\sigma}}

\newcommand{\Phat}{\widehat{P}}

\newcommand{\sigmahat}{\widehat{\sigma}}

\newcommand{\convdist}{\stackrel{d}{\longrightarrow}}
\newcommand{\convas}{\stackrel{a.s.}{\longrightarrow}}

\newcommand{\reals}{\mathbb{R}}
\newcommand{\iid}{\text{i.i.d.}\xspace}

\newcommand{\mc}[1]{\mathcal{#1}}

\newcommand{\wor}{WoR\xspace}

\newcommand{\klinf}{\mathrm{KL}^{+}_{\inf}}
\newcommand{\klinfminus}{\mathrm{KL}^{-}_{\inf}}
\newcommand{\dkl}{d_{\mathrm{KL}}}

\newcommand{\eqas}{\stackrel{a.s.}{=}}
\newcommand{\simiid}{\stackrel{\iid}{\sim}}

\newcommand{\hatKplus}{\widehat{K}^{+}_n}
\newcommand{\hatKminus}{\widehat{K}^{-}_n}

\newcommand{\Kmax}{K^{\max}}

\newcommand{\gammabet}{\gamma_1^{(\mathrm{bet})}}
\newcommand{\widthbet}{w_n^{(\mathrm{bet})}}
\newcommand{\widthfan}{w_n^{(\dag)}}

\newcommand{\CIbet}{C_n^{(\mathrm{bet})}}
\newcommand{\CIfan}{C_n^{\dag}}
\newcommand{\CItilde}{\widetilde{C}_n}
\newcommand{\Deltahat}{\widehat{\Delta}}
\newcommand{\prpieb}{\text{PrPl-EB}\xspace}
\newcommand{\mpeb}{\text{MP-EB}\xspace}
\newcommand{\klinfmulti}{\mathrm{KL}_{\inf}}
\newcommand{\numerator}{\text{numerator}\xspace}
\newcommand{\denominator}{\text{denominator}\xspace}

\newcommand{\rhoM}{\lceil \rho M \rceil}

\title{On the near-optimality of betting confidence sets \\ for  bounded means}
\usepackage{authblk}
\author[1]{Shubhanshu Shekhar}
\author[1, 2]{Aaditya Ramdas} 
\affil[1]{Department of Statistics and Data Science, Carnegie Mellon University}
\affil[2]{Machine Learning Department, Carnegie Mellon University}
\affil[ ]{\texttt{\{shubhan2, aramdas\}@andrew.cmu.edu}}
\date{}
\begin{document}
\maketitle

\begin{abstract}
Constructing nonasymptotic confidence intervals~(CIs) for the mean of a univariate distribution from independent and identically distributed~(\iid) observations is a fundamental task in  statistics. For bounded observations, a classical nonparametric approach proceeds by inverting standard concentration bounds, such as Hoeffding's or Bernstein's inequalities. Recently, an alternative betting-based approach for defining CIs and their time-uniform variants called confidence sequences  (CSs), has been shown to be empirically superior to the classical methods. In this paper, we provide theoretical justification for this improved empirical performance of betting CIs and CSs. 
Our main contributions are as follows: \textbf{(i)} We first compare CIs using the values of their first-order asymptotic widths (scaled by $\sqrt{n}$), and show that  the  betting CI of~\citet{waudby2023estimating}  has a smaller limiting width than  existing empirical Bernstein~(EB)-CIs. \textbf{(ii)} Next, we establish two lower bounds  that characterize the minimum width achievable by any method for constructing CIs/CSs in terms of certain inverse information projections. \textbf{(iii)} Finally, we show that the betting CI and CS match the fundamental limits, modulo an additive logarithmic term and a multiplicative constant. Overall these results imply that  the betting CI~(and CS) admit stronger theoretical guarantees than the existing state-of-the-art EB-CI~(and CS); both in the asymptotic and finite-sample regimes. 
\end{abstract}

\section{Introduction}
\label{sec:introduction}
    This paper studies the fundamental limits of the width of nonasymptotic confidence intervals~(CIs) and confidence sequences (CSs) for the mean $\mu$ of a distribution $P^*$ on $\mc{X}=[0,1]$ from an \iid sample $X_1, X_2, \ldots, X_n$ drawn according to $P^*$.  We define these below, but note first that all the methods analyzed in this paper are fully nonparametric, and make no assumptions whatsoever about knowledge of any aspect of the distribution $P^*$ except that it is bounded on $[0,1]$. Throughout this paper, we use $P^*$ to denote the  distribution generating the observations $(X_t)_{t \geq 1}$, and use $P$ when referring to an arbitrary probability distribution.
    
    Formally, a level-$(1-\alpha)$ CI for $\mu$ is a $\sigma(X_1, \ldots, X_n)$-measurable subset ${C}_n$ of the  domain $\mc{X}$ that satisfies the coverage guarantee $\mathbb{P}\lp \mu \in {C}_n \rp \geq 1-\alpha$. CIs can be constructed for datasets of fixed  (and non-random) size, and cannot be used for tasks that involve processing streams of observations of possibly data-dependent random size. In such cases, a time-uniform variant of CIs, called confidence sequences~(CSs), are the appropriate tool for inference. Given $X_1, X_2, \ldots \simiid P^*$, a level-$(1-\alpha)$ CS for the mean $\mu$ is a collection of sets $\{C_t: t \geq 1\}$ such that $C_t$ is $\sigma(X_1, \ldots, X_t)$ measurable for all $t \geq 1$, and  these sets satisfy the following uniform coverage guarantee: $\mathbb{P}\lp \forall t \in \mathbb{N}: \mu \in C_t \rp \geq 1-\alpha$.

    The size of a CI,  assuming it satisfies the $(1-\alpha)$ coverage guarantee,  is a natural metric for evaluating its quality. Since we are concerned with observations on the unit interval, a good measure of the `size' is the \emph{width} of the CI, which is the length of the smallest interval containing the CI. Formally, the width of ${C}_n$ is 
    \begin{align}
        |{C}_n| = \inf \{ b-a: {C}_n \subset [a,b]\}. 
    \end{align}
    Note that the width of a CI is a possibly random quantity, and most sensible strategies for constructing CIs usually ensure that it converges to $0$ almost surely. In fact, for our setting of bounded univariate observations,  the width typically converges to $0$ at a $\mc{O}\lp 1/\sqrt{n} \rp$ rate. Another desirable property of CIs, in addition to the order optimal convergence rate, is \emph{variance adaptivity}. That is, we want the leading constant (in the width of the CI) to be proportional to the standard deviation of the distribution. This means, for example, that the width of the CI for observations drawn from a $\texttt{Bernoulli}(0.99)$ distribution will be significantly tighter than that for a $\texttt{Bernoulli}(0.5)$ distribution (for the same value of $n$) for variance adaptive CIs. 
    
    A standard approach for constructing non-asymptotic CIs proceeds by inverting finite-sample concentration inequalities. For bounded random variables, the earliest such result was derived by~\citet{hoeffding1963probability}, who constructed a CI with the order-optimal width~(i.e., decaying at $1/\sqrt{n}$ rate). However, the resulting CI, referred to as Hoeffding's CI, is not variance adaptive, as it uses the worst case variance of $1/4$ for random variables supported on the unit interval. This was addressed by CIs based on the inequalities of \citet{bennett1962probability} and~\citet{bernstein1927theory}, but these methods require  knowledge of the variance (or at least a good upper bound), which limits their applicability. This has led to the construction of so-called empirical Bernstein~(EB) inequalities, that manage to replace the true (unknown) variance with their empirical estimates~\citep{audibert2009exploration, maurer2009empirical}. Additionally, time-uniform analogs~(i.e., confidence sequences) of all these CIs have been recently derived~\citep{howard2021time}.

    In a recent discussion paper, \citet{waudby2023estimating} developed a new approach for constructing CIs and CSs for the mean of bounded random variables (both with and without replacement, the latter being a setting we briefly return to at the end of this paper). They consider a family of hypothesis tests, each testing whether the true mean $\mu$ equals $m$, for all values of $m \in [0,1]$. The CI is then defined by inversion (as is standard): it is the set of the values of $m$, for which the corresponding null hypothesis has not yet been rejected. The key novelty is the type of test that is employed, both for the CI and the CS, which involves ``testing by betting''. In short, one converts the hypothesis test into a game, such that if the null is true, no bettor can make money in that game, but if the null is false, a smart bettor can grow their wealth exponentially. We introduce the details of this strategy in~\Cref{subsec:betting-approach}. Using this strategy, the authors constructed two new CIs and CSs: a closed-form ``predictable plug-in'' empirical Bernstein CI/CS~(that they refer to as \prpieb CI/CS), as well as a non-closed-form (but easy to numerically solve) betting CIs/CSs. 
    
    In a thorough empirical evaluation spanning a variety of competing methods derived over the last 60 years, these new CIs and CSs were shown to \emph{significantly} outperform all existing methods. However, while some intuition was provided by the authors, the theoretical results of~\citet{waudby2023estimating} do not explain this large improvement. We were also not able to find any information-theoretic lower bounds on the best possible width achievable by any CI/CS. \emph{The main motivation of our current paper is to close the gap between theory and practice for this basic and well-studied problem}.

    \begin{figure}[t]
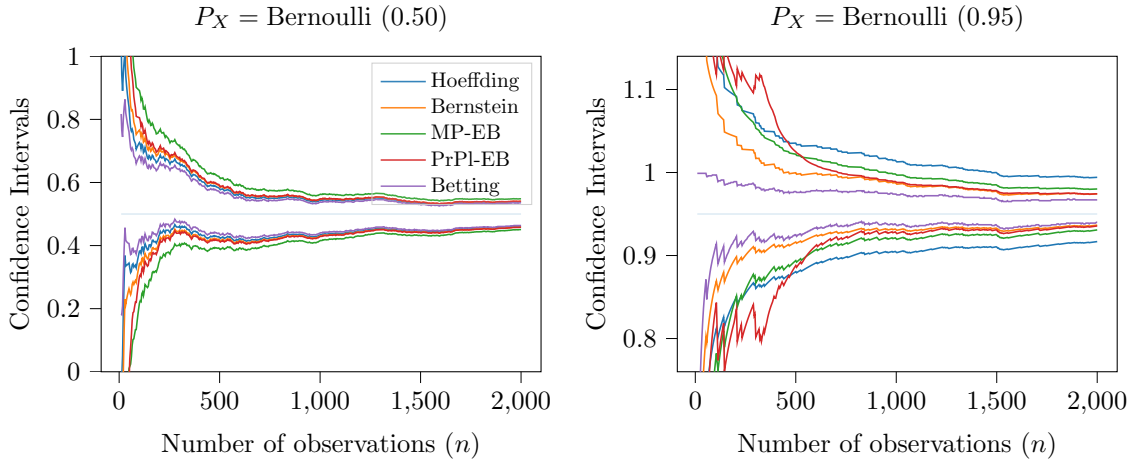

        \def\figwidth{0.45\columnwidth}
        \def\figheight{0.35\columnwidth} %
        \centering
        \hspace*{-1em}
        \input{Figures/Bernoulli_0_5_.tex}
        \input{Figures/Bernoulli_0_95_.tex}
        \caption{Comparison of the widths of different level-$(1-\alpha)$ CIs of the mean of \iid Bernoulli observations with means $0.50$~(left) and $0.95$~(right), with $\alpha=0.005$. In both cases,  the betting CI of~\citet{waudby2023estimating} dominates other methods, while the plot on the right further differentiates the variance-adaptive CIs from Hoeffding CI, that fails to exploit the low-variance observations. In the right plot, the CIs can be clipped to lie in $[0,1]$ --- we have not done that for higher visual clarity.  
        }
        \label{fig:comparison-of-CIs}
    \end{figure}

    \paragraph{Overview of results.} Our results provide theoretical justification of the benefits of the betting CI over competing methods. Despite the CIs being nonasymptotically valid, it turns out to be quite insightful to first compare their asymptotic behavior. To do this, we use the notion of first-order limiting widths of CIs, denoted by $\gamma_1$. In particular, for any level-$(1-\alpha)$ CI, whose width is a possibly random quantity $w_n$, we define its first-order limiting (half-)width as 
        \begin{align}
            \gamma_1 =  \limsup_{n \to \infty} \sqrt{n}\, w_n.  \label{eq:first-order}
        \end{align}
    For all the CIs we consider in this paper, $\gamma_1$ is  either a constant, or bounded above by a constant almost surely. We refer to $\gamma_1$ as the `first-order' limiting (half-)width, because it is solely determined by the dominant $\mc{O}(1/\sqrt{n})$  term in the width. This already leads to a separation amongst the CIs: they do not all have the same first order limiting width. For those that do,  we can also similarly analyze the `higher-order' limiting widths the $\mc{O}(1/n)$,  $\mc{O}(1/n^{3/2})$, etc.\ terms; this is not the goal of the paper but we return to this briefly in~\Cref{subsec:second-order}. Using $\gamma_1$, we can define a total  order over the space of CIs: 
    for any two level-$(1-\alpha)$ CIs, denoted by $\text{CI}^{(a)}$ and $\text{CI}^{(b)}$, we have  
    \begin{align}
        \text{CI}^{(a)} \leq_{\gamma} \text{CI}^{(b)} \quad \Longleftrightarrow \quad 
        \gamma_1^{(a)} \stackrel{a.s.}{\leq} \gamma_1^{(b)}.  \label{eq:CI-ordering}
    \end{align}
    In words, $\text{CI}^{(a)}$ is smaller (hence better) than $\text{CI}^{(b)}$, if the first-order limiting width $\gamma_1^{(a)}$ is smaller than $\gamma_1^{(b)}$. We use a strict inequality $<_\gamma$, and equality $=_{\gamma}$ in the natural way. In our first main result,~\Cref{prop:betting-CI-vs-EB-CI} in~\Cref{sec:limiting-width-betting},  we use this ordering to compare  betting CIs with some popular existing CIs (namely Hoeffding, Bernstein, \mpeb, and \prpieb CIs, formally defined in~\Cref{sec:preliminaries}). More specifically, we show that the betting CI, with a suitable choice of bets, is a strict improvement over Hoeffding and \mpeb CIs, and its limiting width is at least as good as that of the Bernstein CI, without the knowledge of the variance of the distribution~(needed to construct Bernstein CI).  To summarize, this result implies that 
    \begin{align}
        \text{betting CI } \leq_{\gamma} \prpieb \text{ CI}  =_{\gamma} \text{Bernstein CI} <_{\gamma} \mpeb \text{ CI} <_{\gamma}^* \text{ Hoeffding CI}. 
    \end{align}
    \sloppy The asterisk in the last inequality is to indicate that the strict inequality is valid when $\sigma < (1/2) \sqrt{\log(2/\alpha)/\log(4/\alpha)}$. For example, with $\alpha=0.05$, this condition reduces to $\sigma < 0.458$~(recall that $\sigma$ is upper bounded by $0.5$). 

    The previous result establishes the benefits of betting CI in an asymptotic sense, by showing that $\gammabet$ compares favorably with the limiting first-order widths of competing CIs. Our next set of results, presented in~\Cref{sec:nonasymp-CI}, demonstrate the near-optimality of the betting CI in a non-asymptotic regime. 
    
    We first establish a fundamental, method-agnostic, lower bound on the width achievable by any valid level-$(1-\alpha)$  CI in~\Cref{subsec:lower-bound-ci}. This lower bound is in terms of certain inverse information projection terms. Such nonasymptotic width (or, more generally, volume) lower bounds appear to be quite rare in the nonparametric \emph{estimation} (but not testing) literature, and we have not been able to find analogous references for other problem classes. 
    
    We then show that the betting CI width nearly matches the lower bound in~\Cref{subsec:upper-bound-ci}. To elaborate,  for  $P^* \in \mc{P}([0,1])$ and an $m \in [0,1]$, define the information projection $\klinf(P^*, m)$~(resp. $\klinfminus(P^*, m))$ as the smallest value of $\dkl(P^*, P)$  over all distributions $P \in \mc{P}([0,1])$ with mean at least $m$~(resp. at most $m$). Accordingly, the inverse information projection, $\klinf(P^*, \cdot)^{-1}(x)$, is the smallest $m$ such that $\klinf(P^*, m)$ is at least $x$~(the inverse of $\klinfminus$ is defined similarly). The information projection terms $\klinf$ and $\klinfminus$, formally introduced in~\Cref{def:klinf}, have been used in the multi-armed bandits literature to characterize the optimal instance-dependent cumulative regret~\citep{lai1985asymptotically, honda2010asymptotically}. Our results provide an analogous instance-dependent characterization of the width of CIs using their inverse. In particular, if $w_n$ denotes the width of any level-$(1-\alpha)$ CI, and $\widthbet$ denotes the width of the betting CI, then we have the following, for $a(n, \alpha) \defined \log\lp (1-\alpha)^{1-\alpha} \alpha^{2\alpha -1}\rp/n$: 
    \begin{align}
        &w_n \geq  \max \lbr \klinf(P^*, \cdot)^{-1}\lp a(n,\alpha) \rp - \mu, \; \mu -  \klinfminus(P^*, \cdot)^{-1}\lp a(n, \alpha)\rp \rbr,  \\
        \text{and} \quad 
        & \widthbet \leq 2\max \lbr \klinf(P^*, \cdot)^{-1}\lp b(n, \alpha) \rp - \mu, \; \mu - \klinfminus(P^*, \cdot)^{-1}\lp b(n, \alpha) \rp \rbr, \\
        \quad \text{with} \quad & a(n, \alpha) = \frac{ \log \lp (1-\alpha)^{1-\alpha} \alpha^{2\alpha -1} \rp}{n}, \quad \text{and} \quad b(n, \alpha) = \frac{\log(1/\alpha) + \mc{O}(\log n)}{n}. 
    \end{align}
    Since $a(n, \alpha) \approx \log(1/\alpha)/n$ as $\alpha \to 0$, this proves that the  betting CI  nearly matches the fundamental performance limit, but for two differences:  an extra leading multiplicative factor of $2$, and an additional $\mc{O}(\log(n)/n)$ term in the argument of the inverse information projection. The main takeaway form these results is that (i) the optimal instance-dependent width of any CI is precisely characterized by an instance-dependent `complexity parameter' based on the inverse of $\klinf$ and $\klinfminus$, and (ii) the width of the betting CI  also depends on nearly the same complexity term.  To the best of our knowledge, none of the existing CIs admit such upper bounds on their widths. 

    Having analyzed the performance of betting CI, we then establish similar guarantees for the betting CS in~\Cref{sec:nonasymp-CS}. Due to their time-uniform coverage guarantees, it is  not reasonable to order CSs by comparing their widths at some specific value of $n$ (for any typical CS, one can always find another CS that is tighter than it for some $n$, while being looser at some other $n$). To address this, we introduce a notion of \emph{effective width}, that is more suitable for comparing CSs (Definition~\ref{def:effective-width}). More specifically, we define the effective width of a CS after $n$ observations, as the smallest $w>0$ for which the expected number of observations needed for the width of the CS to go below $w$, is at most $n$. 
    Using this definition, we first establish a general lower bound on the effective width of any valid level-$(1-\alpha)$ CS in~\Cref{prop:lower-bound-2} in~\Cref{subsec:lower-bound-cs}. This lower bound also depends on an appropriate inverse information projection term, similar to the corresponding bound for CIs discussed above. We then obtain a nearly matching upper bound~(modulo the same constant multiplicative factor, and an additive logarithmic term in the argument of the inverse information projection) on the effective width of the betting CS in~\Cref{prop:betting-cs-width}.  
    
    Overall, our results provide rigorous justification for the empirically observed supremacy of the betting CI/CSs of~\citet{waudby2023estimating} over previously state-of-the-art methods, and also opens us several interesting directions for future work.

    \paragraph{Outline of the paper.} We present the background information about some popular existing CIs in~\Cref{subsec:existing-CIs}, and then formally describe the betting approach to constructing CIs in~\Cref{subsec:betting-approach}. The next three sections contain the details of the main technical results. In particular, in~\Cref{sec:limiting-width-betting} we show that the betting-CI of \citet{waudby2023estimating} is at least as good as the  best existing EB-CI in terms of the first-order limiting width defined in~\eqref{eq:first-order}. Next, in~\Cref{sec:nonasymp-CI}, we show that the betting CI has nearly-optimal width even in the finite sample setting, by establishing a method-agnostic lower bound, and then showing that the width of the betting CI is close to the lower bound. Finally, in~\Cref{sec:nonasymp-CS}, we show that the betting CS also satisfies a similar notion of near-optimality. We discuss some extensions of our results in~\Cref{sec:discussion}, and conclude the paper with some interesting questions for future work in~\Cref{sec:conclusion}.

\section{Preliminaries} 
\label{sec:preliminaries}
    We begin by describing some popular existing CIs in~\Cref{subsec:existing-CIs}, and then present the betting based approach for CI/CS constructing~\Cref{subsec:betting-approach}. Finally, we end with a discussion on a class of powerful betting strategies used for constructing CI/CS in~\Cref{subsec:betting-strategies}. 

    \subsection{CIs/CSs based on concentration inequalities}
    \label{subsec:existing-CIs}
        A standard approach for constructing CIs is by inverting high probability, non-asymptotic, concentration inequalities for the empirical mean based on \iid observations. Here, we recall three of the most popular CIs constructed using exponential concentration inequalities for bounded observations. 
        \paragraph{Hoeffding CI.} \citet{hoeffding1963probability} derived one of the earliest concentration inequalities for the mean of bounded random variables. The main idea is to use the general Cramer-Chernoff method, with an appropriate exponential upper bound on the moment generating function of bounded random variables. 
        On inverting the resulting concentration inequality, we get the following closed form CI: 
        \begin{align}
            \label{eq:hoeffding-ci}
            C_n^{(H)} = \lb \muhat_n \pm \frac{w_n^{(H)}}{2}\rb \defined \lb \muhat_n - \frac{w_n^{(H)}}{2}, \, \muhat_n + \frac{w_n^{(H)}}{2} \rb, \quad \text{with } \muhat_n = \frac{1}{n} \sum_{i=1}^n X_i, \quad \text{and } w_n^{(H)}= 2\sqrt{\frac{\log(2/\alpha)}{2n}}. 
        \end{align}
        Note that the width of Hoeffding CI converges to zero at an order-optimal $1/\sqrt{n}$ rate. Despite this order-optimality, it can be quite loose for random variables with low variance; in fact, the Hoeffding CI can be interpreted as assuming of a worst case variance of $1/4$.
        
        \paragraph{Bernstein CI.}  The aforementioned weakness can be addressed by constructing a CI based on a variance-aware concentration inequality, such as Bernstein exponential concentration inequality~\citep[Chapter 2]{boucheron2013concentration}. In particular, inverting this concentration inequality results in the following CI: 
        \begin{align}
            \label{eq:bernstein-ci}
            C_n^{(B)} = \lb \muhat_n \pm  \frac{w_n^{(B)}}{2} \rb, \quad \text{where } w_n^{(B)} = 2\sigma \sqrt{ \frac{ 2 \log(2/\alpha)}{n}} + \frac{4\log(2/\alpha)}{3n}. 
        \end{align}
        Since the maximum value of $\sigma$ of a random variable supported on $[0,1]$ is $1/2$,  the dominant $\mc{O}(1/\sqrt{n})$ term reduces to the Hoeffding CI width for this value of $\sigma$. However for problem instances with $\sigma \ll 1/2$, the above CI is a significant improvement over Hoeffding's. One drawback of Bernstein CI is that constructing it requires the additional information of the variance $\sigma$, or at least a good upper bound on it. In our setup, we do not assume that such additional side information is available.
        
        \paragraph{Maurer \& Pontil's Empirical Bernstein~(MP-EB) CI.}  \citet{maurer2009empirical} addressed the above issue, by deriving a fully data-driven empirical Bernstein CI: 
        \begin{align}
            \label{eq:emp-bern-ci-maurer}
            &C_n^{(\mpeb)} = \lb \muhat_n \pm  \frac{w_n^{(\mpeb)}}{2} \rb,\\ 
            \text{where } \quad  &w_n^{(\mpeb)} = 2\sigmahat_n\sqrt{\frac{2 \log(4/\alpha)}{n}} + \frac{14 \log(4/\alpha)}{3(n-1)}, \; \text{and} \; \sigmahat_n^2 = \frac{1}{n-1} \sum_{i=1}^n (X_i - \muhat_n)^2.  
        \end{align}
        The above CI satisfies many qualities of a good CI: it is variance-adaptive, has  order-optimal dependence on $n$, and it is fully data-driven, not requiring any prior information about the variance. However, it can still be significantly improved in theory and practice.  

        \paragraph{Predictable Plugin-EB (\prpieb) CI.} Recently, \citet{waudby2023estimating} obtained an improved empirical Bernstein CI using a martingale technique that avoids a union bound, thus achieving a better dependence on $\alpha$. 
        Define $\muhat_0:=0$, $V_0:=1/4$, and
            \begin{align}
                \psi_E(\lambda) \defined \frac{-\log(1-\lambda) - \lambda}{4}, \quad \muhat_{t} = \frac{1}{t}\sum_{i=1}^{t}X_i,\quad   tV_t =  \frac{1}{4} + \sum_{i=1}^t (X_i - \muhat_{i-1})^2, 
                \quad 
                \lambda_{t,n} = \sqrt{\frac{2 \log(2/\alpha)}{nV_{t-1}}}.
                \label{eq:eb-ci-lambda}
            \end{align}
            Then, the level-$(1-\alpha)$ \prpieb CI after $n$ observations is defined as  
            \begin{align}
                &C_n^{(\prpieb)}  = \lb \mutilde_n \pm \frac{w_n^{(\prpieb)}}{2}\rb, \label{eq:prpieb-def} \\  
                \text{where} \quad  & \mutilde_n \defined \frac{\sum_{t=1}^n \lambda_{t,n}X_i}{\sum_{t=1}^n \lambda_{t,n}}, \quad \text{and} \quad w_{n}^{(\prpieb)} = 2\frac{\log(2/\alpha) + 4\sum_{t=1}^n \psi_{E}(\lambda_{t,n})(X_i - \muhat_{i-1})^2}{\sum_{t=1}^n \lambda_{t,n}}.  \label{eq:prpieb-width}
            \end{align}
        The first-order limiting width, introduced in~\eqref{eq:first-order}, for the three classical CIs introduced above~(Hoeffding, Bernstein, \mpeb) can be easily calculated: 
        \begin{align}
            &\gamma_1^{(H)} = 2\sqrt{\frac{\log(2/\alpha)}{2}}, \quad 
            \gamma_1^{(B)} = 2\sigma \sqrt{2\log(2/\alpha)}, \quad \text{and} \quad 
            \gamma_1^{(\mpeb)} \eqas 2\sigma\sqrt{2\log(4/\alpha)}. %
        \end{align}
        These values succinctly characterize the relative strengths of these three CIs: 
        \begin{itemize}
            \item Hoeffding~CI achieves the order-optimal $1/\sqrt{n}$ dependence, but is not variance adaptive. 
            \item Bernstein~CI obtains a strictly smaller limiting width than Hoeffding CI for all values of $\sigma$ (since $\sigma$ can be at most $1/2$), but requires the additional information about the variance $\sigma^2$.  
            \item \mpeb~CI relaxes the requirement of knowing the variance, at the price of a larger $\alpha$-dependent term, which results from the EB-CI spending some of the $\alpha$ budget on estimating the variance. 
        \end{itemize}
        \citet{waudby2023estimating} showed that the \prpieb~CI has a limiting width equal to~$\gamma_1^{(B)}$, thus closing the gap that exists between \mpeb and Bernstein CI.
        This improvement is also witnessed in practice, with \prpieb being significantly tighter than \mpeb in simulations across a wide range of examples.
        In~\Cref{prop:betting-CI-vs-EB-CI} of this paper, we prove that the limiting width of the betting CI is no larger than $\gamma_1^{(B)}$, suggesting the possibility that the betting CI might admit stronger performance guarantees than existing CIs. This motivates our subsequent investigation of the width of the betting CI~(and CS) in the non-asymptotic regime, in which we characterize the fundamental limits of performance of any CI/CS, and then show that the betting CI/CS nearly matches this limit.

        We end this section by noting that there also exist time-uniform analogs of the above confidence intervals~(i.e., CSs), and while we refer the reader to~\citet{howard2021time, waudby2023estimating} for the details, we will analyze the CS widths as well later in this paper. 

        \subsection{The betting approach for CI construction}
        \label{subsec:betting-approach}
            \citet{waudby2023estimating} proposed a general method for constructing CIs/CSs for the mean of bounded random variables, by inverting the following continuum of hypothesis testing problems: 
            \begin{align}
                H_{0,m}: \mu = m, \quad \text{versus} \quad H_{1,m}: \mu \neq m, \quad \text{for all } m \in [0,1]. 
            \end{align}
            Then the CI~(resp.\ CS) after $n$ observations is defined as those values of $m$ for which $H_{0,m}$ has not yet been rejected by the corresponding hypothesis test~(resp. sequential hypothesis test). For simplicity, we discuss the construction of betting CIs here, and present the details of the betting CS at the beginning of~\Cref{sec:nonasymp-CS}. 
            
            To design each  test required for the betting CI, \citet{waudby2023estimating} build upon the principle of \emph{testing-by-betting}, recently popularized by \citet{shafer2021testing}. This principle states that the evidence against~$H_{0,m}$ can be precisely quantified by the wealth of a fictitious gambler who sequentially bets against the null in a fair (under the null) repeated game of chance, starting with an initial wealth of $1$ and never betting more money than they have. Denoting the wealth after $n$ rounds as $W_n(m)$ (with $W_0(m)=1$), the the fairness of the game under the null implies that $\mathbb{E}[W_n(\mu)] \leq 1$. 
            Markov's inequality then implies that one can reject $H_{0,m}$ at level $\alpha$ if $W_n(m)$ exceeds $1/\alpha$. Inverting these tests defines the CI at time $n$ as 
            \begin{align}
             C_n = \{m \in [0,1]: W_n(m) < 1/\alpha\}, \quad \text{with} \quad W_n(m) \defined \prod_{i=1}^n f_i(X_i, \lambda_i(m), m),  \label{eq:general-betting-CI}
            \end{align}
            \sloppy where $f_i$ is a predictable~(i.e., $\mc{F}_{i-1}$-measurable), nonnegative `payoff' function satisfying $\mathbb{E}[f_i(X, \lambda, \mu) \mid \mc{F}_{i-1}] \leq 1$ for all feasible values of  of $\lambda$, and $\lambda_i(m)$ is the  predictable `bet', denoting the fraction of the wealth, $W_{i-1}(m)$, that is put at stake by the bettor in round $i$. 
            This approach thus  transforms the task of constructing a CI/CS into that of designing wealth processes~(or equivalently, payoff functions and betting strategies) that grow quickly for $H_{0,m}$ with $m \neq \mu$. 

            \begin{remark}
                \label{remark:prpieb-ci-as-betting}
                Note that the Hoeffding and Bernstein CIs, as well as \prpieb, can be written as special cases of the general betting framework introduced above. For example, Hoeffding CI can be defined as the intersection of two level-$(1-\alpha/2)$ betting CIs, $C_n^0$ and $C_n^1$, that use the same bets ($i.e., \lambda_i(m)$, for all $i, m$), but different payoff functions $f_i^0$ and $f_i^1$ for $i \in [n]$, defined as 
                \begin{align}
                    f_i^j(x, \lambda, m) = \exp ( (-1)^j\lambda(x-m) - \lambda^2/8), \; \text{for } j \in \{0, 1\}, \quad \text{and} \quad \lambda_i(m) = \sqrt{\frac{8 \log(2/\alpha)}{n}}, 
                \end{align}
                for all $t \in [n]$ and $m \in [0,1]$. 
                Similarly, we can recover \prpieb CI using 
                \begin{align}
                    f_i^j(x, \lambda, m) = \exp \lp (-1)^j \lambda (x-m) - \psi_E(\lambda)(x- \muhat_{i-1})^2\rp, \quad \text{and} \quad \lambda_i(m) = \lambda_{i,n} \defined \sqrt{\frac{2 \log(2/\alpha)}{n V_{i-1}}},  
                \end{align}
                for all $m \in [0,1]$. 
                The bets $\{\lambda_i(m): i \in [n], \; m \in [0,1]\}$ in both examples are optimized for a fixed value of $n$, and hence we refer to them as CIs. We briefly discuss the constructing of betting CSs in~\Cref{def:betting-CS-def}, at the beginning of~\Cref{sec:nonasymp-CS}.  
            \end{remark}

            While the \prpieb CI has the benefit of a closed form expression, it does not utilize the full power of the general framework, as it can be interpreted as playing a sub-fair betting game at the (unknown) true mean $\mu$. Mathematically, this means that the wealth process $\{W_n(\mu): n \geq 1\}$ is a nonnegative super-martingale. \citet{waudby2023estimating} also developed an alternative, and more powerful, approach that relies on nonnegative martingales, that results in a less conservative CI in practice. However, the drawback of the resulting CI that we refer to as betting CI is that the intervals do not have a closed-form expression, and have to be obtained numerically. We recall this construction next. 
            \begin{definition}[Betting CI]
            \label{def:betting-CI-def}
                Given $X_1, X_2, \ldots, X_n$ drawn \iid from a distribution $P^*$ supported on $[0,1]$ with mean $\mu$, consider an instance of the general CI of~\eqref{eq:general-betting-CI}, with 
                \begin{align}
                    f_i(x, \lambda, m) = 1 + \lambda(x-m), \quad 
                    \text{and} \quad \lambda_{t}(m) \in \lb \frac{-1}{1-m}, \frac{1}{m} \rb, \quad \text{for all} \;m \in [0,1]. 
                \end{align}
                Then, we define the betting CI after $n$ observations as $C_n^{(bet)} = \{m \in [0, 1]: W_n(m) < 1/\alpha\}$. 
                Recall that for all $m \in [0,1]$ and $t \in [n]$,  the bet $\lambda_t(m)$ is $\mc{F}_{t-1}$ measurable, but can use knowledge of the horizon $n$. 
            \end{definition}
            The key component influencing the practical and theoretical performance of the betting CI is the strategy used for choosing the bets or the betting fractions,  $\{\lambda_t: t\geq 1, \; m \in[0,1]\}$.  \citet{waudby2023estimating} proposed and empirically compared various betting strategies that aim to optimize, either exactly or approximately, the growth rate of the wealth process for $m \neq \mu$. We formally describe this optimality criterion, and  recall a powerful betting strategy next. 
            
    \subsection{Log-optimal betting and regret}
    \label{subsec:betting-strategies}
        The width of the CI (or, as we return to later, CS) constructed by the betting method depends strongly on the choice of the betting strategy. Good betting strategies ensure that the wealth for $m \neq \mu$ grows rapidly, which in turn leads to CIs or CSs whose width shrinks fast. 
        We begin by introducing the $\log$-optimal betting strategy~\citep{kelly1956new} that selects $\lambda_t(m) = \lambda^*(m)$ for all $t \geq 1$,  and $m \in [0,1]$, where $\lambda^*(m)$ is defined as 
        \begin{align}
            \lambda^*(m) \in \argmax_{\lambda \in \lb - \frac{1}{1-m}, \frac{1}{m}\rb } \; \mathbb{E}_{X \sim P^*} \lb \log \lp 1 + \lambda(X-m) \rp \rb.  \label{eq:log-optimal}
        \end{align}
        The above constraints on $\lambda$ ensure that the quantity inside the logarithm is nonnegative.
        Clearly this strategy cannot be employed in practice, as the distribution $P^*$ is unknown. Instead, practical betting CIs/CSs proceed by choosing a sequence of predictable bets, $\{\lambda_t(m): t \geq 1, m \in [0,1]\}$, which incrementally approximate the $\log$-optimal bets in a data-driven manner. The quality of a betting strategy can be measured by its  regret; that is, the cumulative suboptimality with respect to the optimal betting strategy:
        \begin{align}
            \mc{R}_n(m) =  \sum_{t=1}^n \log\lp 1 + \lambda^*(m)(X_t-m) \rp - \sum_{t=1}^n \log\lp 1 + \lambda_t(m)(X_t-m) \rp; \quad m \in [0,1].  \label{def:regret-1}
        \end{align}
        While the term $\mc{R}_n(m)$ is a random quantity (as a function of $X_1^n$), there exist strategies~(called \emph{no-regret} strategies) that can guarantee uniformly decaying average regret over all possible realizations of $X_1^n$. 
        Formally, we say a betting strategy is \emph{no-regret} if for every $n \geq 1$ and $m \in [0,1]$, there exists a deterministic sequence, $\{r_{n}(m): n \geq 1\}$ with $r_n(m)/n \to 0$ for all $m \in [0,1]$, satisfying
            \begin{align}
                \mc{R}_n(m) \leq r_n(m) \quad \text{for all } n \in \mathbb{N}. 
            \end{align}
        Many strategies, such as the mixture method or the online Newton step strategy~\citep{hazan2007logarithmic, cutkosky2018black}, admit a logarithmic regret bound for all $n$ and $m$. That is, for these strategies, we have $r_n(m) = c_m \log(n)$ for some constant $c_m$ depending on $m$. For our nonasymptotic results in~\Cref{sec:nonasymp-CI} and~\Cref{sec:nonasymp-CS}, we will consider an instance of the general betting CS of~\citet{waudby2023estimating} in which the bets are chosen via the following version of the mixture method~\citep[\S~4.3]{hazan2016introduction}. 
        \begin{definition}[Mixture method] 
            \label{def:mixture-method}  
            This betting strategy sets $\lambda_1(m)=0$, for all $m \in [0,1]$, and sets $\lambda_t(m)$ to be the weighted average of all possible $\lambda$ values, with the weight assigned to a candidate value $\lambda$ being proportional to the hypothetical wealth~(denoted by $W_{t-1}^{\lambda}(m)$) if the bettor chose $\lambda_i(m) = \lambda$ for all $1 \leq i\leq t-1$. More formally, we have 
            \begin{align}
                \lambda_t(m) = \frac{\int_{-1/(1-m)}^{1/m} \lambda\, W_{t-1}^{\lambda}(m) d\lambda }{\int_{-1/(1-m)}^{1/m} W_{t-1}^{\lambda}(m) d\lambda}, \quad \text{with} \quad  
                W_{t-1}^{\lambda}(m) \defined \prod_{i=1}^{t-1} \lp 1 + \lambda(X_i-m) \rp. 
            \end{align}
            This betting scheme is known to have a  regret bound $\mc{R}_n(m) \leq \log(n) + 2$ for all $m \in [0,1]$, and this can be further upper bounded by $c \log(n)$, for a constant $c\approx 1.8$ for all values of $n \geq 13$.  
        \end{definition}
        Using (discrete or continuous) mixtures to design appropriate (super-)martingales is a standard technique in sequential inference~\citep{robbins1970statistical, howard2021time, waudby2023estimating}. Indeed, in the present context of betting CSs, \citet{waudby2023estimating} proposed a parallelizable discrete mixture strategy, while \citet{orabona2021tight} developed a method of employing continuous mixtures for CS construction by leveraging their regret bounds. More specifically, \citet{orabona2021tight} proposed the following CS using the regret~(denoted by $r_n$) of  Cover's universal portfolio~\citep{cover1991universal, cover1996universal} scheme:
        \begin{align}
            &C_n = \left\{m \in [0, 1]: \log \widetilde{W}_n(m)< \log(1/\alpha) + r_n \right\}, \label{eq:orabona-jun} \\
            \text{where} \quad &\log \widetilde{W}_n(m) \defined \sup_{\lambda \in \lb \frac{-1}{1-m}, \frac{1}{m} \rb} \sum_{t=1}^n \log( 1+ \lambda(X_t - m)).  
        \end{align}
        Our nonasymptotic results about betting CIs and CSs stated in~\Cref{prop:betting-ci-width} and~\Cref{prop:betting-cs-width} respectively are also applicable to the above CS with some minor modifications.

\section{Limiting width of the betting CI} 
\label{sec:limiting-width-betting}
    We now present the first main result of this paper, which says that for the same choice of bets, the limiting width of the betting CI is upper bounded by that of the \prpieb CI. The  limiting width of \prpieb CI was shown by~\citet{waudby2023estimating} to be equal to that of the Bernstein CI~(we recall this in~\Cref{fact:eb-ci}). Thus our result shows that the betting CI is also at least as good as as the Bernstein CI in terms of the order induced by the first-order limiting width. 
    \begin{fact}
        \label{fact:eb-ci} 
        The  \prpieb CI of~\eqref{eq:prpieb-def} is a valid level-$(1-\alpha)$ CI for the mean $\mu$ of $P^*$, and it has the following first-order limiting width: 
        \begin{align}
            \gamma_1^{(\prpieb)} \eqas 2\sigma \sqrt{2 \log(2/\alpha)}. 
        \end{align}
    \end{fact}

    \citet{waudby2023estimating} proved that for certain betting strategies, the width of the betting CI converges to $0$ at the order-optimal $1/\sqrt{n}$ rate. However, they did not establish whether or not the width of the betting CI is variance adaptive, and how it compares with their \prpieb CI. Nevertheless, through extensive empirical evaluation, they showed that the betting CI is variance-adaptive, and is often much tighter than all the competitors in practice. 
    Our next result takes the first step to address this theory-practice gap, by showing that the limiting  width of the betting CI is upper bounded by that of the \prpieb CI~(and hence, also the usual Bernstein CI). This establishes that the width of the betting CI is variance-adaptive, and furthermore, it is also asymptotically tighter than (or at least, no worse than) the improved \prpieb CI. This result also  motivates our further exploration of the non-asymptotic behavior of betting CIs/CSs  in the subsequent sections of this paper. 
    
    The betting CI we analyze in this section is constructed  with the same bets used to define the \prpieb CI. More specifically, we have the following, with $(\lambda_{t,n})$ as defined in~\eqref{eq:eb-ci-lambda}: 
        \begin{align}
            &\CIbet = \{m \in [0,1]: W_n(m) < 1/\alpha\}, \\ \text{where} \quad  \label{eq:betting-ci-with-prpl-bets}
            &W_n(m) =  \frac{1}{2} \lp W_n^+(m) + W_n^-(m) \rp, \quad 
            \text{and} \quad 
             W_n^{\pm}(m) = \prod_{t=1}^n \lp 1 \pm \lambda_{t,n}(X_t - m) \rp. 
       \end{align}
    We now state the main result of this section. 
    \begin{theorem}
        \label{prop:betting-CI-vs-EB-CI}
         Suppose the width  of the betting CI defined in~\eqref{eq:betting-ci-with-prpl-bets} after $n$ observations is denoted by $|\CIbet| = w_n^{(bet)}$. Then, we have 
        \begin{align}
            \limsup_{n \to \infty} \sqrt{n} \times w_n^{(bet)} 
            \;\stackrel{a.s.}{\leq}\;   \lim_{n \to \infty} \sqrt{n} \times w_n^{(\prpieb)}  \;\eqas\; 2\sigma \sqrt{2 \log(2/\alpha)}. 
        \end{align}
        In other words, with the same choice of bets as used by the \prpieb~CI defined in~\eqref{eq:prpieb-def}, the limiting width achieved by the betting CI~(\Cref{def:betting-CI-def}) is always no worse than the \prpieb~CI. 
    \end{theorem}
    \noindent \emph{Proof outline of~\Cref{prop:betting-CI-vs-EB-CI}.}  
    The first step in proving this result is to show that $C_n^{(\mathrm{bet})}$,  the betting CI, is always contained inside a larger CI denoted by $C_n^{\dagger}$, and furthermore, this larger CI has a finite limiting width $c_1<\infty$. Using the inclusion $C_n^{(\mathrm{bet})} \subset C_n^{\dagger}$, we can rewrite the betting CI as 
    \begin{align}
        C_n^{(\mathrm{bet})} = \{m \in C_n^{\dagger}: W_n(m) < 1/\alpha \}. 
    \end{align}
    This simple modification, replacing ``$m \in [0,1]$'' to ``$m \in C_n^{\dagger}$'' plays an important role in our analysis. This is because for  large values of $n$, we know that $|C_n^{\dagger}| \leq 2c_1/\sqrt{n}$, which allows us to restrict our attention to values of  $m$ in a narrow band, instead of the whole unit interval. We exploit this to obtain the required approximation and concentration results, to bound width of betting CI with that of \prpieb CI. In particular, we show that $w_n^{(bet)} \leq w_n^{(\prpieb)} + o(1/\sqrt{n})$, which leads to the required conclusion about the limiting widths. The details of the argument are presented in~\Cref{proof:betting-CI-vs-EB-CI}. 
    \hfill \qedsymbol
    
    \Cref{prop:betting-CI-vs-EB-CI} characterizes the performance of a betting CI  in the asymptotic regime. In the next section, we analyze its behavior in the non-asymptotic regime.

\section{Non-asymptotic analysis of betting CI}
\label{sec:nonasymp-CI}
    To benchmark the performance of the betting CI in the non-asymptotic regime, we first   establish a fundamental, method-agnostic, lower bound on the width of any CI, in terms of certain inverse information projections~(\Cref{def:klinf}) in~\Cref{prop:lower-bound-1}.
    This result gives us a precise, problem-dependent measure of complexity~(or hardness) of estimating the mean $\mu$ of a distribution $P^*$. 
    Then, in~\Cref{prop:betting-ci-width}, we show that the width of the betting CI nearly matches that lower bound, modulo a leading multiplicative constant and a logarithmic term in the argument of the information projection. 

    \subsection{Lower bound on any CI width}
    \label{subsec:lower-bound-ci}
        Let $\mc{C}$ denote a method of constructing CIs for the mean of a univariate distribution; that is, given observations $X_1, \ldots, X_n$ drawn $\iid$ from a distribution $P^* \in \mc{P}_0 \subset \mc{P}(\mathbb{R})$, and a confidence level $\alpha \in (0,1)$, the method $\mc{C}$ gives us a CI, denoted by $C_n = \mc{C}(X^n, \alpha)$, which satisfies $\mathbb{P}\lp \mu \in C_n \rp \geq 1-\alpha$. Here $\mc{P}_0$ denotes some pre-specified class of probability distributions, such as all distributions supported on $[0,1]$. Furthermore, assume that the method $\mc{C}$ returns CIs whose width can be upper bounded by a deterministic quantity~(as a function of $n$, $\alpha$, and the distribution $P^*$). Formally, assume there exists a function $w(n,P^*,\alpha)$ such that 
        \begin{align}
            |\mc{C}(X^n, \alpha)| = |C_n| \defined \inf \{b-a: C_n \subset [a, b]\} \stackrel{a.s.}{\leq}  w(n, P^*, \alpha). 
        \end{align}
        For some CIs, such as Hoeffding and Bernstein CIs, the actual widths are deterministic 
        \begin{align}
            w_H(n, P^*, \alpha) = 2\sqrt{ \frac{\log(2/\alpha)}{2 n}}, \quad \text{and} \quad 
            w_B(n, P^*, \alpha) = 2\sigma \sqrt{\frac{2 \log(2/\alpha)}{n}} + \frac{4 \log(1/\alpha)}{3n}, 
        \end{align}
        for distributions $P^* \in \mc{P}_0 =  \mc{P}\lp [0,1]\rp$. 

        Our main result of this section characterizes the achievable limit~(i.e., minimum width) of CIs in terms of the minimum KL divergence~(or relative entropy) between the distribution $P^*$, and a class of distributions whose means are either larger or smaller than some level $m$. We refer to these terms as information projections, and recall their formal definition below. 
        \begin{definition}[Information projection]
            \label{def:klinf} Let $P^*$ denote any distribution with mean $\mu$, lying in a class of distributions, $\mc{P}_0$, supported on $\mathbb{R}$. For any $m \in \mathbb{R}$, let  $\mc{P}_{0,m}^+$~(resp. $\mc{P}_{0,m}^-$) denote a subset of $\mc{P}_0$, containing distributions whose means are at least~(resp. at most) $m \in \reals$; that is,  $\mc{P}^+_{0,m} = \{Q \in \mc{P}_0: \mu_Q \geq m\}$, and $\mc{P}_{0,m}^- = \{Q \in \mc{P}_0: \mu_Q \leq m\}$. Using these, we define the following two information projection terms: 
            \begin{align}
                \klinf(P^*, m, \mc{P}_0) \defined \inf \{ \dkl(P^*, Q): Q \in \mc{P}^+_{0,m} \}, \quad \text{and} \quad 
                \klinfminus(P^*, m, \mc{P}_0) \defined \inf \{ \dkl(P^*, Q): Q \in \mc{P}^-_{0,m} \}. 
            \end{align}
            For any $x \geq 0$, we define the inverse information projections as 
            \begin{align}
                &\klinf(P^*, \cdot, \mc{P}_0)^{-1}(x) \defined \inf \lbr m \in \reals: \klinf(P^*, m, \mc{P}_0) \geq x \rbr,  \label{eq:inverse-klinf} \\
                \text{and} \quad 
                &\klinfminus(P^*, \cdot, \mc{P}_0)^{-1}(x) \defined \sup \lbr m \in \reals: \klinfminus(P^*, m, \mc{P}_0) \geq x \rbr. \label{eq:inverse-klinf-minus}
            \end{align}
        \end{definition}
        \begin{remark}
            \label{remark:klinf-def}
            Unlike the rest of this paper, the results in this section~(as well as in~\Cref{subsec:lower-bound-cs}) are valid for arbitrary classes of distributions on the real line, and not necessarily restricted to distributions supported on $[0,1]$. When, dealing with arbitrary distribution classes, we will include $\mc{P}_0$ in the expression $\klinf(P, m, \mc{P}_0)$. However, for distributions supported on $[0,1]$, we will drop the $\mc{P}_0$ dependence, to simplify the notation.   
        \end{remark}
        
        We now show a lower bound on the width that can be achieved by any method~($\mc{C}$) of constructing level-$(1-\alpha)$ CIs for the mean  of real-valued random variables.
        \begin{proposition}
            \label{prop:lower-bound-1} Suppose $\mc{C}$ denotes any method for constructing confidence intervals whose  width satisfies $|\mc{C}(X^n, \alpha)|\leq w(n, P, \alpha)$ almost surely, for all $P$ belonging to some class of distributions $\mc{P}_0 \subset \mc{P}(\mathbb{R})$.  Introduce the term $a(n, \alpha) =  \log\lp (1-\alpha)^{1-\alpha}\alpha^{2\alpha-1} \rp / n$, and  define  
            \begin{align}
                w^*(n, P, \alpha)  =  \max \lbr \klinf(P, \cdot, \mc{P}_0)^{-1}\lp {a(n, \alpha)}\rp - \mu, \; \mu -  \klinfminus(P, \cdot, \mc{P}_0)^{-1}\lp {a(n, \alpha)}\rp \rbr.  
            \end{align}
            Recall that $\klinf(P, \cdot, \mc{P}_0)^{-1}(x)$ and $\klinfminus(P, \cdot, \mc{P}_0)^{-1}$ were introduced in~\Cref{def:klinf}. Then, the width of the CI must satisfy $w(n, P, \alpha) \geq w^*(n, P, \alpha)$. 
        \end{proposition}
        \begin{remark}
            \label{remark:lower-bound-tightness} 
            The key conclusion of this result is that best achievable width~(modulo a leading multiplicative constant) of any CI is characterized by the distribution dependent information projection terms. To the best of our knowledge, this is the first result providing an explicit characterization of the smallest achievable width of CI in terms of a distribution-dependent complexity term. As we see later in~\Cref{subsec:lower-bound-instantiations}, for some common distributions (namely, Bernoulli and Gaussian), the evolution of the above lower bound  with the sample size $n$ exactly matches that of the  widths of the optimal CIs (again, modulo a constant multiplicative term). 
        \end{remark}

        \begin{remark}
            \label{remark:random-width-ci}
            \Cref{prop:lower-bound-1}, as stated, is not directly  applicable to CIs whose width after $n$ observations is a random quantity, such the the \mpeb CI of~\citet{maurer2009empirical}, or the \prpieb CI, and betting CI of~\citet{waudby2023estimating}. However, this result can be used to obtain a lower bound on the width of the best deterministic envelope containing these CIs. Such an envelope can be constructed by using a part of the available $\alpha$ on replacing the random terms in the CIs, with their population counterparts, plus a deviation term. We obtain such a deterministic envelope for the betting CI in~\Cref{subsec:upper-bound-ci}. 
        \end{remark}

        \noindent \emph{Proof outline of~\Cref{prop:lower-bound-1}.} The proof of this statement relies on the duality between confidence intervals~(CIs) and (fixed sample-size) hypothesis tests. In other words, if we can construct tight level-$(1-\alpha)$ CIs around the mean of a distribution $P^*$, it means that we can also accurately test for the mean, and the detection boundary of such a test is determined by the width of the CI. In particular, we proceed in the following steps: 
        \begin{itemize}
            \item  Given $X_1, \ldots, X_n \simiid P^*$, we first consider a hypothesis testing problem with $H_0: P^* = P$ versus $H_1: P^* = Q$, where we assume that $\mu_Q$~(the mean of $Q$) is at least $\mu_P + w(n, P, \alpha)$~(here, we use $\mu_P$ and $\mu_Q$ to denote the means of $P$ and $Q$ respectively). For this problem, we define a test $\Psi$ that rejects the null if $\mu_Q \in \mc{C}(X^n, \alpha)$, and show that this test controls the type-I and type-II errors at level $\alpha$ for this hypothesis testing problem. 
        
            \item  Next, using the properties of the test $\Psi$, the chain rule of KL divergence, and the data processing inequality, we show that the KL divergence between $P$ and $Q$ must satisfy $\dkl(P, Q) \geq a(n,\alpha)$. Since the right side in this inequality is independent of $Q$, we conclude that $\klinf\big(P, \mu_P + w(n, P, \alpha) \big) \geq a(n, \alpha)\defined  \log\lp (1-\alpha)^{1-\alpha} \alpha^{2\alpha-1} \rp/n $, by taking an infimum over all possible distributions $Q$ in the set $\mc{P}_{0,\mu_P + w(n, P, \alpha)}^{+}$. Recall that $\mc{P}_{0,\mu_P + w(n, P, \alpha)}^{+}$ denotes the class of all distributions in $\mc{P}_0$, with mean at least $\mu_P + w(n, P, \alpha)$. 
        
            \item  Finally, we use the definition of $w^* \equiv w^*(n, P, \alpha)$, and the continuity of $\klinf(P, \cdot)$,  to observe that $\klinf\big(P, \mu_P + w(n, P, \alpha) \big) \geq \klinf\big(P, \mu_P + w^*(n, P, \alpha) \big)$. This implies one part of the required result due to the monotonicity of $\klinf(P, \cdot)$. 
        \end{itemize}
        Repeating the same argument, but with $Q$ such that $\mu_Q$ is smaller than $\mu_P-w(n, P, \alpha)$ gives us the other term depending on $\klinfminus$. The details of this proof are in~\Cref{proof:lower-bound-1}. \hfill \qedsymbol

    \subsection{Width achieved by betting CI}
    \label{subsec:upper-bound-ci}
        Having established an unimprovable limit on the width of any level-$(1-\alpha)$ CI in~\Cref{prop:lower-bound-1}, we now analyze how close the betting CI comes to achieving this optimal performance.  Our main result shows the width of the betting CI nearly matches the lower bound derived in \Cref{prop:lower-bound-1}, but with two minor differences: there is an extra leading multiplicative factor of $2$~(see~\Cref{remark:lower-bound-tightness}), and there is an extra $\mc{O}(\log n /n)$ term in the  argument of the inverse information projection terms. Despite these two sources of suboptimality, the main takeaway of this section is that the width of the betting CI has the `right' functional form, depending on the same inverse information projection based complexity measure that arise in the method-agnostic lower bound. To the best of our knowledge, none of the existing CIs are known to admit similar upper bounds on their widths.

        Before proceeding to the statement of the result, we need to introduce some additional notation.
        In particular, recall from~\citet{honda2010asymptotically} that for distributions supported on $[0,1]$, the two information projection terms~(\Cref{def:klinf}) have the following dual representation for $a \in \{+, -\}$, that makes their connection to the betting CI/CS explicit: 
        \begin{align}
            \text{KL}_{\inf}^a(P^*, m) = \sup_{\lambda \in S_a} \mathbb{E}[\log(1+\lambda(X-m))], \quad \text{where} \quad S_+ = \lb\frac{-1}{1-m}, 0\rb, \quad \text{and} \quad S_- = \lb0, \frac{1}{m}\rb. \label{eq:klinf-dual}
        \end{align}
        For any $m \in [0,1]$, we also introduce their corresponding optimal betting fractions 
        \begin{align}
             &\lambda^*_a(m) \in \argmax_{\lambda \in S_a} \; \mathbb{E}_{X \sim P^*} \lb \log \lp 1 + \lambda(X-m) \rp \rb,\quad  
             \text{for} \; a \in \{+, -\}. \label{eq:lambda-def-dual}
        \end{align}
        We can now state the result characterizing the width of the betting CI. 
        \begin{theorem}
            \label{prop:betting-ci-width}
            Consider a betting CI constructed using the mixture betting strategy of~\Cref{def:mixture-method}.  Then, the width of such a level-$(1-\alpha/3)$ CI, denoted by $w^{(bet)}(n, P^*, \alpha/3)$,  satisfies:
            \begin{align}
                 w^{(bet)}(n, P^*, \alpha/3) \leq    2\max \lbr \klinf(\Phat_n, \cdot)^{-1}\lp \log(3n^2/\alpha) \rp - \mu, \;  \mu - \klinfminus(\Phat_n, \cdot)^{-1}\lp \log(3n^2/\alpha) \rp \rbr, \label{eq:ci-random-upper-bound}
            \end{align}
            where $\Phat_n$ denotes the empirical distribution of $\{X_1, \ldots, X_n\}$. 
            Since the above upper bound on $w^{(bet)}(n, P^*, \alpha/3)$ is random, we next derive a deterministic bound on it, that holds with probability at least $1-\alpha$, for $n$ large enough~(precise conditions in~\Cref{appendix:n-large-enough}): 
            \begin{align}
                &w^{(bet)}(n, P^*, \alpha/3) \leq \max \lbr \klinf(P^*, \cdot)^{-1}\lp b(n, \alpha) \rp - \mu, \;  \mu - \klinfminus(P^*, \cdot)^{-1}\lp b(n, \alpha) \rp \rbr + \frac{1}{n^2}, \\ 
                &\text{where} \quad b(n,\alpha) = \frac{\log(3n^2/\alpha)}{n} + \frac{9 \log(3n^2/\alpha)}{n\sigma^2} = \mc{O}\lp \frac{ \log(n/\alpha)}{n} \rp.  \label{eq:ci-deterministic-upper-bound}
            \end{align}
        \end{theorem}

        \begin{remark}
            \label{remark:anytime-valid-bound} Since the regret guarantee of the mixture method of~\Cref{def:mixture-method} is valid for all $n$, the first part of~\Cref{prop:betting-ci-width}; that is, the bound in~\eqref{eq:ci-random-upper-bound}, is simultaneously true for all values of $n$. Thus,~\eqref{eq:ci-random-upper-bound} is in fact a random upper bound on the width of the betting CS constructed using the mixture method. Interestingly,  the optimal complexity term characterizing the (random) width of the CS, again turns out to depend on  information projections, now computed for the empirical probability distribution, $\Phat_n$, instead of the true distribution $P^*$.  
            The second part of~\Cref{prop:betting-ci-width}, however, is valid only for a fixed value of $n$.   
        \end{remark}
        \begin{remark}
            \label{remark:comparision-betting-ci}  As mentioned in~\Cref{remark:random-width-ci}, since the bound on the width of the betting CI obtained in~\eqref{eq:ci-random-upper-bound} is a random quantity, it is not subject to the fundamental limit obtained in~\Cref{prop:lower-bound-1}. To address this, we use one-third of the available $\alpha$ on constructing the CI, and the remaining two-thirds on obtaining a deterministic upper bound on the random width. The size of the resulting level-$(1-\alpha)$, deterministic CI  is then subject to the result of~\Cref{prop:lower-bound-1}, and it matches the dependence of the lower bound on the information projection term, with two differences. It has an extra leading multiplicative factor of $2$, and it has an additional $\mc{O}(\log n /n )$ additive term in the argument of $\klinf(P^*, \cdot)^{-1}$. This term arises due to the regret incurred by the betting strategy, as well as a union bound over a uniform grid consisting of $m = \mc{O}(n^2)$ points, that we used to obtain the deterministic upper bound on the inverse empirical information projection. 
        \end{remark}
        \begin{remark}
        \label{remark:comparison-betting-ci-2}
            It is easy to verify that the term $a(n, \alpha)$ used in the statement of the lower bound~(\Cref{prop:lower-bound-1}) is approximately equal to $\log(1/\alpha)/n$ in the limit of $\alpha \to 0$.  
            Thus, the main conclusion of~\Cref{prop:betting-ci-width} is that we can construct a deterministic envelope containing the betting CI, whose width has very nearly the same functional dependence on $n$, $P$, and $\alpha$, as the method-independent lower bound.  To the best of our knowledge, none of the other commonly used CIs~(such as Hoeffding, Bernstein, \mpeb, or \prpieb) are known to exhibit this near-optimal behavior. 
        \end{remark}
        \noindent \emph{Proof outline of~\Cref{prop:betting-ci-width}.} 
            We begin by appealing to the logarithmic regret of the betting strategy, to show that for any $m \neq \mu$, the wealth after $n$ steps, $W_n(m)$, grows at an exponential rate (with $n$). Furthermore, the growth rate of the wealth is approximately equal to the maximum of two empirical information projection terms. Inverting this lower bound on the wealth gives us the bound on the width stated in~\eqref{eq:ci-random-upper-bound}. 

            To obtain the deterministic upper bound on the width, we need a high probability bound for the concentration of the empirical information projection about its population value for (almost) all $m \in [0,1]$. Existing concentration inequalities, such as those derived by~\citet[\S~7]{honda2015non}, when applied directly are not sufficiently tight for our purposes. Instead, we take an alternative approach, by first obtaining a wider CI that contains $\CIbet$, and has a  width  of $2b_n(n, \alpha)$. By focusing on this narrow band, we show that the information projection can be well approximated by the first two terms of their Taylor's expansion. Hence by controlling the deviations of the first two empirical moments from their population terms, we get a sufficiently tight concentration result for the empirical information projection about its population value. The final step is to take a union bound for $m$ values over grid $\mc{M}_n$ of size $n^2+1$, consisting of equally spaced points inside the larger band of size $2b(n, \alpha)$. 
            This gives us a lower bound on the wealth, $W_n(m_i)$, for points $m_i \in \mc{M}_n$. We obtain the final bound~\eqref{eq:ci-deterministic-upper-bound}  by taking the inverse of the information projection (i.e., $\klinf(P^*, \cdot)^{-1}$) over the points in the grid $\mc{M}_n$, and then adding the grid spacing term, $1/n^2$, to account for the possible discretization error. 
        The details of these steps are in~\Cref{proof:betting-ci-width}. \hfill \qedsymbol

\section{Non-asymptotic analysis of betting CS}
\label{sec:nonasymp-CS}
    We now turn our attention to the betting confidence sequences~(CS) constructed by~\citet{waudby2023estimating}, and show that they can be also be shown to achieve a near-optimal performance, similar to the betting CI. The construction of betting CSs follows the same general steps introduced in~\Cref{subsec:betting-approach} as we recall below. 
    \begin{definition}[Betting CS]
        \label{def:betting-CS-def} Consider a stream of observations, $X_1, X_2, \ldots$, drawn \iid from a distribution $P^*$ supported on $[0,1]$ with mean $\mu$. For all $m \in [0,1]$, set $W_0(m)=1$, and define the wealth process as 
        \begin{align}
            W_n(m) = W_{n-1}(m) \times \lp 1 + \lambda_n(m)(X_n - m) \rp, \quad \text{for all } m \in [0, 1], 
        \end{align}
        where $\{\lambda_n(m): n \geq 1, \, m \in [0,1]\}$ is any sequence of predictable bets. Then, the level-$(1-\alpha)$ betting CS is a collection of sets $\{C_n \subset [0,1]: n \geq 1\}$, defined as 
        \begin{align}
            C_n = \{m \in [0,1]: W_n(m) < 1/\alpha \}, \label{eq:betting-cs-def}
        \end{align}
        satisfying the uniform coverage guarantee $\mathbb{P}\lp \forall n \geq 1: \mu \in C_n \rp \geq 1-\alpha$. As a consequence of the uniform coverage guarantee, we can also assume that the sets in the CS are nested; that is, $C_n \subset C_{n'}$ for all $n' \leq n$. In other words, for any $n \geq 1$, we can let $C_n$ be the running intersection of all $C_{n'}$ for $n'\leq n$ due to the time-uniform nature of CS.
    \end{definition}
    The main difference between the betting CI of~\Cref{def:betting-CI-def}, and the betting CS defined above, is that the bets $\{ \lambda_t(m): t \geq 1, \, m \in [0,1]\}$ are not optimized with a pre-specified horizon $n$. The uniform coverage guarantee of the betting CS is a simple consequence of Ville's inequality, a time-uniform variant of Markov's inequality, and we refer the reader to~\citet{waudby2023estimating} for further details.

    To analyze the behavior of betting CSs, with possibly random widths, we first introduce the notion of ``effective width'' in~\Cref{def:effective-width}. Then, in~\Cref{prop:lower-bound-2} we obtain a lower bound on the effective width of any CS, and show that the betting CS instantiated with the mixture method stated in~\Cref{def:mixture-method},  nearly matches this fundamental limit in~\Cref{prop:betting-cs-width}. 
    
    \subsection{Lower Bound on any CS effective width}
    \label{subsec:lower-bound-cs}
        In this section, we establish fundamental limits on the performance of any valid level-$(1-\alpha)$~CS, with possibly random widths. We begin by introducing a notion of \emph{effective width}, that we will use to characterize the performance of  CSs. 
        \begin{definition}[Effective width]
            \label{def:effective-width} 
             Let $\mc{C}$ denote a method for constructing confidence sequences. Given a stream of observations $X_1, X_2, \ldots \simiid P^*$,  define the random stopping times 
             \begin{align}
                 T_w \equiv T_w(P^*, \alpha) \defined \inf \lbr n \geq 1:  |\mc{C}(X^n, \alpha)| \leq w \rbr, 
                 \quad \text{for } w \in (0, \infty). 
             \end{align}
             Then, the effective width of the CS  after $n$ observations, is defined as:
             \begin{align}
                 w_e(n, P^*, \alpha) \defined \inf \lbr w > 0: \mathbb{E}[T_w] \leq n  \rbr. 
             \end{align}
            In words, the effective width of a CS after $n$ observations is the minimum $w>0$, for which the expected number of observations needed by the CS to reduce its width below $2w$, no larger  than $n$. 
        \end{definition}
        \begin{remark}
            While this stopping time based definition of effective width turns out to be particularly well suited to analyzing confidence sequences, we note that the time-uniform coverage guarantee is not necessary for this definition to be valid. In particular, the same definition can also be applied to CIs with random or even deterministic widths. 
            In fact,  the definition of effective width reduces to that of the usual width for the case of CIs with deterministic widths. To see why this is true, suppose $\mc{C}$ constructs CIs with deterministic widths $w(n, P^*, \alpha)$, that are  non-increasing~in $n$. Then, the term $T_w$ in this case is simply a constant. Hence, $w_e(n, P^*, \alpha)$ is equal to $\inf \{w \geq 0: T_w(P^*, \alpha) \leq n\}$. Now, note that for all $w \in [w(n, P^*, \alpha),  w(n-1, P^*, \alpha))$, we have $T_w = n$. This implies that 
            \begin{align}
                w_e(n, P^*, \alpha ) = \inf \{ w: w \in [w(n, P^*, \alpha), w(n-1, P^*, \alpha)) \} = w(n, P^*, \alpha), 
            \end{align}
            as claimed.
            For example, in the case of Hoeffding CI, we have $T_w = \left \lceil2 \log(2/\alpha)/w^2 \right \rceil$, for all $w>0$, which implies that
            \begin{align}
                w_e^{H}(n, P^*, \alpha) = \inf \lbr w \geq 0:  \left\lceil \frac{2\log(2/\alpha)}{w^2} \right\rceil \leq n \rbr = 2\sqrt{\frac{\log(2/\alpha)}{2n}} = w_H(n, P^*, \alpha). 
            \end{align}
        \end{remark}
    
        We now present the main result of this section, that  characterizes the minimum achievable effective width of a confidence sequence~(CS). 
        \begin{proposition}
            \label{prop:lower-bound-2}
            Suppose $\mc{C}$ is a method for constructing confidence sequences that satisfies $\mathbb{E}[T_w(P, \alpha)] < \infty$ for all $w>0$, and for all $P \in \mc{P}_0 \subset \mc{P}(\reals)$. With  $a(n, \alpha) \defined \log\lp (1-\alpha)^{1-\alpha} \alpha^{2\alpha-1} \rp/n $, introduce 
            \begin{align}
                w_e^*(n, P, \alpha)  = \max \lbr \klinf(P, \cdot, \mc{P}_0)^{-1}\lp {a(n, \alpha)} - \mu \rp, \; \mu -  \klinfminus(P, \cdot, \mc{P}_0)^{-1}\lp {a(n, \alpha)}\rp \rbr,  
           \end{align}
           where the inverse information projection terms, $\klinf(P, \cdot, \mc{P}_0)^{-1}(\cdot)$ and $\klinfminus(P, \cdot, \mc{P}_0)^{-1}(\cdot)$, were defined in~\eqref{eq:inverse-klinf}. 
           Then the effective width of the CS constructed by method $\mc{C}$ using \iid observations from $P \in \mc{P}_0$ must satisfy $w_e(n, P, \alpha) \geq w_e^*(n, P, \alpha)$. 
        \end{proposition}
        \emph{Proof outline of~\Cref{prop:lower-bound-2}.}
        The proof of this proposition uses a similar  high-level strategy outlined for proving~\Cref{prop:lower-bound-1}. The key difference is that we use now set up a sequential testing problem, for which we design a test using the given method $\mc{C}$  for constructing CSs.  We present the details in~\Cref{proof:lower-bound-2}.  \hfill \qedsymbol

    \subsection{Effective width of betting CS}
    \label{subsec:upper-bound-cs}
        In the previous subsection, we obtained a lower bound on the effective width of any method, $\mc{C}$, for constructing confidence sequences. We now show that the betting CS proposed by~\citet{waudby2023estimating} nearly matches this lower bound. We  first present a simplifying assumption about the data generating distribution $P^*$. 
        \begin{assumption}
        \label{assump:support}
            The distribution $P^*$ is supported over a  closed interval that is a strict subset of $[0,1]$. That is, $\text{supp}(P^*) = [A, B]$ with $A, B \in (0,1)$. 
        \end{assumption}       
        This assumption implies that for all values of $m$, the logarithmic increment of the wealth process, $\log(1+\lambda_t(m)(X_t-m))$, is a bounded random variable. While this assumption is not strictly necessary for obtaining the next result, it does simplify the final expression of the upper bound. 
        \begin{theorem}
            \label{prop:betting-cs-width}
            Suppose~\Cref{assump:support} is true, and consider a betting CS constructed using the mixture betting strategy~(\Cref{def:mixture-method}) that incurs a logarithmic regret. Then, the effective width of this CS satisfies: 
            \begin{align}
                &w_e^{(bet)}(n, P^*, \alpha) \leq 2\max \lbr \klinf(P^*, \cdot)^{-1}(b(n,\alpha)) - \mu, \; \mu - \klinfminus(P^*, \cdot)^{-1}(b(n,\alpha)) \rbr,\\
                 \text{where } &b(n,\alpha) = \frac{ 2\log(n^2/\alpha) + 2C }{n}, 
            \end{align}
            for a constant $C$ that depends on the distribution $P^*$~(\Cref{assump:support}).  
        \end{theorem}
        \begin{remark}
            As in the case of~\Cref{prop:betting-ci-width}, this result implies that the effective width is of the same order as the method-agnostic lower bound established in~\Cref{prop:lower-bound-2}, modulo a constant multiplicative factor and an additive $2\log(n^2)+2C$ term in the argument to the inverse information projection. 
            The term $C$ arises due to the possible overshoot of the wealth process beyond the threshold $1/\alpha$, and it is generally unavoidable, without further assumptions. 
            Interestingly, the upper bound has much a smaller logarithmic penalty term, as compared to the upper bound on the betting CI width of~\Cref{prop:betting-ci-width}. This is because of the expectation in the definition of effective width, which causes the population information projection~(instead of the empirical information projection) to naturally appear in the upper-bound expression,  and we don't need to apply a  concentration result over a uniform grid to get the required form matching the lower bound. 
        \end{remark}
        \noindent \emph{Proof outline of~\Cref{prop:betting-cs-width}.} By appealing to the regret bound of the mixture betting strategy, we can reduce the problem into that of analyzing the `oracle' wealth processes $\{W^{*,a}_t(m): t \geq 1, \; m \in [0,1]\}$ for $a \in \{+, -\}$, with $W^{*,a}_t(m) = \prod_{i=1}^t\lp 1 + \lambda^*_a(m)(X_i-m) \rp$. Due to the \iid nature of the multiplicative factors, this process is significantly more tractable to analyze than the true wealth process.
        In particular, we proceed as follows: 
        \begin{itemize}
            \item In the first step, we show that for any $w$, we can upper bound the stopping time  $T_w$ in terms of the maximum of two stopping times, $\tau_w^+$ and $\tau_w^-$, associated with the oracle wealth process. In particular, $\tau_w^+$~(resp. $\tau_w^-$) is the first time, the $\log$ of the oracle wealth process at $m=\mu+w/2$~(resp. $m=\mu-w/2$) exceeds $\log(n^2/\alpha)$.  This implies the bound $\mathbb{E}[T_w] \leq \max \{ \mathbb{E}[\tau_w^+], \, \mathbb{E}[\tau_w^-] \}$. 

            \item The key benefit of working with the oracle process is that $\log(W_n^{*, \pm}(m))$ are the sum of \iid terms, as the oracle bets defined in~\eqref{eq:lambda-def-dual} are  fixed, non-random quantities. This is unlike the actual wealth process with predictable  bets. Hence, by applying appropriate optional stopping arguments, we show that $\mathbb{E}[\tau_w^+]$~(resp. $\mathbb{E}[\tau_w^-]$)  can be upper bounded in terms of $\klinf(P^*, \mu+w)$~(resp. $\klinfminus(P^*, \mu-w)$). 

            \item The final step in concluding the proof involves inverting the two information projection terms that bound $\mathbb{E}[T_w]$, and some simplification to get the required bound. 
        \end{itemize}
        The details are in~\Cref{proof:betting-cs-width}. 
        \hfill \qedsymbol

\section{Discussion, extensions, and open questions}
\label{sec:discussion}
    In~\Cref{subsec:lower-bound-instantiations}, we instantiate our general lower bound for CIs~(\Cref{prop:lower-bound-1}) for some univariate distributions, and compare them with the optimal CI width for those distributions. Next, in~\Cref{subsec:multivariate}, we show that the techniques we developed for obtaining the lower bounds on the widths of CIs and CSs can be easily generalized to the case of multivariate observations. In~\Cref{subsec:without-replacement}, we extend our asymptotic analysis to the practically interesting case of estimating the mean of $M$ unknown values in the unit interval, by uniformly sampling them without replacement. We end  with some preliminary discussion about appropriate notions of second-order limiting widths of CIs  in~\Cref{subsec:second-order}.

    \subsection{Instantiations of the lower bound}
    \label{subsec:lower-bound-instantiations}
        Since the non-asymptotic lower bounds in~\Cref{prop:lower-bound-1} and~\Cref{prop:lower-bound-2} are stated in terms of the abstract information projections, they might be difficult to interpret. For many distributions, such as members of one-parameter exponential family, the lower bounds admit closed-from expressions. In this section, we study the tightness of the CI lower bound~(since the expressions for the CS lower bound are almost identical), by comparing its instantiations for two specific distributions with the width of an oracle CI constructed using the quantiles of the true distributions.

        \paragraph{Bernoulli distribution.} We begin with the simplest case, where the distribution $P^*$ is a Bernoulli with mean $\mu$. In this case, $\klinf$ and $\klinfminus$ can be explicitly calculated. In particular, both $\klinf(P^*, m)$ and $\klinfminus(P^*, m)$ are equal to  $\mu \log(\mu/m) + (1-\mu) \log( (1-\mu)/(1-m))$ for all $m \geq \mu$ and $m \leq \mu$ respectively. We can numerically invert these expressions to obtain a good approximation of the lower bounds derived in~\Cref{prop:lower-bound-1}. 
        
        To evaluate the tightness of the lower bound, we compare it to the width of an oracle CI constructed with the knowledge of the distribution of the sample mean of $X_1, \ldots, X_n$. In particular, $\sum_{i=1}^n X_i$ is distributed according to the Binomial distribution with parameters $n$ and $\mu$. Hence, we can define an oracle CI as $C_n^* = [\mu \pm w^*/2]$, with $w^* = (z_{1-\alpha/2} - z_{\alpha/2})/n$, where $z_{\beta}$ denotes the $\beta$-quantile of the Binomial distribution with parameters $n$ and $\mu$. We plot the variation of the lower bound, and the width of the oracle CI, with the sample size $n$ for $\mu = 0.9$ in~\Cref{fig:lower-bound}. 
    
         \paragraph{Gaussian distribution.} Similarly, for the class of Gaussian distributions with variance $\sigma^2$,  $\klinf(P^*, m)$ and $\klinfminus(P^*, m)$ $ = (\mu - m)^2/2\sigma^2$, if $m \geq \mu$. Hence, the result of~\Cref{prop:lower-bound-1} implies that the width of any CI for this class must satisfy $w(n, P^*, \alpha) \geq 2\sigma\sqrt{ a(n, \alpha)/2}$, where $a(n,\alpha)$ is as defined in~\Cref{prop:lower-bound-1}.  

         As with the Bernoulli case, we will compare the lower bound with the width of the oracle CI, defined as $C_n^* = [\mu \pm w^*/2]$ with  $w^* = \sigma(z_{1-\alpha/2} - z_{\alpha})/\sqrt{n}$, where $z_{\beta}$ denotes the $\beta$-quantile of the standard normal distribution. The variation of the predicted lower bound and the oracle upper bound with the sample size $n$, for $P^* = N(0, 4)$ is shown in~\Cref{fig:lower-bound}. 
        \begin{figure}
            \centering
            \includegraphics[width=0.45\columnwidth]{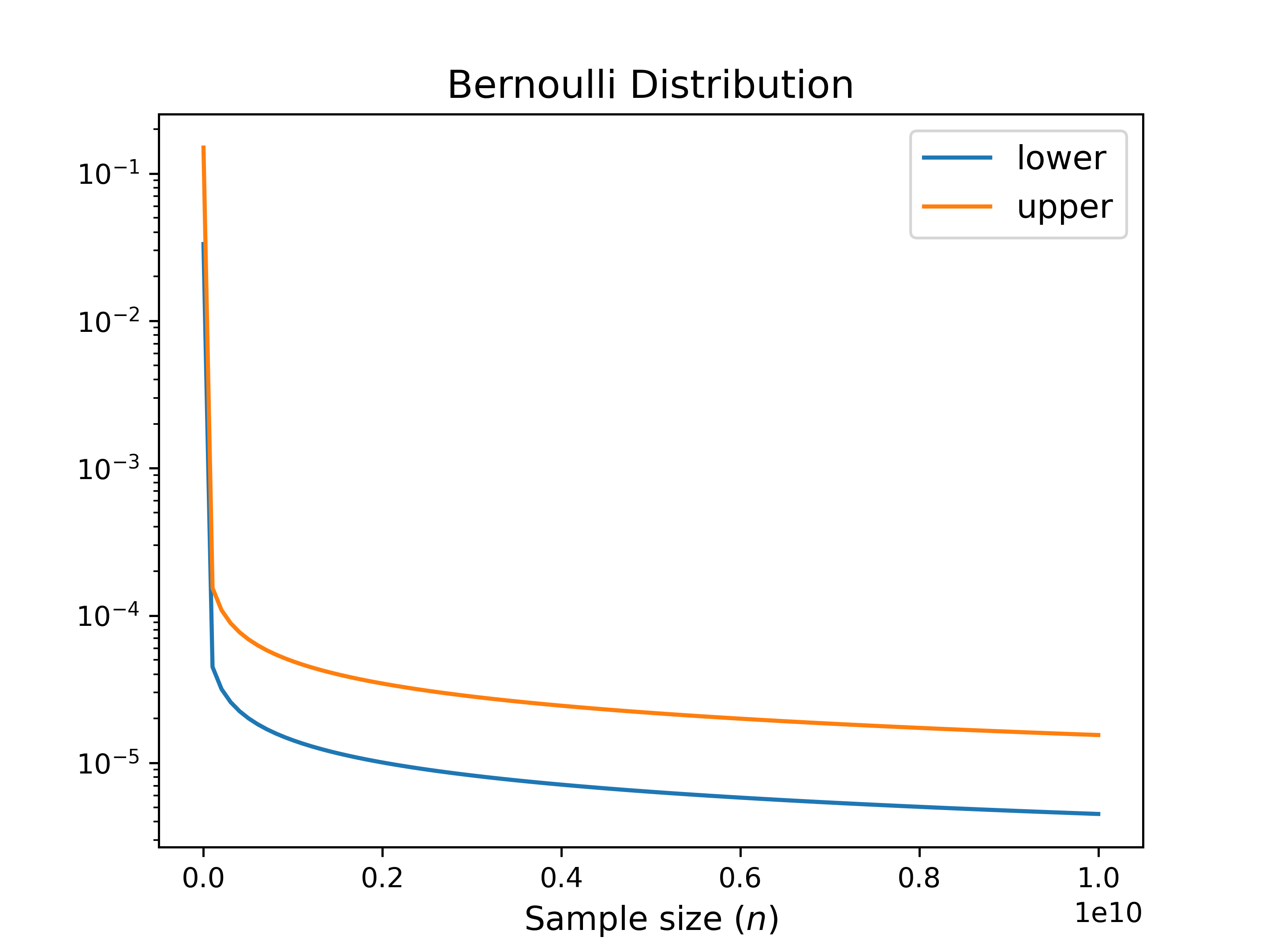}
            \includegraphics[width=0.45\columnwidth]{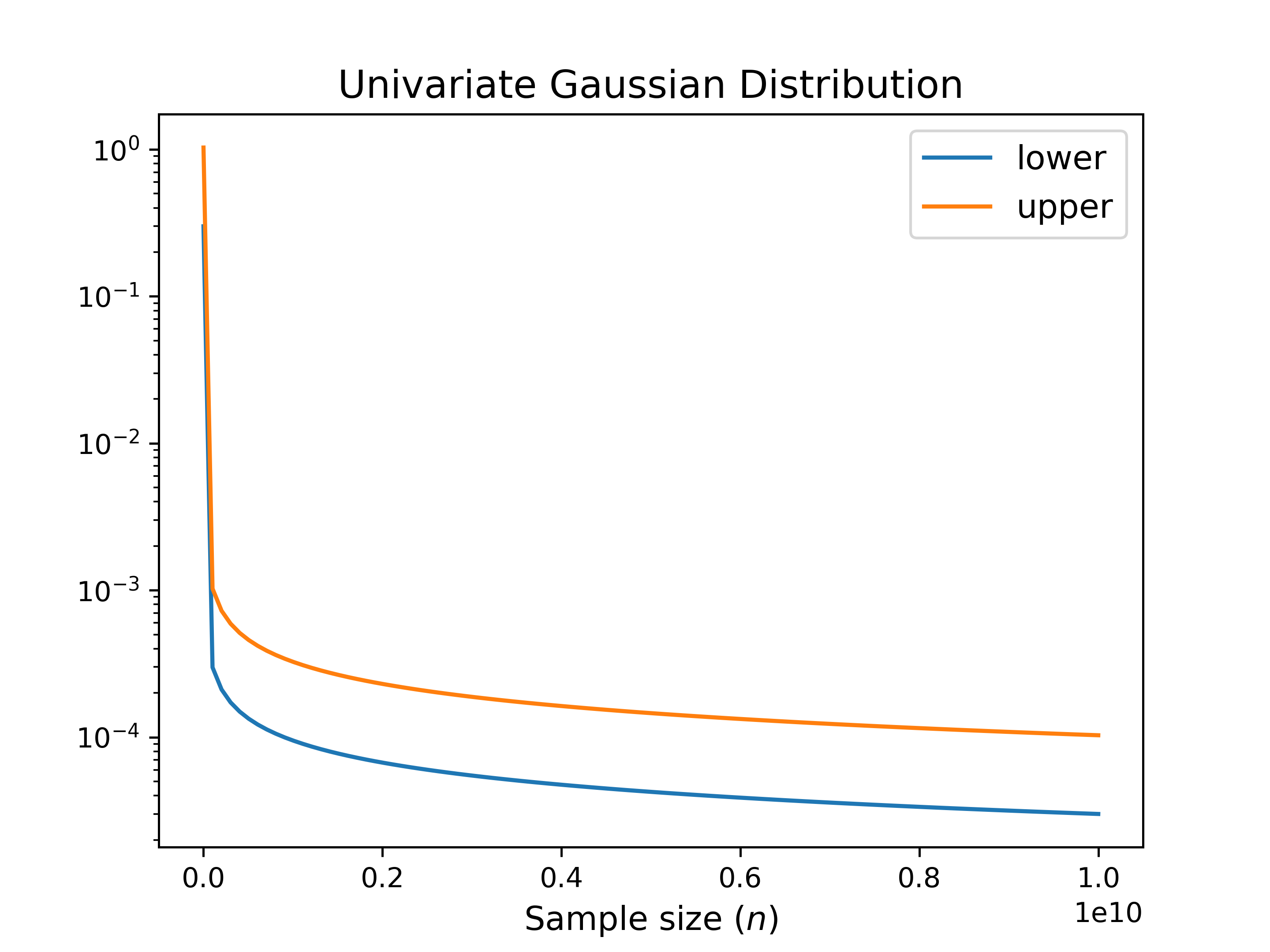}
            \caption{The figures plot the variation~(with the sample size $n$) of the predicted lower bound~(\Cref{prop:lower-bound-1}) on the width of any CI~(denoted by `lower'), and  the width of the oracle CI~(denoted by `upper')  for the mean of Bernoulli~(left) and Gaussian~(right) distributions~(with  $\alpha$  set to $0.01$). In both cases, the ratio of the upper and lower bounds is approximately $3.43$ for all values of $n \geq 1000$, indicating that the lower bound captures the `right' behavior of the optimal width with $n$, modulo a constant multiplicative factor.}
            \label{fig:lower-bound}
        \end{figure}

    \subsection{Extending lower bounds to the multivariate case}
    \label{subsec:multivariate}
        While our focus in this paper is on the case of univariate observations, we note that the techniques used to obtain the lower bounds~(\Cref{prop:lower-bound-1} and~\Cref{prop:lower-bound-2}) can be easily extended to more general observation spaces.
        To illustrate this, we state the generalizations of these two lower bounds for the case of observations lying in $\mathcal{X} = \mathbb{R}^d$, for some integer $d \geq 1$. In particular, suppose we are given observations $X_1, X_2, \ldots$, drawn \iid from a distribution $P^* \in \mc{P}_0 \subset \mc{P}(\mc{X})$, with mean vector $\mu$. 

        Our first result considers a method~($\mc{C}$) for constructing confidence intervals for $P^*$ based on observations $X_1, \ldots, X_n$, which satisfies the condition that 
        \begin{align}
            \sup_{x, x' \in \mc{C}(X_1^n, \alpha)} \; \|x-x'\|_2 \leq w(n, P^*, \alpha), \label{eq:mutltivariate-width-condition-1}
        \end{align}
        where $\|\cdot\|_2$ denotes the usual $\ell_2$-norm on $\reals^d$. 
        In other words, the confidence set constructed by the method $\mc{C}$ based on $n$ observations is contained in a ball of width no larger than  a deterministic value $w(n, P^*, \alpha)$.  To state the lower bound, we need to introduce a notion of  information projection for multivariate case: 
        \begin{align}
            &\klinfmulti(P, r, \mc{P}_0) = \inf_{Q \in \mc{P}_0: \|\mu_Q-\mu_P\|_2 \geq r} \dkl(P, Q),\label{eq:klinf-multi} \\
            \text{and} \quad  
            &\klinfmulti(P, \cdot, \mc{P}_0)^{-1}(x) = \inf \{r \geq 0: \klinfmulti(P, r, \mc{P}_0) \geq x \}. \label{eq:klinf-multi-inv}
        \end{align}
        \begin{proposition}
        \label{prop:multivariate-lower-bound-1}
             Suppose $\mc{C}$ denotes any strategy for constructing confidence sets for the mean vector $\mu$ satisfying~\eqref{eq:mutltivariate-width-condition-1} for all $P \in \mc{P}_0 \subset \mc{P}(\mc{X})$ for $\mc{X} = \mathbb{R}^d$.   Introduce $a(n, \alpha) = \frac{ \log\lp (1-\alpha)^{1-\alpha} \alpha^{2\alpha-1} \rp }{n}$, and 
            \begin{align}
                w^*(n, P, \alpha)  =  \klinfmulti(P, \cdot, \mc{P}_0)^{-1}\lp {a(n, \alpha)}\rp, 
            \end{align}
            where $\klinfmulti(P, \cdot, \mc{P}_0)^{-1}$ was defined in~\eqref{eq:klinf-multi-inv}.  Then, the width of the confidence set must satisfy $w(n, P, \alpha) \geq w^*(n, P, \alpha)$. 
        \end{proposition}
        The proof of this statement follows by adapting the arguments developed for proving~\Cref{prop:lower-bound-1}, and the details are in~\Cref{proof:multivariate-lower-bound-1}. 

        Next, we introduce the notion of effective width for the multivariate case, that is a straightforward generalization of the analogous univariate term stated in~\Cref{def:effective-width}. 
        Let $\mc{C}$ denote a method for constructing confidence sequences for the mean $\mu$ of a distribution $P^* \in \mc{P}_0 \subset \mc{P}(\mc{X})$. Given a stream of observations $X_1, X_2, \ldots \simiid P^*$,  define the random stopping time 
        \begin{align}
            T_w \equiv T_w(P^*, \alpha) \defined \inf \lbr n \geq 1:  \sup_{x, x' \in \mc{C}(X^n, \alpha)} \, \|x-x'\|_2 \leq w \rbr,  \quad \text{for } w \in (0, \infty). 
        \end{align}
        Then, the effective width of the CS  after $n$ observations, is defined as:
        \begin{align}
            w_e(n, P^*, \alpha) \defined \inf \lbr w \geq 0: \mathbb{E}[T_w(P^*, \alpha)] \leq n \rbr. 
        \end{align}
        Our next result obtains a lower bound on the effective width of any scheme for constructing a CS for the mean vector based on $\iid$ observations. 
        \begin{proposition}
        \label{prop:multivariate-lower-bound-2}
            Suppose $\mc{C}$ is a method for constructing confidence sequences that satisfies $\mathbb{E}[T_w(P, \alpha)] < \infty$ for all $w>0$, and for all $P \in \mc{P}_0 \subset \mc{P}(\mc{X})$ for $\mc{X} = \mathbb{R}^d$. With  $a(n, \alpha) \defined  \log\lp (1-\alpha)^{1-\alpha} \alpha^{2\alpha-1} \rp /n$, introduce the following term: 
            \begin{align}
                w^*(n, P, \alpha)  =  \klinfmulti(P, \cdot, \mc{P}_0)^{-1}\lp {a(n, \alpha)} \rp. 
           \end{align}
           Then the effective width of the CS constructed by strategy $\mc{C}$ using \iid observations from $P \in \mc{P}_0$ must satisfy $w_e(n, P, \alpha) \geq w^*(n, P, \alpha)$. 
        \end{proposition}
        The proof of this result is also a simple generalization of the corresponding univariate result~(\Cref{prop:lower-bound-2}), and we present the details in~\Cref{proof:multivariate-lower-bound-2}. 
    
    \subsection{Sampling without replacement}
    \label{subsec:without-replacement}
     Suppose $\Xsample_M$ denotes a collection of $M$ numbers, $\{x_1, \ldots, x_M\}$, each lying in the unit interval $[0,1]$. Define the mean and variance terms as 
    \begin{align}
        \mu \equiv \mu_M = \frac{1}{M} \sum_{i=1}^M x_i, \quad \text{and} \quad \sigma^2 \equiv \sigma_M^2 = \frac{1}{M} \sum_{i=1}^M (x_i- \mu)^2. 
    \end{align}
    \begin{sloppypar}
    Suppose $X_1, X_2, \ldots$ are drawn uniformly from this set without replacement~(\wor); that is $X_1 \sim \text{Uniform}(\Xsample_M)$, $X_2 \sim \text{Uniform}\lp \Xsample_M \setminus \{X_1\} \rp$, and so on. Our goal is to construct a high probability (i.e, with probability at least $1-\alpha$, for a given $\alpha$) estimate of the mean $\mu$, based on these observations drawn \wor. We being be recalling some existing results on this topic most relevant to us. 
    \end{sloppypar}

    As in the without replacement case, we can construct CIs based on standard exponential concentration inequalities. However, one key feature of the \wor setting is that the uncertainty about the mean $\mu$ rapidly decays to zero (exactly) as the sample-size $n$ approaches $M$. This fact was captured by~\citet{serfling1974probability}, who obtained an improved version of an earlier inequality by~\citet{hoeffding1963probability}, with the term $n$ replaced by $\frac{n}{1 - (n-1)/M}$. \citet{bardenet2015concentration} obtained a slight improvement of Serfling's result when $n \geq M/2$, and then used it to derive variants of Hoeffding-Serfling, Bernstein-Serfling, and Empirical-Bernstein-Serfling inequalities. Finally, \citet{waudby2023estimating} showed that their techniques also extend easily to the \wor  case, and in particular, they proposed  \wor variants of their \prpieb CI and betting CI. We discuss the details about the construction of these CIs in~\Cref{appendix:wor-sampling}.

    In this section, we introduce an analog of the first-order limiting width defined in~\eqref{eq:first-order} for the \wor case. First we state our main assumption. 
    \begin{assumption}
        \label{assump:eb-cs-limiting-width}
        Considering a sequence of problems $(\Xsample_M)_{M \geq 1}$, with $\mu_M \defined \frac{1}{M} \sum_{i=1}^M x_i$, assume that: 
        \begin{align}
            \lim_{M \to \infty} \frac{1}{M} \sum_{i=1}^M (x_i - \mu_M)^2 = \lim_{M \to \infty} \sigma_M^2 = \sigma^2 >0. 
        \end{align}
        A simpler version of the above condition is assuming that $\sigma_M^2=\sigma^2$ and $\mu_M=\mu$ for all $M \in \mathbb{N}$. 
    \end{assumption}
    To define the limiting width of the CI in the \wor case, we fix a $\rho \in (0,1]$, and consider the limiting value of the scaled width of the CI based on $n = \lceil \rho M \rceil$ observations, as $M \to \infty$. That is, 
    \begin{align}
        \gamma_1(\rho) \defined \limsup_{M \to \infty, n = \lceil \rho M \rceil} \, \sqrt{n} w_n. 
    \end{align}
    The main result of this section is a \wor analog of~\Cref{prop:betting-CI-vs-EB-CI}. In particular, it says that for the same choice of the bets, the limiting width of the \wor betting CI is never larger than that of the \wor \prpieb CI. Furthermore, both of these CIs capture the improvement in the width of the CI when $n$ approaches $M$, or equivalently when $\rho \to 1$. 
    \begin{proposition}
        \label{prop:wor-limiting-width}
        Suppose~\Cref{assump:eb-cs-limiting-width} is true. Then, with $n = \lfloor \rho M \rfloor$ for some $\rho>0$, we have 
        \begin{align}
             \gamma_1^{(bet)}(\rho) \leq \gamma_1^{(\prpieb)}(\rho) =  2\sigma \sqrt{ 2\log (2/\alpha)} \lp  \frac{\rho}{-\log\lp1-\rho\rp} \rp. 
        \end{align}
    \end{proposition}
        The proof of this statement is in~\Cref{proof:wor-limiting-width}. 
        \begin{sloppypar}
            \begin{remark}
                Consider the case when $\rho=1/2$. Then, we have $\lim_{M \to \infty} \sqrt{n} w_n =  \frac{2\sigma \sqrt{2 \log(2/\alpha)}}{2\, \log 2} \approx 1.44 \sigma \sqrt{2 \log(2/\alpha)}$. In comparison,  the Bernstein CI derived  by~\citet{bardenet2015concentration} with the knowledge of the variance, has a limiting value of $2\sigma\sqrt{2 \log(2/\alpha)(1-1/2)} \approx 1.414 \sigma \sqrt{2 \log(2/\alpha)}$. While close, the Bernstein CI of \citet{bardenet2015concentration} has a better limiting width. Investigating whether the limiting widths of betting and \prpieb can be improved, as well as establishing the fundamental performance limits in the \wor case~(i.e., analogs of~\Cref{prop:lower-bound-1} and~\Cref{prop:lower-bound-2}) are interesting directions for future work.  
            \end{remark}
        \end{sloppypar}       

    \subsection{Second-order limiting width}
    \label{subsec:second-order}
        In~\eqref{eq:first-order}, we defined the notion of the first-order limiting width~(denoted by $\gamma_1$) to compare different CIs. Since the width of most practical CIs converges to zero at $\mc{O}(1/\sqrt{n})$ rate, their first-order limiting width is either a constant~(for Hoeffding, Bernstein, \mpeb, \prpieb), or upper bounded by a constant~(betting CI)  almost surely. In this section, we explore the notion of second-order limiting widths for CIs. 

        \begin{definition}
            \label{def:second-order-limiting-width} 
            Let $\mc{C}$ denote a method of constructing CIs: that is, for any $n \geq 1$, $\alpha \in (0,1)$, and $(X_1, \ldots, X_n) \simiid P^*$, the set $C_n = \mc{C}(X^n, \alpha)$ is a level-$(1-\alpha)$ CI for the mean $\mu$. Denote the width of $C_n$ with $2w_n$, and assume that its first-order limiting width, $\gamma_1$, is a constant almost surely. Then, we define the second-order limiting width as 
            \begin{align}
                \gamma_2 = \lim_{n \to \infty} \,n \times \lp w_n - \frac{\gamma_1}{\sqrt{n}} \rp.  
            \end{align}
        \end{definition}
        For  CIs with deterministic widths, such as Hoeffding or Bernstein, $\gamma_2$ is a constant. For example, $\gamma_2^{(H)} = 0$, and $\gamma_2^{(B)} = 4 \log(2/\alpha)/3$. For CIs with random widths, such as \mpeb, due to the fluctuations in the estimated parameter being inflated by a factor of $n$, characterizing $\gamma_2$ in an almost sure sense is ruled out. Instead, we need to explore weaker forms of convergence for the second-order limiting widths of such CIs. For example, we can prove that $\gamma_2^{(\mpeb)}$ converges in distribution to a Gaussian random variable with mean $14\log(4/\alpha)/3$, as we show in our next result. In the statement of this result, we use $\mu_4$ to denote the fourth central moment of $X$; that is, $\mu_4 = \mathbb{E}[(X-\mu)^4]$.  
        \begin{proposition}
            \label{prop:second-order-mpeb} With $w_n^{(\mpeb)}$ denoting the  width of the \mpeb CI, we have 
            \begin{align}
                n \lp w_n^{(\mpeb)} - \frac{\gamma_1^{(\mpeb)}}{\sqrt{n}} \rp = \frac{14 \log(4/\alpha)}{3} + T^{(\mpeb)}_n + o_P(1), 
            \end{align}
            where $o_P(1)$ is a term converging to $0$ in probability, and $T_n^{(\mpeb)}$ denotes a zero-mean statistic, that converges in distribution to $N(0, 2\log(4/\alpha)(\mu_4 - \sigma^4))$. Similarly, the \prpieb CI satisfies 
             \begin{align}
                 n \lp w_n^{(\prpieb)} - \frac{\gamma_1^{(\prpieb)}}{\sqrt{n}} \rp = \frac{4 \log(2/\alpha)}{3} + T^{(\prpieb)}_n + o_P(1), 
            \end{align}
            where $T_n^{(\prpieb)}$ is again a zero-mean random variable. 
        \end{proposition}       
        The proof of this result is in~\Cref{proof:second-order-mpeb}. 
        Unlike \mpeb, the leading constant, $4\log(2/\alpha)/3$ in the above expression for the second-order width statistic of \prpieb CI matches the second-order term of the Bernstein CI. However, we have not managed to prove the  convergence of $T_n^{(\prpieb)}$ in any suitable sense. 
        It is an interesting question for future work whether we can modify the \prpieb CI~(or consider an alternative the definition of second-order limiting width, that may be more suited to CIs with random widths), to guarantee that it dominates $T_n^{(\mpeb)}$ in an appropriate sense.

\section{Conclusion}
\label{sec:conclusion}
    In this paper, we provided theoretical justification for the superior empirical performance of the betting confidence intervals~(CIs) and confidence sequences~(CSs) reported by~\citet{waudby2023estimating}. We first showed that the first-order limiting width of the betting CI (with a specific choice of bets) is strictly smaller than that of the empirical Bernstein CIs of~\citet{maurer2009empirical}, and is at least as small as the limiting width of a Bernstein CI constructed with the knowledge of the true variance. Then, we established the near-optimality of the betting CI (constructed using a mixture strategy) in the finite-$n$ regime, by showing that the width of the betting CI is of the same order as the fundamental, method-agnostic, limit. Finally, we showed that a similar near-optimality guarantee also holds for  betting CSs. Together these results provide strong theoretical arguments supporting the empirically observed superiority of betting CI/CS in both the asymptotic, and finite-sample regimes. 

    Our results open up several interesting directions for future work. For instance, our non-asymptotic analysis leaves a gap~(inflation by a constant factor, and a $\mc{O}\lp \log(n)/n\rp$ additive term in the argument of inverse information projection) between the upper bound achieved by the betting CI/CS, and the corresponding method-agnostic lower bounds. Investigating whether this gap can be closed by using more refined analysis techniques is an important question for future work. Another important direction is to extend our results to the without-replacement~(\wor) sampling setting. We presented a preliminary result on this topic in~\Cref{subsec:without-replacement}, but developing the analogs of our non-asymptotic results for the \wor case will require fundamentally new ideas. Finally, as we showed in~\Cref{subsec:multivariate}, our non-asymptotic lower bounds on the size of CIs can be easily extended to more general observation spaces, beyond the univariate case that was the main focus of this paper. Designing CIs/CSs that match these lower bounds for these general observation spaces using the betting framework is another interesting question for future work. 

\subsection*{Acknowledgement}
The authors acknowledge support from NSF grants IIS-2229881 and DMS-2310718. 
 
\bibliographystyle{abbrvnat}
\bibliography{ref}
 \newpage
\begin{appendix}

\section{Additional background}
\label{appendix:background}
    In this section, we collect some of the results from prior work that we use often in proving our theoretical results. We begin by recalling the first Borel-Cantelli lemma~\citep[Theorem~2.3.1]{durrett2019probability}. 
    \begin{fact}[Borel-Cantelli lemma]
        \label{fact:borel-cantelli}
        Consider a probability space $(\Omega, \mc{F}, \mathbb{P})$, and let  $\{E_n \in \mc{F}: n \geq 1\}$ denote a sequence of events, satisfying $\sum_{n =1}^{\infty} \mathbb{P}\lp E_n \rp < \infty$. Then, we have 
        \begin{align}
            \mathbb{P}\lp \limsup_{n \to \infty} E_n \rp = \mathbb{P}\lp  \cap_{N\geq 1} \cup_{n \geq N} E_n \rp = \mathbb{P}\lp E_n \, i.o.\rp =0. 
        \end{align}
        The $i.o.$ above is an abbreviation for ``infinitely often''. 
    \end{fact}
    We next recall Wald's equation~\citep[Theorem 4.8.6]{durrett2019probability},  a standard result about the expected value of a sum of a random number of \iid random variables. 
    \begin{fact}[Wald's equation]
        \label{fact:walds-equation}
        Suppose $Z_1, Z_2, \ldots$ are \iid with $\mathbb{E}[Z_i] = b$, and let $S_n$ denote the sum $\sum_{i=1}^n Z_i$ for all $n \geq 1$. Then, if $T$ is any stopping time with $\mathbb{E}[T]<\infty$, we have $\mathbb{E}[S_T] = b \mathbb{E}[T]$.  
    \end{fact}
    Next, we recall a useful inequality derived by~\citet{fan2015exponential} that we use in obtaining lower bounds on the wealth process used in the betting CI/CS. 
    \begin{fact}[Fan's inequality]
    \label{def:fans-inequality}
    For any $x \geq -1$, and $\lambda \in (0,1)$, we have the following 
    \begin{align}
        \log(1+\lambda x) \geq \lambda x - 4 \psi_E(\lambda) x^2, \quad \text{where} \quad 
        \psi_E(\lambda) = \frac{-\lambda - \log(1-\lambda)}{4}.  \label{eq:fan-inequality-0}
    \end{align}
    This was first proved by~\citet{fan2015exponential}. 
    \end{fact}

    Finally, we state some results derived by~\citet{waudby2023estimating} while analyzing their betting CIs/CSs, that will be used several times in our proofs. 
    \begin{fact}
        \label{fact:limiting-computations} Consider the bets $\{\lambda_{t,n}: 1 \leq t \leq n\}$ as used in the definition of \prpieb in~\eqref{eq:prpieb-def}. Then, we have the following: 
        \begin{align}
            \lim_{n \to \infty} \frac{1}{\sqrt{n}} \sum_{t=1}^n \lambda_{t,n} \eqas \sqrt{ \frac{2 \log(2/\alpha)}{\sigma^2}}, \quad \text{and} \quad 
            \lim_{n \to \infty} \sum_{t=1}^n 4\psi_E(\lambda_{t,n}) (X_t - \muhat_{t-1})^2 \eqas \sqrt{ \log(2/\alpha)}. 
        \end{align}
    \end{fact}

\section{Proof of the limiting width of betting CI~(Theorem~\ref{prop:betting-CI-vs-EB-CI})}
    \label{proof:betting-CI-vs-EB-CI}
    
        Following the proof outline presented in~\Cref{sec:limiting-width-betting}, we break down the proof into several smaller components.
        First, we recall some notation: 
        \begin{align}
            \lambda_{t,n} = \sqrt{\frac{2 \log(2/\alpha)}{n V_{t-1}}},\; \text{ for } t\geq 2,  \quad \text{where} \quad 
            V_{t} =  \frac{1}{t} \lp \frac{1}{4} +  \sum_{i=1}^{t}(X_i - \muhat_{i-1})^2 \rp, \quad \text{and} \quad 
            \muhat_{t} = \frac{1}{t} \sum_{i=1}^{t} X_i, 
        \end{align}
        We begin by showing that the betting CI is contained in a larger CI whose width converges to $0$ at the order-optimal $\mc{O}(1/\sqrt{n})$ rate. 
        \begin{lemma}
            \label{lemma:betting-CI-1} Suppose we construct a betting CI based on $n$ \iid observations $X_1, \ldots, X_n$ drawn from $P^*$, with bets $\{\lambda_{t,n}: t \geq 1\}$ for all $m \in [0,1]$, where $\lambda_{1,n} = 0$, and $\lambda_{t,n}$ is as defined above for $t \geq 2$. 
            Then, the betting CI is contained inside a larger CI~(that we denote by $\CIfan$), defined as 
            \begin{align}
                \CIbet \subset \CIfan \defined \lb \frac{ \sum_{t=1}^n\lambda_{t,n} X_t}{\sum_{t=1}^n\lambda_{t,n}} \pm \frac{ \log(2/\alpha) + \sum_{t=1}^n \psi_E(\lambda_{t,n})}{\sum_{t=1}^n \lambda_{t,n}} \rb. 
            \end{align}
            Recall that $\psi_E(\lambda) = \frac{-\lambda - \log(1-\lambda)}{4}$ was defined in~\eqref{eq:fan-inequality-0}. As a result of the above, we have the following bound on the limiting width of the betting CI: 
            \begin{align}
                \limsup_{n \to \infty} \,\sqrt{n} \times |C_n^{\text{bet}}|   \; \stackrel{a.s.}{\leq} \; c_1 \defined \sqrt{ \frac{ \log(2/\alpha)}{2}}\lp \sigma + \frac{1}{\sigma}\rp < \infty. 
            \end{align}
        \end{lemma}
        The proof of this statement is in~\Cref{proof:betting-CI-1}. 
        While this result implies that the width of the betting CI converges to zero at the correct~(i.e., $1/\sqrt{n}$) rate, it does not explain the  fact that the width of the betting CI shrinks appropriately when the variance~($\sigma^2$) is small. Nevertheless, this result will play an important role in obtaining a more refined analysis of the betting CI width, as it allows us to concentrate our analysis in a narrow band of width $\approx 4c_1/\sqrt{n}$ for large enough $n$. First, we state a technical lemma about the concentration of the following mean and variance estimates: 
        \begin{align}
            \muhat_n \defined \frac{1}{n} \sum_{i=1}^n X_i, \quad \mutilde_n \defined  \big(  \sum_{t=1}^n \lambda_{t,n} X_t\big) / \big( \sum_{t=1}^n \lambda_{t,n}\big), \quad \text{and} \quad \sigmatilde_n^2 = \frac{1}{n} \sum_{i=1}^n (X_i - \mu)^2.  \label{eq:empirical-estimates-1}
        \end{align}
        \begin{lemma}
            \label{lemma:betting-CI-2}
            Let $E_t$ denote the intersection of the following three events: 
            \begin{align}
                &E_{t,1} = \lbr |\muhat_t - \mu| \leq \sqrt{\log(t)/t} \rbr, \quad 
                E_{t,2} = \lbr |\sigmatilde_t^2 - \sigma^2_X| \leq \sqrt{\log(t)/t} \rbr, \\ \text{and} \quad 
                &E_{t,3} = \lbr |\mutilde_t - \muhat_t| = \mc{O}\lp (1/\sigma^2)\sqrt{\log(t)/t} \rp\rbr. 
            \end{align}
            Then we have $\mathbb{P}\lp E_t^c \rp \leq 8/t^2$. 
        \end{lemma}
         We prove this result by showing that $\mathbb{P}\lp E_{t,i}^c \rp \leq 2/t^2$ for $i=1,2,$ and $\mathbb{P}(E_{t,3}^c) \leq 4/t^2$ by using some standard concentration techniques for bounded observations. Since our objective is to study the asymptotic behavior of the width, we have not attempted to optimize deviation bound in the definition of $E_{t,3}$. The details are in~\Cref{proof:betting-CI-2}.
         
        The two results above set the stage for our main approximation result that will lead to the improved first-order limiting width bound.  To state this result, we  introduce the following new  CI, based on~\Cref{lemma:betting-CI-1}: 
        \begin{align}
            \widetilde{C}_n = [\mutilde_n - 2c_1/\sqrt{n}, \,\mutilde_n + 2c_1/\sqrt{n}], 
        \end{align}
        where recall that $\mutilde_n$ is the empirical estimate introduced in~\eqref{eq:empirical-estimates-1}, and the constant $c_1$ was defined in the statement of~\Cref{lemma:betting-CI-1}. Furthermore, introduce the random time 
        \begin{align}
            N \defined \sup \lbr n \geq 1: |\CIfan| > \frac{4 c_1}{\sqrt{n}} \rbr,  \label{eq:def-N}
        \end{align}
        where $\CIfan$ was introduced in~\Cref{lemma:betting-CI-1}.  We can now state our approximation result. 
        \begin{lemma}
            \label{lemma:betting-CI-3}
            The random variable $N$ is finite almost surely, and for all $n \geq N$ we have $\CIbet \subset \CIfan \subset \CItilde$. 
            Furthermore, the following is true for all $m$ lying in $\CItilde$: 
            \begin{align}
                \boldsymbol{1}_{n \geq N} \times \lp \sum_{t=1}^n \psi_E(\lambda_{t,n})(X_t-m)^2 \boldsymbol{1}_{E_{t-1}} \rp \; \stackrel{a.s.}{\leq} \; \sum_{t=1}^n \psi_E(\lambda_{t,n}) (X_t-\muhat_{t-1})^2 + o(1), 
            \end{align}
            where $\boldsymbol{1}_{E_0} \eqas 1$, and  $E_t$ for $t \geq 1$ was introduced in~\Cref{lemma:betting-CI-2}. 
        \end{lemma}
        The proof of this statement is in~\Cref{proof:betting-CI-3}. 

        We now have all the components to show that the first-order limiting width is upper bounded by $\gamma_1^{(\prpieb)} \defined \sigma\sqrt{2 \log(2/\alpha)}$.  According to the definition of the random time $N$, we have 
        \begin{align}
            \CIbet \subset \widetilde{C}_n, \quad \text{under the event} \{N \leq n\}. 
        \end{align}
        Next, recall the definition of the wealth process used in the statement of~\Cref{prop:betting-CI-vs-EB-CI}:  
        \begin{align}
            W_n(m) = \frac{1}{2} \lp W_n^+(m) + W_n^-(m) \rp, \quad \text{where} \quad 
            W_n^{\pm}(m) = \prod_{t=1}^n \lp 1 \pm \lambda_{t,n} \lp X_t - m \rp \rp. 
        \end{align}
        As a simple consequence of this choice of the wealth process, we note that  the betting CI  satisfies
        \begin{align}
            \CIbet \subset \cap_{a \in \{-, +\} } \lbr m \in [0,1]: \log(W_n^a(m)) < \log(2/\alpha) \rbr. 
        \end{align}
        This immediately implies the following  under $\{N \geq n\}$: 
        \begin{align}
            \CIbet & \subset \lbr m \in \CItilde: \log\lp W_n^+(m) \rp < \log(2/\alpha) \rbr \\
            &\subset \bigg\{ m \in \widetilde{C}_n: \sum_{t=1}^n \lambda_{t,n}(X_t-m) - 4 \sum_{t=1}^n \psi_E(\lambda_{t,n})(X_t-m)^2 < \log(2/\alpha) \bigg\}. \label{eq:betting-ci-proof-4}
        \end{align}
        This inclusion is due to a direct application of Fan's inequality~(\Cref{def:fans-inequality}) in the definition of $W_n^+(m)$. On rearranging~\eqref{eq:betting-ci-proof-4}, we obtain
        \begin{align}
            \CIbet \subset \bigg\{ m \in \widetilde{C}_n:  m \geq \mutilde_n  - \frac{ \log(2/\alpha) + 4\sum_{t=1}^n \psi_E(\lambda_{t,n})(X_t-m)^2}{\sum_{t=1}^n \lambda_{t,n}} \bigg\}, \label{eq:betting-ci-proof-5}
        \end{align}
        where recall that $\mutilde_n$ was introduced in~\eqref{eq:weighted-mean}. Repeating the above argument with $W_n^-(m)$, we get (again, under the event $\{n \geq N\}$) 
        \begin{align}
            \CIbet \subset \bigg\{ m \in \widetilde{C}_n:  m \leq \mutilde_n  + \frac{ \log(2/\alpha) + 4\sum_{t=1}^n \psi_E(\lambda_{t,n})(X_t-m)^2}{\sum_{t=1}^n \lambda_{t,n}} \bigg\}, \label{eq:betting-ci-proof-15}
        \end{align}
        Let $m_*$ and $m^*$ denote the (random) smallest and largest values respectively, that satisfy the condition in~\eqref{eq:betting-ci-proof-5} and~\eqref{eq:betting-ci-proof-15}.  Then, we have the following: 
        \begin{align}
            \boldsymbol{1}_{n \geq N} \times |\CIbet| \; \stackrel{a.s.}{\leq} \; m^* - m_* =  \frac{2 \log(2/\alpha) + 4 \sum_{t=1}^n \psi_E(\lambda_{t,n}) \lp (X_t-m^*)^2 + (X_t-m_*)^2 \rp }{\sum_{t=1}^n \lambda_{t,n}} 
        \end{align}
        Next, note  that since $N$ is known to be finite almost surely from~\Cref{lemma:betting-CI-3}, we have 
        \begin{align}
           \gammabet \defined \limsup_{n \to \infty} \sqrt{n} \times |\CIbet| \eqas \limsup_{n \to \infty} \sqrt{n} \times |\CIbet| \times \boldsymbol{1}_{n \geq N}. 
        \end{align}
        Thus,~\eqref{eq:betting-ci-proof-5} and~\eqref{eq:betting-ci-proof-15} imply the bound: 
        \begin{align}
            \frac{\gammabet}{2} \leq \limsup_{n \to \infty}\; \sqrt{n} \times \frac{ \log(2/\alpha) + 2 \sum_{t=1}^n \psi_E(\lambda_{t,n}) \lp (X_t-m^*)^2 + (X_t-m_*)^2 \rp }{\sum_{t=1}^n \lambda_{t,n}} \times \boldsymbol{1}_{n \geq N}. 
        \end{align}
        This limit appears similar to the width of the improved EB-CI of~\citet{waudby2023estimating}, with the only difference being the second term in the numerator. We first recall from~\Cref{fact:limiting-computations} that 
        \begin{align}
            \lim_{n \to \infty} \frac{1}{\sqrt{n}} \sum_{t=1}^n \lambda_{t,n}  \eqas \sqrt{2 \log(2/\alpha)/\sigma^2}. 
        \end{align}
        Thus, to complete the proof, we need to show the following, with $U_t \defined (X_t-m^*)^2 + (X_t-m_*)^2$:  
        \begin{align}
            \lim_{n \to \infty} \lp 2 \sum_{t=1}^n \psi_E(\lambda_{t,n}) \,U_t \rp \times \boldsymbol{1}_{n \geq N} \; \eqas \;  \sqrt{\log(2/\alpha)}. 
        \end{align}
        Since we already know that $\lim_{n \to \infty} 4 \sum_{t=1}^n \psi_E(\lambda_{t,n})(X_t-\muhat_{t-1})^2 \eqas \sqrt{\log(2/\alpha)}$, it suffices to prove 
         \begin{align}
            \limsup_{n \to \infty} \lp 2 \sum_{t=1}^n \psi_E(\lambda_{t,n}) \,U_t \rp \times \boldsymbol{1}_{n \geq N} \stackrel{a.s.}{\leq}
            \lim_{n \to \infty} 4 \sum_{t=1}^n \psi_E(\lambda_{t,n})(X_t-\muhat_{t-1})^2.  
        \end{align}       
        We recall from~\Cref{lemma:betting-CI-2} that $\mathbb{P}(E_t^c) \leq 8/t^2$, which means that $\sum_{t=1}^\infty \mathbb{P}(E_t^c) < \infty$. An application of Borel-Cantelli lemma then implies that 
        \begin{align}
            \mathbb{P}\lp E_t^c \text{ infinitely often} \rp = 0 \; \Longleftrightarrow \; 
            \mathbb{P}\lp E_t \text{ all but finitely many times} \rp = 1. 
        \end{align}
        As a consequence of this, we have 
        \begin{align}
            \limsup_{n \to \infty} \lp 2 \sum_{t=1}^n \psi_E(\lambda_{t,n}) \,U_t {\boldsymbol{1}_{E_{t-1}}} \rp \times \boldsymbol{1}_{n \geq N} \eqas \limsup_{n \to \infty} \lp 2 \sum_{t=1}^n \psi_E(\lambda_{t,n}) \,U_t \rp \times \boldsymbol{1}_{n \geq N}
        \end{align}
        The final step is provided by~\Cref{lemma:betting-CI-3}, which gives us 
        \begin{align}
            \boldsymbol{1}_{n \geq N} \lp 2\sum_{t=1}^n \psi_E(\lambda_{t,n}) U_t \boldsymbol{1}_{E_{t-1}}\rp \stackrel{a.s.}{\leq} 
            \lp 4 \sum_{t=1}^n \psi_E(\lambda_{t,n})(X_t - \muhat_{t-1})^2 \rp + o(1). 
        \end{align}
        Taking the limit, we obtain 
        \begin{align}
            \limsup_{n \to \infty} \lp 2 \sum_{t=1}^n \psi_E(\lambda_{t,n}) \,U_t {\boldsymbol{1}_{E_{t-1}}} \rp \times \boldsymbol{1}_{n \geq N} &\stackrel{a.s.}{\leq} \lim_{n \to \infty} 4 \sum_{t=1}^n \psi_E(\lambda_{t,n})\lp (X_t - \muhat_{t-1})^2 \rp + o(1) \\
            & = \sqrt{\log(2/\alpha)}. 
        \end{align}
        The last equality follows from the fact that $\psi_E(\lambda_{t,n}) \approx \lambda_{t,n}^2/8 = \log(2/\alpha)/(4n \sigmahat_t^2) \approx \log(2/\alpha)/(4n\sigma^2)$, as recalled in~\Cref{fact:limiting-computations}.   This completes the proof. \hfill \qedsymbol

        \subsection{Proof of~Lemma~\ref{lemma:betting-CI-1}}
        \label{proof:betting-CI-1}
            To prove this statement, we recall the definition of the betting CI used in~\Cref{prop:betting-CI-vs-EB-CI}: 
            \begin{align}
                &\CIbet = \{m \in [0,1]: W_n(m) < 1/\alpha\}, \\ \text{where} \quad 
                &W_n(m) =  \frac{1}{2} \lp W_n^+(m) + W_n^-(m) \rp, \quad 
                \text{and} \quad 
                 W_n^{\pm}(m) = \prod_{t=1}^n \lp 1 \pm \lambda_{t,n}(X_t - m) \rp. 
            \end{align}
            Since both $W_n^{+}(m)$ and $W_n^{-}(m)$ are nonnegative by construction, it implies that 
            \begin{align}
                \{m: W_n(m) < 1/\alpha \} \subset \{m: W_n^{+}(m) < 2/\alpha  \} \cap \{m: W_n^{-}(m) < 2/\alpha  \}. 
            \end{align}
            We now consider the two events in the right side of the above display, and use Fan's inequality~(\Cref{def:fans-inequality}) to obtain an interval containing them. In particular, Fan's inequality implies the following lower bound 
            \begin{align}
                \log(W_t^+(m)) = \sum_{t=1}^n \log(1 + \lambda_t(X_t-m)) \geq \sum_{t=1}^n \lambda_t (X_t-m) - 4\psi_E(\lambda_t) (X_t-m)^2. 
            \end{align}
            Now, since $(X_t-m)^2 \leq 1$ for all $t \geq 1$ and $m \in [0,1]$, we can further lower bound $W_t^+(m)$ as follows: 
            \begin{align}
                \log(W_t^+(m)) \geq \sum_{t=1}^n \lambda_t(X_t-m) - \psi_E(\lambda_t). 
            \end{align}
            This leads us to the following inclusion: 
            \begin{align}
                \{m: W_n^+(m) < 2/\alpha\} &\subset \lbr m \in [0,1]:  \sum_{t=1}^n \lambda_{t,n}(X_t-m) - \sum_{t=1}^n \psi_E(\lambda_{t,n}) < \log(2/\alpha)\rbr,  \label{eq:betting-fan-1}\\
                & \defined [L, 1], \quad \text{where} \quad 
                L = \frac{\sum_{t=1}^n \lambda_t X_t}{\sum_{t=1}^n\lambda_t}  - \frac{\log(2/\alpha) + \sum_{t=1}^n \psi_E(\lambda_t)}{\sum_{t=1}^n\lambda_t}. 
            \end{align}
            To simplify the notation, we will use 
            \begin{align}
                \mutilde_n = \frac{\sum_{t=1}^n \lambda_t X_t}{\sum_{t=1}^n\lambda_t}, \quad 
                \frac{w_n}{2} = \frac{\log(2/\alpha) + \sum_{t=1}^n \psi_E(\lambda_t)}{\sum_{t=1}^n\lambda_t}.  \label{eq:width-of-Fan-CI}
            \end{align}
            Next, we repeat the same argument for $W_n^-(m)$, to obtain 
            \begin{align}
                \{m: W_n^-(m) < 2/\alpha\} &\subset \lbr m \in [0,1]:  \sum_{t=1}^n \lambda_{t,n}(m-X_t) - \sum_{t=1}^n \psi_E(\lambda_{t,n}) < \log(2/\alpha)\rbr, \label{eq:betting-fan-2}  \\
                & \defined [0, U], \quad \text{where} \quad  
                U = \mutilde_n + \frac{w_n}{2}.
            \end{align}
            Combining~\eqref{eq:betting-fan-1} and~\eqref{eq:betting-fan-2}, we get 
            \begin{align}
                \CIbet &\subset \{m: W_n^+(m) < 2/\alpha\} \cap \{m: W_n^-(m) < 2/\alpha\} \\
                & \subset \CIfan \defined [L, U] = \lb \mutilde_n  \pm \frac{w_n}{2}\rb. \label{eq:CI-Fan-def}
            \end{align}
            To conclude the proof, we recall that~(\Cref{fact:limiting-computations}): 
            \begin{align}
                &\sum_{t=1}^n \lambda_t \eqas \sqrt{ \frac{ 2 \log(2/\alpha) n }{\sigma^2}}\lp 1 + o(1) \rp, \quad \text{and} \quad \sum_{t=1}^n \psi_E(\lambda_t) \eqas \frac{\log(2/\alpha)}{\sigma^2}\lp 1  + o(1) \rp. 
            \end{align}
            This implies that the width of the CI satisfies: 
            \begin{align}
                \limsup_{n \to \infty} \sqrt{n} \times |\CIbet| \leq   2\sqrt{\log(2/\alpha)/2}\lp \sigma + 1/\sigma \rp = 2c_1. 
            \end{align}
            This completes the proof.  \hfill \qedsymbol
            
        \subsection{Proof of~Lemma~\ref{lemma:betting-CI-2}}
            \label{proof:betting-CI-2}
            The proof of concentration of $\muhat_t$ and $\sigmatilde_t$ follows by a direct application of Hoeffding's inequality. For the concentration of~$\mutilde_t$, we need a slightly more involved peeling argument, because it relies on a uniform concentration of $\sigmatilde_i$ for all $1 \leq i \leq t$. 
            
            \paragraph{Concentration of sample mean and variance.} Recall the events 
            \begin{align}
                E_{t,1} = \lbr |\muhat_t - \mu| \leq \sqrt{\log(t)/t} \rbr, \quad 
                E_{t,2} = \lbr \left\lvert \frac{1}{t} \sigmatilde_t^2 - \sigma^2 \right\rvert \leq \sqrt{\log(t)/t} \rbr, 
            \end{align}
            where $\muhat_t = (1/t) \sum_{i=1}^t X_i$, and $\sigmatilde_t^2 = (1/t)\sum_{i=1}^t (X_i-\mu)^2$. 
            Since the observations are bounded, Hoeffding's inequality implies that
            \begin{align}
                \mathbb{P}\lp |\muhat_t - \mu| \leq \sqrt{ \log(2/\delta)/(2t)} \rp \leq \delta, \quad \text{and} \quad 
                \mathbb{P}\lp |\sigmatilde_t^2 - \sigma^2| \leq \sqrt{ \log(2/\delta)/(2t)} \rp \leq \delta. 
            \end{align}
            Setting $\delta \leftarrow 2/t^2$ in the two inequalities above leads to the conclusion $\mathbb{P}\lp E_{i,t}^c \rp \leq 2/t^2$, for $i=1,2$. 

            \paragraph{Concentration of the weighted mean $\boldsymbol{\mutilde_t}$.} We now look at the remaining event, $E_{t,3}$, defined as 
            \begin{align}
                E_{t,3} = \lbr |\mutilde_t - \mu| \leq \sqrt{\log(t)/t} \rbr, \quad \text{where} \quad  \mutilde_t  \defined \big(  \sum_{t=1}^n \lambda_{t,n} X_t\big) / \big( \sum_{t=1}^n \lambda_{t,n}\big). 
            \end{align}
            Since $\lambda_{t,n} = \sqrt{2 \log(2/\alpha) / (V_{t-1} n)}$, the estimate $\mutilde_t$ can we re-written as 
            \begin{align}
                \mutilde_t = \frac{ \sum_{i=1}^t X_i/\sqrt{V_{i-1}} }{\sum_{i=1}^t 1/\sqrt{V_{i-1}}}, \quad \text{with } V_0 = 1/4, \text{ and } V_i = \frac{1}{i}\lp \frac{1}{4} +  \sum_{j=1}^{i}(X_j-\muhat_{j-1})^2 \rp, \quad \text{for } i \geq 2. 
            \end{align}
            The above definition implies that to get a concentration result for $\mutilde_t$, we need to control the deviation of all $V_i$, for $i$ smaller than $t$. 
            By adding and subtracting $\mu$ in the definition of $V_i$, we can decompose $V_i$  into the following terms, for $i \geq 2$: 
            \begin{align}
                V_i = \frac{1}{i} \sum_{j=1}^i (X_j-\mu)^2 + \frac{1}{i} \sum_{j=1}^i(\mu - \muhat_{j-1})^2 + \frac{2}{i} \sum_{j=1}^i(X_j-\mu)(\mu - \muhat_{j-1}) + \frac{1}{4i}.  \label{eq:mutilde-conc-1}
            \end{align}
            In the above display, we use the convention that $\muhat_{0}=0$. The expression stated in~\eqref{eq:mutilde-conc-1} implies that in order to show the concentration of all $V_i$ around $\sigma^2$, it suffices to (i) show the uniform concentration of $\frac{1}{i}\sum_{j=1}^i (X_j-\mu)^2$ around $\sigma^2$, and (ii) a uniform concentration of $\muhat_i$ around $\mu$, for all $1 \leq i \leq t$. The concentration of the third term around $0$, then follows as a consequence of (ii). Thus, we will show that for the events, $G_{t,1} \defined \{\forall 1 \leq i \leq t: |\sigmatilde_i^2 - \sigma^2| \leq u_{i,t} \}$, and  $G_{t,2} \defined \{\forall 1 \leq i \leq t: |\muhat_i - \mu| \leq u_{i,t} \}$ with $u_{i,t} \defined \sqrt{3 \log(t)/i}$, the probabilities $\mathbb{P}\lp G_{t,1}^c \rp$ and $\mathbb{P}(G_{t, 2}^c)$ are both upper bounded by $2/t^2$. 
            Since the arguments are identical, we present the details only for $G_{t,1}$ below. 

            We begin by partitioning the set $[t] = \{1, 2, \ldots,  t\}$ into intervals of exponentially increasing lengths, $I_0, I_1, \ldots, I_{k_t}$, with $|I_j| = 2^j$. That is $I_0 = \{1\}$,  $I_1 = \{2, 3\}$, and $I_j = \{2^j, \ldots, 2^{j+1}-1\}$. The number of such intervals, is equal to $k_t= \lfloor \log_2(t) + 1\rfloor \leq \log_2t + 1$.  Since we want to control the probability of $G_{t,1}^c$ by $2/t^2$, we divide this `budget' equally among the $k_t$ intervals.  Let the elements of the $j^{th}$ interval be denoted by $\{i_j, i_{j+1}, \ldots, i_{j+1} -1 \}$, where we use the notation $i_j$ to denote the smallest element of the interval $I_j$, for all $1 \leq j \leq k_t$.  
            By using the time-uniform Hoeffding inequality in each interval, we get 
            \begin{align}
                \mathbb{P}\lp \forall i \in I_j: |\sigmatilde_i^2 - \sigma^2| \leq \sqrt{\log(2/\delta_t)/(2 i_j)} \rp \leq \delta_t, \quad \text{where} \quad 
                \delta_t \defined \frac{2}{t^2 k_t} \geq \frac{2}{t^2 \lp \log_2 t + 1 \rp} \geq \frac{2}{t^3} . 
            \end{align}
            Then, by taking a  union bound over the $k_t$ intervals, and using the fact that $2i_j \geq i$ for all $i \in I_j$, we get 
            \begin{align}
                \mathbb{P}\lp G_{t,1} \rp = \mathbb{P}\lp \forall i \in [t]: |\sigmatilde_i^2-\sigma^2| \leq \sqrt{3\log(t)/i} \rp \geq 1-\frac{2}{t^2}.  
            \end{align}
            An exactly analogous argument also implies that $\mathbb{P}(G_{t,2}) \geq 1-2/t^2$. 

            Going back to~\eqref{eq:mutilde-conc-1}, we see that under $E_{t,3} \defined G_{t,1} \cap G_{t,2}$, we have  for all $i \in \{2, \ldots, t\}$
            \begin{align}
                |V_i - \sigma^2| &\leq |\sigmatilde_i^2-\sigma|^2 + \frac{1}{i} \sum_{j=0}^{i-1}|\muhat_j-\mu|^2 + \frac{2}{i} \sum_{j=0}^i |\muhat_j - \mu|  + \frac{1}{4i} \\
                & \leq u_{i,t}^2 + \frac{1}{i} \sum_{j=1}^{i-1} u_{j,t}^2 + \frac{2}{i} \sum_{j=1}^{i-1} u_{j,t} + \frac{1}{4i} \\
                &\leq \frac{3 \log t}{i} + \frac{3 \log t}{i} \lp 1 + \sum_{j=1}^{i-1} \frac{1}{j} \rp + \frac{2\sqrt{3 \log t}}{i} \lp 1+ \sum_{j=1}^{i-1} \frac{1}{\sqrt{j}} \rp  + \frac{1}{4i} \\
                & \leq \frac{3 \log t \log ei }{i} + \sqrt{\frac{12 \log t}{i}} \defined v_{i,t}. \label{eq:vit-def}
            \end{align}
            Since we need to show that $|\mutilde_t - \muhat_t| = \mc{O}\lp \frac{1}{\sigma^2} \sqrt{\log t /t} \rp$, it is sufficient to prove that this bound holds for all $t \geq t_0$, for some finite $t_0$. This is because, we can choose a sufficiently large  leading constant in $\mc{O}(\cdot)$ to ensure that the condition holds trivially for all $t < t_0$. To identify such a $t_0$, we first fix a positive constant $a <1/3$~(for example $a=0.32$), and introduce the term $i_t = t^a\log^2(t)$. We want to select $t_0$ to ensure that for all $t \geq t_0$, and $i \geq i_t$, we have $v_{i,t} < \sigma^2/2$, where $v_{i,t}$ refers to the expression in~\eqref{eq:vit-def}. It is easy to verify that the following is true: 
            \begin{align}
                t > t_0 \defined \lp \frac{192}{\sigma^4} \rp^{1/a} \quad \Longrightarrow \quad v_{i,t} < \frac{\sigma^2}{2}, \;\; \text{for all } i \geq i_{t} \defined  t^a\log^2(t). 
            \end{align}
            To complete the proof we will show that $\mutilde_t \leq \muhat_t + \mc{O}\lp \frac{1}{\sigma^2}\sqrt{\log t/t} \rp$ for all $t > t_0$. The lower bound can be obtained by following a similar argument, and we omit the details to avoid repetition. 
            Observe that $\mutilde_t$ can be written as 
            \begin{align}
                \mutilde_t = \frac{ \sum_{i=1}^t X_i/\sqrt{V_{i-1}}}{\sum_{i=1}^t 1/ \sqrt{V_{i-1}}} = \frac{ \sum_{i=1}^t X_i/\sqrt{1 + \frac{V_{i-1} - \sigma^2}{\sigma^2} }}{\sum_{i=1}^t 1/ \sqrt{1 + \frac{V_{i-1}- \sigma^2}{\sigma^2}}} = \frac{A}{B}. 
            \end{align}
            We now use the convexity of the mapping $x \mapsto 1/\sqrt{1 -x}$ over the domain $[-1/2, 1/2]$, to get an upper bound (resp. lower bound) on the numerator $A$~(resp. denominator $B$). In particular, we have 
            \begin{align}
                A \leq \sum_{i=1}^{i_t} \frac{X_i}{\sqrt{1/4i}} + \sum_{i>i_t} \frac{X_i}{\sqrt{1 - \frac{v_{i,t}}{\sigma^2} }} \leq 2 i_t^{3/2} + \sum_{i>i_t} \frac{X_i}{\sqrt{1 - \frac{v_{i,t}}{\sigma^2} }} \leq 2 i_t^{3/2} + \sum_{i>i_t} X_i \lp 1 + \frac{v_{i,t}}{\sigma^2} \rp.  \label{eq:temp-upper-bound-1}
            \end{align}
            In the last inequality above, we used the fact that $\frac{1}{\sqrt{1-x}} \leq 1 + 2(\sqrt{2}-1)x$ for $x \in [0, 0.5)$, and that $2(\sqrt{2}-1) \leq 1$. Similarly, we obtain 
            \begin{align}
                B \geq \sum_{i=1}^{i_t} 0 + \sum_{i > i_t} \frac{1}{ \sqrt{1 + \frac{v_{i,t}}{\sigma^2}}} \geq (t - i_t) \lp 1 -  \frac{1}{2(t-i_t)}\sum_{i>i_t} \frac{v_{i,t}}{\sigma^2} \rp,  
            \end{align}
            where we used the fact that $1/\sqrt{1-x} \geq 1 + x/2$  for $x \in (-0.5, 0]$. 
            This implies that 
            \begin{align}
                \frac{1}{B} \leq \frac{1}{t- i_t} \lp 1 + \frac{1}{2(t - i_t)}\sum_{i>i_t} \frac{v_{i,t}}{\sigma^2} \rp. \label{eq:temp-lower-bound-1} 
            \end{align}
            Combining~\eqref{eq:temp-upper-bound-1} and~\eqref{eq:temp-lower-bound-1}, and on some further simplification, we get 
            \begin{align}
                \mutilde_t = \frac{A}{B} &\leq \frac{2 i_t^{3/2}}{t - i_t} \lp 1 + o(1) \rp +  \frac{1}{t-i_t} \sum_{i>i_t}X_i \lp 1 +  \frac{v_{i,t}}{\sigma^2} \rp \lp 1 - \frac{1}{2(t-i_t)} \sum_{j > i_t} \frac{v_{j,t}}{\sigma^2} \rp \\ 
                & = o\lp \frac{1}{\sqrt{t}} \rp + \frac{1}{t-i_t} \sum_{i=i_t+1}^{t} X_i + \mc{O}\lp \frac{1}{t - i_t} \sum_{i= i_t + 1}^t \frac{v_{i,t}}{\sigma^2} \rp \\
                & =  \muhat_t + \mc{O}\lp \frac{1}{t} \sum_{i= i_t + 1}^t \frac{v_{i,t}}{\sigma^2} \rp = \muhat_t + \mc{O}\lp \frac{1}{t} \sum_{i= i_t + 1}^t \frac{1}{\sigma^2} \sqrt{\frac{\log t}{i}} \rp \\
               &  =  \muhat_t + \mc{O}\lp \sqrt{ \frac{\log t}{t}} \rp. 
            \end{align}
           Using same approximations, we can also obtain a similar lower bound on $\mutilde_t$ under $E_{t,3}$. 
        \subsection{Proof of~Lemma~\ref{lemma:betting-CI-3}}
        \label{proof:betting-CI-3}
            \paragraph{$\boldsymbol{N}$ is finite almost surely.} This is an easy consequence of~\Cref{lemma:betting-CI-1}. In particular, recall that~\Cref{lemma:betting-CI-1} shows that the width of~$\CIfan$, denoted by $w_n$, satisfies 
            \begin{align}
                \lim_{n \to \infty} \sqrt{n} \times |\CIfan| = \lim_{n \to \infty} \sqrt{n} \times  w_n \eqas  2c_1. 
            \end{align}
            The exact expression of $w_n$ is stated in~\eqref{eq:width-of-Fan-CI}. 
            Now, let $\Omega$ denote the probability $1$ event in which the limiting width of $\CIfan$ is $c_1$. Then, note that 
            \begin{align}
                \Omega & \subset \cap_{a \in \mathbb{N}} \lbr w_n \leq \frac{2c_1}{\sqrt{n}}\lp 1 + \frac{1}{a} \rp \text{ all but finitely many times} \rbr \\
                &= \cap_{a \in \mathbb{N}}\lp  \cup_{m \in \mathbb{N}} \cap_{n \geq m} \lbr w_n \leq \frac{2c_1}{\sqrt{n}}\lp 1 + \frac{1}{a} \rp \rbr \rp  \\ 
                & \subset \cup_{m \in \mathbb{N}} \cap_{n \geq m} \lbr w_n \leq \frac{{4}c_1}{\sqrt{n}} \rbr = \{N < \infty \}, \label{eq:N-finite-1}
            \end{align}
            where the second inclusion follows by choosing $a=1$.  Since $\mathbb{P}(\Omega)=1$,~\eqref{eq:N-finite-1} implies that $N$ is finite almost surely. Also for all $n \geq N$, $w_n$ is smaller than $4c_1/\sqrt{n}$, which implies the required inclusion: 
            \begin{align}
                \CIbet \subset \CIfan \subset \CItilde  = \lb \mutilde_n \pm \frac{2c_1}{\sqrt{n}} \rb\quad \text{under } \{n \geq N\}. 
            \end{align}

            \paragraph{The approximation result.} For any $m \in [0,1]$, observe that 
            \begin{align}
                (X_t - m)^2 &=  (X_t-\muhat_{t-1})^2 + (\muhat_{t-1} - m)^2 + 2 (\muhat_{t-1} - m)(X_t - \muhat_{t-1})  \\
                & \stackrel{(i)}{\leq} (X_t-\muhat_{t-1})^2 + |\muhat_{t-1}-m| + 2|\muhat_{t-1}-m| \\
                & = (X_t-\muhat_{t-1})^2 + 3|\muhat_{t-1} - \mu + \mu - \mutilde_n + \mutilde_n - m| \\
                & \stackrel{(ii)}{\leq} (X_t-\muhat_{t-1})^2 + 3 \lp |\muhat_{t-1}-\mu| +  |\mu - \mutilde_n| + |\mutilde_n - m|\rp. 
            \end{align}
            The inequality (i) uses the fact that $x^2 \leq |x|$ for $x \in (-1,1)$, and that $|X_t - \muhat_{t-1}|\leq 1$, while (ii) follows from two applications of the triangle inequality after adding and subtracting $\mu$, and $\mutilde_n$.

            Since under the event $E_{t-1}$,  $|\muhat_{t-1} -  \mu|$ is $\mc{O}(\sqrt{\log t/t})$, we obtain
            \begin{align}
                (X_t - m)^2 \boldsymbol{1}_{E_{t-1}} \leq (X_t-\muhat_{t-1})^2 + 3 \lp  | \mutilde_n - m| + |\mutilde_n - \mu| \rp + \mc{O}(\sqrt{\log t/ t}). 
            \end{align}

            Next, we will restrict our attention to a smaller subset of values of $m$ lying in $\CItilde$. 
            Under the event $\{n \geq N\}$, we know that $w_n \leq 2c_1/\sqrt{n}$, which implies that for all $m \in \widetilde{C}_n$, we have $|\mutilde_n - m|\boldsymbol{1}_{n \geq N} \leq 4c_1/\sqrt{n}$. 

            Thus, with $I_{t,n} \defined \boldsymbol{1}_{n \geq N} \times \boldsymbol{1}_{E_{t-1}}$, we have 
            \begin{align}
                \sum_{t=1}^n\psi_{E}(\lambda_{t,n}) (X_t-m)^2\, I_{t,n} &\leq  \sum_{t=1}^n\psi_{E}(\lambda_{t,n}) (X_t-\muhat_{t-1})^2 +  \lp \frac{4c_1}{\sqrt{n}}+ |\mutilde_n - \mu|  \rp \sum_{t=1}^n\psi_{E}(\lambda_{t,n}) I_{t,n}  \\
                & \quad + \mc{O}\lp \sum_{t=1}^n\psi_{E}(\lambda_{t,n})  \sqrt{ \frac{\log t}{t}} I_{t,n} \rp. 
            \end{align}
            We will show that the last two terms converge to $0$ almost surely, to complete the proof. We do this, by observing the following three facts: 
            \begin{itemize}
                \item $|\mutilde_n - \mu| \convas 0$. This is true because from the proof of~\Cref{lemma:betting-CI-2}, we know that the probability  $\{ |\mutilde_n - \mu| > \mc{O}(\sqrt{\log n / n} ) \text{ infinitely often} \}$ is equal to $0$, which implies that the probability that $\{\mutilde_n \to \mu\}$ equals one, as required. 
                \item The term $\sum_{t=1}^n \psi_E(\lambda_{t,n})$ converges almost surely to $\log(2/\alpha)/4 \sigma^2$. This was proved by~\citet{waudby2023estimating}, and it implies that $\lp 4c_1/\sqrt{n} + |\mutilde_n - \mu| \rp \sum_{t=1}^n \psi_E(\lambda_{t,n}) I_{t,n} \convas 0$.
                \item Finally, the term $\sum_{t=1}^n \psi_E(\lambda_{t,n}) \sqrt{\log t / t}$ converges almost surely to $0$. This is a simple consequence of the previous statement that $\sum_{t=1}^n \psi_E(\lambda_{t,n}) \convas \log(2/\alpha)/4\sigma^2$. 
            \end{itemize}

            Together, these three facts imply the required result. \hfill \qedsymbol

\section{Details of the non-asymptotic analysis of betting CI~(Section~\ref{sec:nonasymp-CI})}
\label{appendix:impossibility}

    \subsection{Proof of CI lower bound~(Proposition~\ref{prop:lower-bound-1})}
    \label{proof:lower-bound-1}
        The proof of this result proceeds in the following steps, as we outlined in~\Cref{subsec:lower-bound-ci}: 
        \begin{itemize}
            \item First, using the method $\mc{C}$ for constructing valid CIs, we design a hypothesis test for distinguishing between two distributions with means separated by $w(n, P, \alpha)$~(the width of the CI constructed by  the method $\mc{C}$ using $n$ \iid observations from the distribution $P$). We show that this test controls both type-I and type-II errors at the level $\alpha$. 
            \item Next, we show that the tight control over the type-I and type-II errors simultaneously implies a fundamental lower bound on the KL divergence between the two distributions. This follows essentially by an application of the data processing inequality for relative entropy. 
            
            \item Finally, we translate this into a statement about the information projections, and obtain the required result by inverting this bound on $\klinf$ and $\klinfminus$. 
        \end{itemize}

        \paragraph{Step 1: Design a hypothesis test using $\boldsymbol{\mc{C}}$.}
        Given observations $X_1, X_2, \ldots \simiid P^*$, we consider the hypothesis testing problem with $H_0: P^* = P$ and $H_1: P^* = Q$, with $\mu_Q > \mu_P + w(n, P, \alpha)$. For this problem, we can define a test $\Psi:\mc{X}^n \to \{0,1\}$, as follows: 
        \begin{align}
            \Psi(X^n) = \begin{cases}
                1, & \text{ if } \mu_Q \in C_n = \mc{C}(X^n, \alpha), \\
                0, & \text{ otherwise}. 
            \end{cases}
        \end{align}
        Here $\Psi(X^n)=1$ indicates that the test rejects the null. 
        We now claim that under the stated assumption that $\mu_Q - \mu_P >  w(n, P, \alpha)$, the test $\Psi$ controls both, its type-I error and type-II error, at level $\alpha$. This is because, if $H_0$ is true, then $\mathbb{P}(\Psi = 1) = \mathbb{P}(\mu_Q \in C_n) \leq  \mathbb{P}(\mu_P \not \in C_n) \leq \alpha$.  Note that the inequality uses the assumption that $\mu_Q - \mu_P >  w_n(n, P, \alpha)$, to infer that both $\mu_P$ and $\mu_Q$ cannot simultaneously lie in $C_n$. 
        Similarly, if $H_1$ is true, then $\mathbb{P}(\Psi=0) = \mathbb{P}\lp \mu_Q \not \in C_n \rp  \leq \alpha$, by definition of a level-$(1-\alpha)$ CI, and the separation assumption of $\mu_P$ and $\mu_Q$. 
    
        \paragraph{Step 2: Lower bound on $\boldsymbol{\dkl(P, Q)}$.} Now, we observe the following: 
        \begin{align}
            n \dkl\lp P, Q \rp = \dkl\lp P^{\otimes n}, Q^{\otimes n}  \rp \geq \dkl \lp \mathbb{E}_{P^{\otimes n}}[1-\Psi],\; \mathbb{E}_{Q^{\otimes n}}[1-\Psi] \rp. \label{eq:lower-bound-1-1}
        \end{align}
        The equality above is due to the standard chain rule for KL divergence, while the inequality follows from an application of the data processing inequality for KL divergence. Note that the last term in the display above denotes the KL divergence between two Bernoulli random variables. 
        Now, for any $p, q \in [0,1]$, we have 
        \begin{align}
            \dkl(p, q) &= p \log(1/q) + (1-p)\log(1/(1-q)) + \lp p \log p + (1-p) \log (1-p) \rp \\
            & \geq p \log(1/q) + 0 - h(p). 
        \end{align}
        Here, $h(p)$ denotes the entropy (in nats) of the Bernoulli distribution with parameter $p$. Using the above inequality in~\eqref{eq:lower-bound-1-1}, and noting that $p \equiv \mathbb{E}_{Q^{\otimes n}}[\Psi] \geq 1-\alpha$, and $q \equiv \mathbb{E}_{P^{\otimes n}}[\Psi] \leq \alpha$, we obtain the following:  
        \begin{align}
             \dkl(P, Q) \geq a(n, \alpha) \defined \frac{(1-\alpha) \log(1/\alpha) + \alpha \log(\alpha) + (1-\alpha) \log(1-\alpha)}{n} = \frac{\log\lp (1-\alpha)^{1-\alpha} \alpha^{2\alpha-1} \rp}{n}.  \label{eq:lower-bound-1-2}
        \end{align}
        To get the above inequality, we used the fact that for $\alpha < 0.5$, the binary entropy $h(p)$ is no larger than $ h(1-\alpha)$, since we have  $p \geq 1-\alpha$. 
        
        \Cref{eq:lower-bound-1-2} says that the fact that we could construct a test to distinguish between $P$ and $Q$, with simultaneous control over the type-I and type-II errors, implies that the distributions $P$ and $Q$ must be sufficiently well-separated in terms of KL divergence. 
        
        \paragraph{Step 3: Connections to $\boldsymbol{\klinf}$.}
        Next, recall that for any $w \geq 0$, we can write
        \begin{align}
            \klinf(P, \mu_P+w, \mc{P}_0) \defined \inf_{Q' \in \mc{P}^+_{0, \mu_P+w}}\, \dkl(P, Q'),
        \end{align} 
        where $\mc{P}^+_{0, \mu_P + w}$  represents the class of distributions $\{Q' \in \mc{P}_0: \mu_Q \geq \mu_P+w\}$. Since it is known that $\klinf(P, \mu_P+w, \mc{P}_0)$ is continuous in $w$~\citep[Appendix F, Theorem 4]{jourdan2022top},  we can also take the infimum over the set of all distributions $Q'$, whose mean is strictly greater than $\mu_P+w$. 
        This fact, along with the inequality in~\eqref{eq:lower-bound-1-2} leads us to the conclusion: 
        \begin{align}
            \klinf\lp P, \mu_P + w(n, P, \alpha), \mc{P} \rp \geq a(n, \alpha). \label{eq:lower-bound-1-3}
        \end{align}
        
        To complete the proof,  we introduce the term 
        \begin{align}
            w^{*, +} \equiv w^{*, +}(n, P, \alpha) & = \klinf(P, \cdot, \mc{P}_0)^{-1} \lp a(n, \alpha) \rp - \mu_P. \label{eq:lower-bound-1-4}
        \end{align}
        Recall, that the inverse information projection is defined as 
        \begin{align}
            \klinf(P, \cdot, \mc{P}_0)^{-1}(x) &\defined \inf \lbr m \geq \mu_P: \klinf(P, m, \mc{P}_0) \geq x \rbr.  
        \end{align}
        
        Because of the continuity of $\klinf$, we note that the above infimum is achieved at some some $w^{*, +}$, which combined with~\eqref{eq:lower-bound-1-3}, gives us 
        \begin{align}
            \klinf\big( P, \mu_P +  w(n, P, \alpha), \mc{P}_0 \big) \geq \klinf \lp P, \mu_P + w^{*, +}, \mc{P}_0 \rp. 
        \end{align}
        Finally, the fact that $\klinf$ is monotonically increasing in the second argument implies the required $w(n, P, \alpha) \geq w^{*, +}(n, P, \alpha)$ completing the first part of the proof.  

        Repeating the exact same argument, now with $Q$ such that $\mu_Q < \mu_P - w(n, P, \alpha)$, we get the conclusion that $w(n, P, \alpha) \geq w^{*, -} \equiv w^{*, -}(n, P, \alpha)$, defined as 
        \begin{align}
            w^{*, -} \equiv w^{*, -}(n, P, \alpha) & = \mu_P -  \klinfminus(P, \cdot, \mc{P}_0)^{-1} \lp a(n, \alpha) \rp. \label{eq:lower-bound-1-5}
        \end{align}
        Combining these two results, we get that $w(n, P, \alpha) \geq \max\{w^{*, +}, w^{*, -} \}$, which completes the proof. 
    
            \subsection{Proof of CI upper bound~(Theorem~\ref{prop:betting-ci-width})}
        \label{proof:betting-ci-width}

            We proceed in the following steps: 
            \begin{itemize}
                \item First, we note that the wealth $W_n(m)$ after $n$ rounds for a betting strategy with logarithmic regret~(such as the mixture method of \Cref{def:mixture-method}) grows exponentially, with the exponent approximately equal to the larger of the two terms: $\hatKplus(P^*, m)$, and $\hatKminus(P^*, m)$ introduced below in~\eqref{eq:empirical-info-projection}. 
                
                \item Next, we obtain a lower bound on the wealth values in terms of the population analogs of the two information projection terms. This step is broken down into several smaller steps. In particular, we first obtain a more coarse CI with deterministic width that contains the betting CI~(\Cref{lemma:coarse-betting-CI}). Then, in the next two lemmas~(\Cref{lemma:small-lambda-star} and~\Cref{lemma:klinf-concentration}), we obtain the required result about the concentration of the empirical information projection about its population value. 
                
                \item The final expression for the deterministic envelope for the betting CI is obtained by inverting  the wealth process bound so obtained. 
            \end{itemize}
    
            \paragraph{Step 1: Exponential growth of wealth.} First we need to introduce the empirical versions of the two information projection terms, defined in~\eqref{def:klinf}: 
            \begin{align}
                \hatKplus(P^*, m) = \frac{1}{n} \sum_{t=1}^n \log \lp 1 + \lambda^*_+(m)(m-X_t) \rp, \quad \text{and}  \quad 
                \hatKminus(P^*, m) = \frac{1}{n} \sum_{t=1}^n \log \lp 1 + \lambda^*_-(m)(m-X_t) \rp. \label{eq:empirical-info-projection}
            \end{align}
            Here $\lambda_+^*$ and $\lambda_-^{*}$ denote the optimal betting fractions that define $\klinf$ and $\klinfminus$ respectively. Observe that for $X_1, \ldots, X_n$ drawn \iid from $P^*$, the term $\widehat{K}_n^a(P^*, m)$ is an unbiased empirical estimate of $\text{KL}_{\inf}^a(P^*, m)$ for $a \in \{+, -\}$. Furthermore, we note here, that the two terms in the display above are different from~$\klinf(\Phat_n, m)$ and $\klinfminus(\Phat_n, m)$. 
            
            Next, recall that the  regret incurred by the betting strategy (for all $m \in [0,1]$) is upper bounded by $c \log(n)$ for some fixed $c<2$. Then, we have the following for any $m \in [0,1]$:
            \begin{align}
                W_n(m) &= \prod_{t=1}^n \lp 1 + \lambda_t(m)(X_t-m) \rp = \exp \lp \sum_{t=1}^n \log\lp 1 + \lambda_t(m)(X_t-m) \rp \rp \\
                & \geq \exp \lp \sup_{\lambda \in \lb \frac{-1}{1-m}, \frac{1}{m} \rb} \sum_{t=1}^n \log\lp 1 + \lambda(X_t-m) \rp - c\log(n) \rp \\ 
                & \geq \exp \lp n \times \max \lbr \klinf(\Phat_n, m),\; \klinfminus(\Phat_n, m) \rbr - c \log(n) \rp \label{eq:ci-upper-bound-proof-0}\\
                & \geq \exp \lp n \times \max \lbr \hatKplus(P^*, m),\; \hatKminus(P^*, m) \rbr - c \log(n) \rp. 
            \end{align}
            The first inequality in the display above follows due to the logarithmic regret incurred by the mixture betting strategy, while the second inequality uses the definition of $\klinf$ and $\klinfminus$ for the empirical distribution $\Phat_n = (1/n) \sum_{i=1}^n \delta_{X_i}$~(with $\delta_x$ denoting the atomic distribution placing all the mass at $x$). 
            The third inequality uses the fact that the supremum is lower bounded by the value of the first term in the exponent at the  specific values $\lambda^*_+(m)$ and $\lambda^*_-(m)$. Thus, we have proved that the wealth process at any $m \in [0,1]$ grows exponentially, with the growth rate determined by the maximum of the two empirical information projection terms. In particular, we get~\eqref{eq:ci-random-upper-bound} by inverting the bound stated above in~\eqref{eq:ci-upper-bound-proof-0}.

            \paragraph{Step 2: A wider CI.}  As in the proof of~\Cref{prop:betting-CI-vs-EB-CI}, we will show that the betting CI at any time $n$ is also contained in a larger CI, that we denote by~$\CItilde$, whose width is of the order $\mc{O}\lp \sqrt{\log n / n}\rp$. While this bound, by itself, is quite loose, it plays a crucial role in proving the stronger $\klinf$-dependent bound, by allowing us to focus our analysis on a small band of $m$ values. 
            \begin{lemma}
                \label{lemma:coarse-betting-CI}
                The betting CI based on $n$ observations is contained inside a larger CI, denoted by~$\CItilde$, whose width satisfies 
                \begin{align}
                    \frac{|\CItilde|}{2} = \sigmahat_n \sqrt{\frac{2\log(3n^c/\alpha)}{n}  }  \leq  \sqrt{\frac{2\log(3n^c/\alpha)}{n}}. 
                \end{align}
                In the above expression,  $c<2$ is from the regret of the mixture betting scheme, and $\sigmahat_n^2 = (1/n) \sum_{i=1}^n (X_i - \muhat_n)^2$ is the empirical variance based on the  $n$ \iid observations.   In other words, the event $E_1 = \{\mu \in [\muhat_n \pm b_n]\}$, where $b_n = 2\sqrt{2\log(3n^c/\alpha)/n}$, occurs with probability at least $1-\alpha/3$. 
            \end{lemma}
            The proof of this statement is in~\Cref{proof:coarse-betting-CI}.

            \paragraph{Step 2: Concentration of information projection.} In this step, we now replace the empirical information projections with their appropriate population counterparts. To do this we need a high probability concentration bound on the deviation of $\hatKplus$ and $\hatKminus$ from their expectations. We will show this in two steps, stated as~\Cref{lemma:small-lambda-star} and~\Cref{lemma:klinf-concentration}. 
            \begin{lemma}
                \label{lemma:small-lambda-star}
                Fix a $\delta \in (0,0.1]$, and let $m>\mu$ be small enough to ensure that $|\lambda_+^*(m)| \leq \delta$. Then, we have 
                \begin{align}
                    \frac{ |\Delta_m|(1-\delta)^2}{\sigma^2 + \Delta_m^2} \leq |\lambda^*_+(m)| \leq \frac{ |\Delta_m|(1+\delta)^2}{\sigma^2 + \Delta_m^2}, \quad \text{where} \quad \Delta_m \defined \mu - m. 
                \end{align}
                A exactly analogous statement is also true for $\lambda^*_{-}(m)$ with $m<\mu$. 
            \end{lemma}
            In other words, the above lemma~(proved in~\Cref{proof:small-lambda-star}) tells us that in the neighborhood of $\mu$, the optimal betting fraction depends almost linearly on $\Delta_m$. This result plays a vital role in our next result, about the concentration of the empirical information projection around its population values. 
            \begin{lemma}
                \label{lemma:klinf-concentration}
                Recall the term $b_n \defined 2 \sqrt{2 \log(3n^c/\alpha)/n}$ from~\Cref{lemma:coarse-betting-CI}, and suppose $n$ is large enough to ensure that for all $m \leq \mu + b_n$, we have $|\lambda^*(m)| \leq \delta$. Then, for any $\beta \in (0,1)$ the following inequality holds for $d(n, \alpha, \beta) \defined \frac{6\sqrt{2\log(3n^c/\alpha) \log(4/\beta)}}{n\sigma^2}$: 
                \begin{align}
                    &\mathbb{P}\lp \hatKplus(P^*, m) \geq \klinf - d(n, \alpha, \beta)   \rp \geq \beta,  \text{ for any } m \in [\mu, \mu + b_n]. 
                \end{align}
                Similarly, we also have 
                \begin{align}
                    &\mathbb{P}\lp \hatKminus(P^*, m) \geq \klinfminus - d(n, \alpha, \beta)   \rp \geq \beta,  \text{ for any } m \in [\mu - b_n, \mu]. 
                \end{align}
            \end{lemma}
           The proof of this statement  uses approximation arguments to show that both $\klinf$ and $\hatKplus$ can be written as their first two Taylor's series terms, plus a negligible remainder term. Hence, to show the concentration of the empirical information projection terms, it suffices to show the concentration of the empirical mean and second moment around their population values. This can be achieved via two applications of Hoeffding's inequality, since all the observations are bounded. The details are in~\Cref{proof:klinf-concentration}. 

            Introduce the term $B = \lceil 2b_n/(1/n^2)\rceil$, and consider the following two girds of equally spaced points: 
            \begin{align}
                \mc{M}_n^+ = \{\mu, \mu + 1/n^2, \ldots, \mu + B/n^2\}, \quad \text{and} \quad 
                \mc{M}_n^- = \{\mu, \mu - 1/n^2, \ldots, \mu - B/n^2\}. 
            \end{align}
            The cardinality of both $\mc{M}_n^+$ and $\mc{M}_n^-$ is equal to $B+1 = \Theta\lp n^{3/2} \sqrt{\log(n/\alpha)}\rp$, which we loosely upper bound by $n^2/8$. 
            Then, using~\Cref{lemma:klinf-concentration},  we can define a probability $1-2\alpha/3$ event $E_2$ as follows: 
            \begin{align}
                &E_2 =  \bigcap_{a \in \{+, -\}}\;  \bigcap_{m_i \in \mc{M}_n} \;\lbr \widehat{K}_n^a(P^*, m_i) \geq  \text{KL}_{\inf}^a\lp P^*, m_i \rp - d(n, \alpha) \rbr, \label{eq:event-E2}\\ 
                \text{where} \quad &d(n,\alpha) \equiv d\lp n, \alpha,\, \frac{2\alpha}{3}\times \frac{8}{2n^2}\rp \defined \frac{6 \sqrt{2 \log(3n^c/\alpha)\log(3n^2/\alpha)}}{n \sigma^2} \leq \frac{9\log(3n^2/\alpha)}{n\sigma^2}. \label{eq:d-n-alpha-def}
            \end{align}

            \paragraph{Step 3: Invert the wealth process.}   Let $E$ denote the event $E_1 \cap E_2$, where $E_1$ was defined in the statement of~\Cref{lemma:coarse-betting-CI}, and $E_2$ is defined in~\eqref{eq:event-E2} above.  Introduce the notation 
            \begin{align}
                K^{\max}(P^*, m) = \max \lbr \klinf(P^*, m),\; \klinfminus(P^*, m) \rbr. 
            \end{align}
            Then, under the event $E$, we can lower bound the wealth at any $m_i \in \mc{M}_n \defined \mc{M}_n^+ \cup \mc{M}_n^-$ as 
            \begin{align}
                W_n(m_i) \geq \exp \lp n K^{\max}(P^*, m_i) - c \log(n) - 9\log(3n^2/\alpha)/\sigma^2 \rp. \label{eq:wealth-lower-bound-1}
            \end{align}
            Under the event $E$, the betting CI at time $n$ can be stated as 
            \begin{align}
                \CIbet = \{m \in [\mu \pm 2b_n]: W_n(m) < 3/\alpha \}. 
            \end{align} 
            We have used $3/\alpha$ here, because we used the remaining $2\alpha/3$ on the event $E_2$. 
            Let $U$ and $L$ denote the end points in the intersection of this CI and the grid $\mc{M}_n $. That is, 
            \begin{align}
              U = \max\{m_i: m_i \in \mc{C}_n \cap \mc{M}_n\}, \quad \text{and} \quad 
             L = \min\{m_i: m_i \in \mc{C}_n \cap \mc{M}_n\}. 
            \end{align}
            Then, since the grid points are uniformly spaced at intervals of length $1/n^2$, it follows that the width of the CI, $2w(n, P^*, \alpha)$ is upper bounded by $U-L+2/n^2$. Using the lower bound on the wealth process stated in~\eqref{eq:wealth-lower-bound-1}, we obtain the following, with $b(n, \alpha) \defined \log(3n^2/\alpha)/n + d(n, \alpha)$~(recall that $d(n, \alpha)$ was defined in~\eqref{eq:d-n-alpha-def}): 
            \begin{align}
                U &= \max \{ m_i \in \mc{M}_n:  W_n(m_i) < \log(3/\alpha)\} \\
                &\leq \max \lbr m_i \in \mc{M}_n: \exp \lp n K^{\max}(P^*, m_i) - 2 \log(n) - nd(n, \alpha) \rp < \log(3/\alpha)\rbr \\
                & = \max \{m_i \in \mc{M}_n: \Kmax(P^*, m_i) < b(n, \alpha)\} \\
                & \leq \sup \{m \in [0,1]: \Kmax(P^*, m) < b(n, \alpha)\} \\
                & = \klinf(P^*, \cdot)^{-1}\lp b(n, \alpha)\rp. 
            \end{align}
            The first inequality in the above display uses the lower bound on $W_n(m_i)$, as stated in~\eqref{eq:wealth-lower-bound-1}, while the second inequality uses the fact that the value of supremum increases by increasing the domain. 
            An analogous argument gives us the lower bound $L \geq \klinfminus(P^*, \cdot)^{-1}\lp b(n, \alpha)\rp$, and together, these two results imply the statement: 
            \begin{align}
                 w(n, P^*, \alpha) &\leq U - L + \frac{2}{n^2} 
                \leq \klinf(P^*, \cdot)^{-1}\lp b(n, \alpha) \rp - \mu + \mu - \klinfminus(P^*, \cdot)^{-1}\lp b(n, \alpha) \rp + \frac{2}{n^2},  
            \end{align}
            which implies the required conclusion 
            \begin{align}
                w(n, P^*, \alpha) \leq 2\max \lbr \klinf(P^*, \cdot)^{-1}\lp b(n, \alpha) \rp - \mu, \;  \mu - \klinfminus(P^*, \cdot)^{-1}\lp b(n, \alpha) \rp \rbr + \frac{2}{n^2}.  
            \end{align}

            \subsubsection{Proof of~\Cref{lemma:coarse-betting-CI}}
            \label{proof:coarse-betting-CI}           
                Consider any $m \in [0,1]$, and observe that the wealth, $W_n(m)$, satisfies the following, due to the $c\log n$ regret of the betting scheme: 
                \begin{align}
                    W_n(m) \geq \exp \lp n \sup_{\lambda \in [-1/(1-m), 1/m]} \frac{1}{n} \sum_{t=1}^n \log(1+\lambda(X_t-m))  \,- c\log n\rp. 
                \end{align}
                Since we can lower bound $\log(1+x)$ with $x-x^2/2$ for all $x \geq 0$, we have 
                \begin{align}
                    W_n(m) &\geq \exp \lp n \sup_{\lambda \in [-1/(1-m), 1/m]}  \lambda \Deltahat_m - \frac{\lambda^2}{2} \lp \sigmahat_n^2 + \Deltahat_m^2\rp  \,- c\log n\rp \label{eq:coarse-betting-CI-proof-1} \\
                    &\geq \exp \lp n \sup_{\lambda \in [-1/(1-m), 1/m]}  \lambda \Deltahat_m - \frac{\lambda^2}{2}  \,- c\log n\rp, \label{eq:coarse-betting-CI-proof-2}
                \end{align}
                where we have used the notation 
                \begin{align}
                    \Deltahat_m \defined \frac{1}{n}\sum_{i=1}^n X_i - m \;=\; \muhat_n - m, \quad \text{and} \quad 
                    \sigmahat_n^2 = \frac{1}{n} \sum_{i=1}^n (X_i - \muhat_n)^2. 
                \end{align}
                On optimizing the RHS of~\eqref{eq:coarse-betting-CI-proof-2}  for $\lambda$, we get that 
                \begin{align}
                    W_n(m) \geq \exp \lp  \frac{n \Deltahat_m^2}{2} - c \log(n) \rp. 
                \end{align}
                
                Using this lower bound, we can construct a larger CI that contains the usual betting CI, as follows: 
                \begin{align}
                    \{m \in [0,1]: \log(W_n(m)) < \log(3/\alpha)\} = \CIbet  \subset \CItilde \defined \lbr m \in [0,1]: \frac{n\Deltahat_m^2}{2} - c \log(n) < \log(3/\alpha) \rbr.  
                \end{align}
                On simplification, we get 
                \begin{align}
                    \CItilde = \lbr m: |m - \muhat_n| \leq \sigmahat_n\sqrt{ \frac{2\log(3n^c/\alpha)}{n} } \rbr.  
                \end{align}

            \subsubsection{Proof of~\Cref{lemma:small-lambda-star}}
            \label{proof:small-lambda-star}
                First, we recall that the mapping $m \mapsto \lambda_+^*(m)$ is continuous. Hence, we can define 
                \begin{align}
                    \epsilon \equiv \epsilon(\delta, P^*) = \sup \{ \epsilon' > 0: \lambda_+^*(m)\leq \delta \text{ for all } m \in [\mu, \mu + \epsilon'] \}. \label{eq:epsilon-def}
                \end{align}
                Thus for all $m \in [\mu, \mu + \epsilon] \cap [0,1]$, the optimal bet satisfies $|\lambda_+^*(m)| \leq \delta$, and  in this proof, we assume $m$ lies in $[\mu, \mu + \epsilon] \cap [0,1]$. 

                Next, we know from~\citet[Theorem 5]{honda2010asymptotically} that for $P^*$ with $\sigma>0$, the optimal betting fraction $\lambda_+^*(m)$ satisfies the following:
                \begin{align}
                    \mathbb{E}\lb \frac{1}{1 + \lambda_+^*(m)(X-m)} \rb = 1, \quad \text{for all } \mu \leq m \leq 1- \frac{1}{\mathbb{E}\lb 1/(1-X) \rb}.  
                \end{align}
                The above condition is obtained by simply calculating the derivative of the objective function in the dual formulation of $\klinf$ by differentiating inside the expectation, and setting it to zero $\lambda = \lambda_+^*(m)$. 

                Since $|X-m|\leq 1$ almost surely, and $|\lambda_+^*(m)| \leq \delta$ by assumption, we will now approximate the above expectation by using a Taylor's expansion of the function $f(x) = 1/(1+x)$ around $0$. In particular, we obtain 
                \begin{align}
                    \frac{1}{1 + \lambda_+^*(m)(X-m)} = 1 - \lambda_+^*(m)(X-m) + \lambda_+^*(m)^2 (X-m)^2 \frac{1}{\big(1 + \delta U \big)^2}, 
                \end{align}
                for some $[-1,1]$ valued random variable $U$. Considering the fact that $1 + \delta U \geq 1- \delta$, we obtain 
                \begin{align}
                     1 - \lambda_+^*(m)\Delta_m + \lambda_+^*(m)^2(\sigma^2 + \Delta_m^2) \frac{1}{(1-\delta)^2} \geq 1, 
                \end{align}
                which on simplification implies 
                \begin{align}
                    |\lambda_+^*(m)| \geq \frac{(1-\delta)^2 |\Delta_m|}{\Delta_m^2 + \sigma^2}. 
                \end{align}
                A similar argument with the other extreme value of $1 + \delta U=1+\delta$, gives us the upper bound 
                \begin{align}
                    |\lambda_+^*(m)| \leq \frac{(1+\delta)^2 |\Delta_m|}{\Delta_m^2 + \sigma^2}. 
                \end{align}
                This completes the proof.

            \subsubsection{Proof of~\Cref{lemma:klinf-concentration}}
            \label{proof:klinf-concentration}
                Recall that, by the dual form of the information projection, we know that 
                \begin{align}
                    \klinf(P^*, m) = \mathbb{E}\lb \log \lp 1 + \lambda^*_+(m)(X-m) \rp \rb. 
                \end{align}
                We approximate the integrand above by its Taylor's expansion around $0$, and obtain 
                \begin{align}
                    \klinf(P^*, m) &= \mathbb{E}\lb \lambda^*_+(m)(X-m) - \frac{\lambda^*_+(m)^2(X-m)^2}{2} + \frac{\lambda^*_+(m)^3 U}{3} \rb, 
                \end{align}
                for some random variable $U$ taking values in $[-1,1]$. On taking the expectation, we get 
                \begin{align}
                    \klinf(P^*, m) = \lambda^*_+(m) \Delta_m - \frac{\lambda^*_+(m)^2}{2} \lp \sigma^2 + \Delta_m^2 \rp + \frac{\lambda^*_+(m)^3}{3} \mathbb{E}[U], \quad \text{where} \; \Delta_m = m - \mu. \label{eq:klinf-conc-proof-1}
                \end{align}

                Now recall that the term $\hatKplus$ is simply the empirical version of $\klinf$, that is, 
                \begin{align}
                    \hatKplus(P^*, m) = \frac{1}{n} \sum_{i=1}^n \log \lp 1 + \lambda^*_+(m)(X_i - m) \rp. 
                \end{align}
                Hence, by the same approximation argument used above for analyzing $\klinf(P^*, m)$, we obtain 
                \begin{align}
                    \hatKplus(P^*, m) = \lambda^*_+(m) \Deltahat_m - \frac{\lambda^*_+(m)^2}{2} \lp \sigmahat_n^2 + \Deltahat_m^2 \rp + \frac{\lambda^*_+(m)^3}{3} \lp \frac{1}{n} \sum_{i=1}^n U_i \rp, \label{eq:klinf-conc-proof-2}
                \end{align}
                where the terms $\Deltahat_m$ and $\sigmahat_n^2$ are defined as
                \begin{align}
                    \Deltahat_m = \muhat_n - m, \quad \text{and} \quad \sigmahat_n^2 = \frac{1}{n} \sum_{i=1}^n (X_i - \muhat_n)^2. 
                \end{align}
                Subtracting~\eqref{eq:klinf-conc-proof-1} from~\eqref{eq:klinf-conc-proof-2}, we get 
                \begin{align}
                    \hatKplus(P^*, m) - \klinf(P^*, m) \geq \lambda^*_+(m) \lp \muhat_n - \mu\rp - \frac{\lambda^*_+(m)^2}{2}\lp \frac{1}{n} \sum_{i=1}^n (X_i-m)^2 - \sigma^2 - \Delta_m^2\rp - \frac{2\lambda^*_+(m)^3}{3}. 
                \end{align}
                Note that to get the above inequality, we have used the fact that $(1/n)\sum_{i=1}^n U_i - \mathbb{E}[U] \geq -2$. 
                Thus, the above inequality implies that to obtain a concentration result for the information projection, it suffices to show the concentration of the mean and second moment of the observations. In particular, for any $\beta \in (0,1)$, we have the following two results by Hoeffding's inequality: 
                \begin{align}
                    &|\muhat_n - \mu| \leq \sqrt{\log(4/\beta)/2n}, \quad \text{w.p. } \geq 1-\beta/2,  \\
                   \text{and} \quad 
                   &\left\lvert \frac{1}{n}\sum_{i=1}^n (X_i-m)^2 - \sigma^2 - \Delta_m^2 \right\rvert \leq \sqrt{\log(4/\beta)/2n}, \quad \text{w.p. } \geq 1-\beta/2. 
                \end{align}
                This implies the statement: 
                \begin{align}
                    &\mathbb{P}\lp \hatKplus(P^*, m) \geq \klinf(P^*, m) - d(n, m, \beta) \rp \geq 1- \beta, \\
                    \text{where} \quad 
                    &d(n, m, \beta) \defined |\lambda^*_+(m)| \sqrt{ \lp \log(4/\beta)\rp / 2n} \,+\, \frac{|\lambda^*_+(m)|^2}{2} \frac{\log(4/\beta)}{2n} + \frac{2 |\lambda^*_+(m)|^3}{3}. 
                \end{align}
                Next, recall that~\Cref{lemma:small-lambda-star} implies that$|\lambda^*_+(m)|  \leq \lp |\Delta_m|(1+\delta)^2\rp/\lp \Delta_m^2 + \sigma^2\rp$. If $\delta \leq 0.1$, then it is easy to check that $(1+\delta)^2 \leq (1+\delta)^4\leq (1+\delta)^6 \leq 2$. This implies that 
                \begin{align}
                    d(n, m, \beta) \leq \frac{2 |\Delta_m| \sqrt{\log(4/\beta)}}{\sigma^2\sqrt{2n}} + \frac{\Delta_m^2 \log(4/\beta)}{n \sigma^4} + \frac{8|\Delta_m|^3}{3 \sigma^6}.  \label{eq:klinf-conc-proof-3}
                \end{align}
                Finally, note that $|\Delta_m|\leq 2\sqrt{2 \log(3n^c/\alpha)/n}$, which implies that $d(n, m, \beta) = \mc{O}\lp \sqrt{\log(n/\alpha)\log(4/\beta)}/n \rp$. In particular, if $n$ is large enough to ensure that 
                the second and third terms in the RHS of~\eqref{eq:klinf-conc-proof-3} are smaller than the first~(see precise conditions in~\Cref{appendix:n-large-enough}), we can upper bound the RHS of~\eqref{eq:klinf-conc-proof-3} with three times the first term: 
                \begin{align}
                    d(n, m, \beta) \leq \frac{6\sqrt{2\log(3n^c/\alpha) \log(4/\beta)}}{n\sigma^2}. 
                \end{align}
                The proof for the concentration of $\hatKminus$ is obtained under the same, $1-\beta$ probability event, following an exactly analogous approximation argument as above. We omit the details to reduce repetition.       

            \subsubsection{Meaning of ``$\boldsymbol{n}$ large enough'' in~\Cref{prop:betting-ci-width}}
            \label{appendix:n-large-enough}
                To simplify the final statement of~\Cref{prop:betting-ci-width}, we have assumed that $n$ is large enough at several points in the proof. We collect all those assumptions here. 
                \begin{itemize}
                    \item In~\Cref{lemma:klinf-concentration}, we assumed that $n$ was large enough to ensure the following for a fixed $\delta \in (0,0.1]$:  
                    \begin{align}
                        b_n = 2\sqrt{\frac{2 \log(3n^c/\alpha)}{n}} < \epsilon(\delta, P^*) \defined \sup \{ \epsilon'>0: |\lambda^*_+(m)| \leq \delta, \text{ for all } m \in [\mu, \epsilon']\}. 
                    \end{align}
                    Thus, for a fixed $\delta$~(say $0.1$), we know that there exists a finite $\epsilon \equiv \epsilon(\delta, P^*)>0$, and the above condition requires $n$ to be $\Omega \lp \frac{\log(1/\epsilon)}{\epsilon^2} \rp$. 
                    \item Next, in~\Cref{lemma:klinf-concentration}, we assumed that $n$ is large enough to satisfy 
                    \begin{align}
                        \min \lbr \frac{n\sigma^2}{\sqrt{\log(3n^c/\alpha) \log(3n^2/4\alpha)}}, \;
                        \frac{3\sigma^4\sqrt{n \log(3n^2/4\alpha)}}{32\sqrt{2}\log(3n^c/\alpha)}  
                        \rbr \geq 1, 
                    \end{align}
                    The above assumption on $n$ is made purely to absorb the two higher order terms into the first  term in~\eqref{eq:klinf-conc-proof-3}, for the value of $\beta = 8\alpha/3n^2$. 
                    \item Then, we assumed that $n$ is large enough to ensure 
                        $4b_nn^2 +2 = 8n^{3/2} \sqrt{2 \log(3n^c/\alpha)}  + 2\leq n^2/8$. 

                 \end{itemize}
            We emphasize that, other than the first bullet point, the other assumptions  on $n$ are made to simplify the expression of the final statement of~\Cref{prop:betting-ci-width}, by absorbing the higher order terms. Also note that the random upper bound on the width of the betting CI, in terms of the empirical information projections, is valid for all $n \geq 1$.

\section{Details of the non-asymptotic analysis of betting CS~(Section~\ref{sec:nonasymp-CS})}
    \subsection{Proof of the CS lower bound~(Proposition~\ref{prop:lower-bound-2})}
    \label{proof:lower-bound-2}
        The starting point of the proof of this result is to use the duality between confidence sequences and sequential tests. In particular, given observations $X_1, X_2, \ldots \simiid P^*$, we consider the hypothesis testing problem with 
        \begin{align}
            H_0: P^*=Q, \quad \text{versus} \quad H_1: P^*=P.
        \end{align}
         We assume that the means of the two distributions~($\mu_P$ and $\mu_Q$) satisfy the constraint $\mu_Q-\mu_P = w_0 > w_e(n, P, \alpha)$.  We now proceed to the main steps of the proof. 
        
        \paragraph{Step 1: Sequential test.} Unlike the proof of~\Cref{prop:lower-bound-1}, where the sample-size $n$ was fixed beforehand, for the above testing problem, we propose a sequential test that stops and makes it decision at a random time. In particular,  this, we define a sequential test that stops at a random time $T$, and makes a decision $\Psi \in \{0, 1\}$, as follows: 
        \begin{align}
            T = \inf \{n \geq 1: |C_n| < \mu_Q-\mu_P\}, \quad \text{and} \quad 
            \Psi = \boldsymbol{1}_{Q \not \in C_T}. \label{eq:T-psi-def}
        \end{align}
        We now show that this sequential test controls both the type-I and type-II errors at level $\alpha$. 
        \begin{lemma}
            \label{lemma:lower-bound-2-1}
            Under the assumptions of~\Cref{prop:lower-bound-2}, the stopping time $T$ is finite almost surely, and the test $\Psi$ introduced above satisfies the following:
            \begin{align}
                \mathbb{E}_{H_0}[\Psi] \leq \alpha, \quad \text{and} \quad 
                \mathbb{E}_{H_1}[\Psi] \geq 1-\alpha. \label{eq:lower-bound-2-0}
            \end{align}
        \end{lemma}
        The proof of this lemma is in~\Cref{proof:lemma-lower-bound-2-1}.

        \paragraph{Step 2: Lower bound on KL divergence.} In the next step, using the properties of the stopping time $T$ and the test $\Psi$, we show that the KL divergence between $P$ and $Q$ must be at least as large as a term depending on $E_{H_1}[T]$ and the binary KL divergence between the test outcome $\Psi$ under the $H_1$ and $H_0$. 
        \begin{lemma}
            \label{lemma:lower-bound-2-2}
            The KL divergence between $P$ and $Q$, with $Q$ such that $\mu_Q > \mu_P + 2w_e(n, P, \alpha)$,  must satisfy 
            \begin{align}
                \dkl(P, Q) \geq \frac{ \log\lp (1-\alpha)^{1-\alpha} \alpha^{2\alpha-1} \rp }{\mathbb{E}_{H_1}[T]}, \label{eq:lower-bound-2-2}
            \end{align}           
            where $T$ was defined in~\eqref{eq:T-psi-def}.
        \end{lemma}
        The above result is reminiscent of the corresponding fixed-sample size inequality~\eqref{eq:lower-bound-1-2}, used in proving~\Cref{prop:lower-bound-1} in~\Cref{proof:lower-bound-1}.  
       The main difference here is that the expected stopping time under the alternative plays the role of the non-random sample-size $n$  in~\eqref{eq:lower-bound-1-2}.

        \paragraph{Step 3: Lower bound on width.} Now, recall that $\mu_Q - \mu_P = w_0$, where $w_0>w_e(n, P, \alpha)$ by construction. Then, we have for any $w \in (w_e(n, P, \alpha), w_0]$: 
        \begin{align}
            \mathbb{E}_{H_1}[T] \;\leq\; \mathbb{E}_{H_1}[T_{w}(P, \alpha)] \;\leq\;   n.  \label{eq:w-eff-lower-bound-proof-temp}
        \end{align}
        The two inequalities above use the following facts:
        \begin{itemize}
            \item Since $T$ is the first time the width of the CS falls below $w_0 = \mu_Q-\mu_P$, it follows that we have $T \leq T_w$ almost surely for all $w \leq w_0$, immediately implying the first inequality.  
            
            \item Since $w>w_e(n, P, \alpha)$ by assumption, we can always choose a $w'$, such that $w_e(n, P, \alpha) < w' \leq w$.  This means that, by the same argument as in the first step, we also have $\mathbb{E}_{H_1}[T_w(P, \alpha)] \leq \mathbb{E}_{H_1}[T_{w'}] \leq n$.  

            \item Finally,  recall that $w_e(n, P, \alpha)$ is defined as $\inf \{w' \geq 0: \mathbb{E}_{H_1}[T_{w'}] \leq n\}$. Hence, the previous step implies that $\mathbb{E}_{H_1}[T_{w'}] \leq n$ for that choice of $w'$, which implies the second inequality in~\eqref{eq:w-eff-lower-bound-proof-temp}. 
        \end{itemize}
        Plugging the inequality from the previous display into~\eqref{eq:lower-bound-2-2}, and taking the infimum over all $Q \in \mc{P}_0^+\lp \mu_P + w_0\rp = \{Q' \in \mc{P}_0: \mu_{Q'} \geq \mu_P + w_0\}$, we get 
        \begin{align}
            \klinf(P, \mu_P + w_0) = \inf_{Q \in \mc{P}_0^+\lp \mu_P + w_0\rp} \; \dkl(P, Q) \geq \frac{\log\lp (1-\alpha)^{1-\alpha} \alpha^{2\alpha-1} \rp}{n} \defined a(n, \alpha). 
        \end{align}
        Next, we introduce the term $w^*(n, P, \alpha) = \inf \{w \geq 0: \klinf(P, \mu_P + 2w ) \geq a(n,\alpha) \}$, and use the above inequality, and the monotonicity of the information projection for a fixed $P$, to conclude that 
        \begin{align}
            w_0 \geq w^*(n, P, \alpha). 
        \end{align}
        The final step is to note that $w_0$ is an arbitrary value larger than $w_e(n, P, \alpha)$, and thus we have 
        \begin{align}
            w_e(n, P, \alpha) = \inf \{w_0: w_0> w_e(n, P, \alpha)\}     \geq w^*(n, P, \alpha). 
        \end{align}
        This concludes the first part of the proof that gets the lower bound on $w_e(n, P, \alpha)$ in terms of $\klinf$. By repeating the exact same steps, but with the distribution $Q$ such that $\mu_Q = \mu_P - w_0$, for any $w_0> w_e(n, P, \alpha)$, we get the other term in terms of $\klinfminus$. We omit the details to avoid repetition. 

        \subsubsection{Proof of~\Cref{lemma:lower-bound-2-1}}
        \label{proof:lemma-lower-bound-2-1}
            This result is proved by following the definitions of the terms. 
           
            \paragraph{Type-I error probability.} We proceed as follows: 
            \begin{align}
                \mathbb{E}_{H_0}[\Psi] &= \mathbb{P}_{H_0} \lp \mu_Q \not \in C_T \rp  
                = \sum_{n \in \mathbb{N}} \mathbb{P}_{H_0}\lp T=n, \mu_Q \not \in C_n \rp
            \end{align}
            In the last equality, we used the fact that the sets $\{E_n: n \geq 1\}$ with $E_n \defined \{T=n\}$ form a disjoint partition of the sample space (since $T<\infty$ almost surely under the conditions of~\Cref{prop:lower-bound-2}). Furthermore, for any $n \geq 1$, we have $\{\mu_Q \not \in C_n\} \subset \mc{E} \defined \{ \exists i \in \mathbb{N}: \mu_Q \not \in C_i\}$. Hence, we have $\mathbb{P}_{H_0}\lp T=n, \mu_Q \not \in C_n\rp \leq \mathbb{P}_{H_0}\lp E_n \cap \mc{E}\rp$, which implies 
            \begin{align}
                \mathbb{E}_{H_0}[\Psi] & \leq \sum_{n\in \mathbb{N}} \mathbb{P}_{H_0}\lp E_n \cap \mc{E} \rp = \mathbb{P}\lp \mc{E} \rp 
                 = \mathbb{P}\lp \exists n \in \mathbb{N}: \mu_Q \not \in C_n \rp \leq \alpha. 
            \end{align}
       
            \paragraph{True detection probability.} To obtain the lower bound on the true detection rate, we proceed as follows: 
            \begin{align}
                \mathbb{E}_{H_1}[\Psi] &= \mathbb{P}_{H_1}\lp \mu_Q \not\in C_T \rp  = \mathbb{P}_{H_1}\lp \mu_Q \not \in C_T, \mu_P \in C_T \rp + \mathbb{P}_{H_1}\lp \mu_Q \not \in C_T, \mu_P \not \in C_T \rp \\
                &\geq 0 +  \mathbb{P}_{H_1}\lp \mu_Q\not \in C_T, \mu_P \in C_T \rp \\
                & \stackrel{(i)}{=} \mathbb{P}_{H_1}\lp \mu_P \in C_T \rp   = 1 -\mathbb{P}_{H_1}\lp \mu_P \not \in C_T \rp   
            \end{align}
            The equality (i) above simply uses the fact that since the width of the CS at $T$ is smaller than $\mu_Q - \mu_P = w_0$, if $\mu_P \in C_T$, it automatically means that $\mu_Q \not \in C_T$. Finally, we note that 
            \begin{align}
                \mathbb{P}_{H_1}\lp \mu_P \not \in C_T \rp &= \sum_{n \in \mathbb{N}} \mathbb{P}_{H_1}\lp T=n, \mu_P \not \in C_n \rp \\
                &\leq \mathbb{P}_{H_1}\lp \exists n \in \mathbb{N}: \mu_P \not \in C_n \rp \leq \alpha. 
            \end{align}
            This completes the proof. 
        
        \subsubsection{Proof of~\Cref{lemma:lower-bound-2-2}}
        \label{proof:lemma-lower-bound-2-2}
            We start by noting that the KL divergence between the $T$-fold product of $P$ and $Q$ can be written as the product the $\mathbb{E}_{H_1}[T]$ and $\dkl(P, Q)$. In particular, note that 
            \begin{align}
                \dkl\lp  P^{\otimes T}, Q^{\otimes T} \rp  = \mathbb{E}_{H_1} \lb \sum_{t=1}^T \log \lp \frac{dP}{dQ} \rp\rb = 
                \mathbb{E}_{H_1}[\Psi] \mathbb{E}_P\lb \lp \frac{dP}{dQ} \rp \rb 
                = \mathbb{E}_{H_1}[\Psi]\, \dkl(P, Q). 
            \end{align}
            The second equality above follows from an application of Wald's identity~(\Cref{fact:walds-equation}), using the assumption that both $\mathbb{E}_{H_1}[T]$ and $\dkl(P, Q)$ are finite. Next, we use the above inequality to obtain the following: 
            \begin{align}
                \mathbb{E}_{H_1}[T] \dkl(P, Q) =  \dkl\lp P^{\otimes T}, Q^{\otimes T} \rp  \geq \dkl\lp \mathbb{E}_{H_1}[\Psi], \mathbb{E}_{H_0}[\Psi] \rp \geq \log\lp (1-\alpha)^{1-\alpha} \alpha^{2\alpha-1} \rp. \label{eq:lower-bound-2-1}
            \end{align}
            The first inequality above uses the  data-processing inequality for relative entropy with random stopping times, proved by~\citet[Lemma 19]{JMLRkaufman16a}. The second inequality simply uses the bounds on the two probabilities stated in~\Cref{lemma:lower-bound-2-1}, to further lower bound the binary KL divergence. 
            On re-arranging~\eqref{eq:lower-bound-2-1}, we immediately obtain the required lower bound on the KL divergence between $P$ and $Q$: 
            \begin{align}
                \dkl(P, Q) \geq \frac{\log\lp (1-\alpha)^{1-\alpha} \alpha^{2\alpha-1} \rp}{\mathbb{E}_{H_1}[T]}. 
            \end{align}

        \subsection{Proof of the CS upper bound~(Theorem~\ref{prop:betting-cs-width})}
        \label{proof:betting-cs-width}
            \paragraph{Step 1: Connection to optimal wealth process.} By assumption, we know that the regret incurred by the mixture betting strategy satisfies 
            \begin{align}
                \sup_{m \in \text{supp}(P^*)} \mc{R}_n(m) \leq  c \log(n), 
            \end{align}
            for some $c<2$. This implies that for any $m$, 
            \begin{align}
               \log\lp W_n(m) \rp   &\geq \sup_{\lambda \in \lb \frac{-1}{1-m}, \frac{1}{m} \rb}  \sum_{i=1}^n \log \lp 1 + \lambda \lp  X_i -m \rp \rp - c \log n  \\ 
               & \geq \max_{\lambda \in \{\lambda_+^*, \lambda_-^*\}} \sum_{i=1}^n \log \lp 1 + \lambda \lp X_i - m \rp \rp - c \log n, 
            \end{align}
            where $\lambda_a^*$ for $a \in \{+, -\}$ were introduced in~\eqref{eq:lambda-def-dual}, and denote the optimal betting fractions used to define $\klinf$ and $\klinfminus$ respectively. For any $m  \in [0, 1]$ and $a \in \{-, +\}$, introduce the notation $W_n^{*, a}$, to denote the oracle wealth process corresponding to a betting strategy that always plays the bet $\lambda_a^*$. 
            Next, we introduce the following stopping times, based on the above lower bound on the wealth process: 
            \begin{align}
                &\tau_w^+ = \inf \{n \geq 1:  \log \lp W_n^{*,+}(\mu+w/2) \rp - c \log n \geq \log(1/\alpha)\},  \label{eq:tau-w-plus-def} \\
                \text{and} \quad 
                &\tau_w^- = \inf \{n \geq 1:  \log \lp W_n^{*, -}(\mu-w/2) \rp - c \log n \geq \log(1/\alpha)\}. 
            \end{align}
            Our first key observation is that $T_w$ can be bounded by the two `oracle' stopping times introduced above. 
            \begin{lemma}
                \label{lemma:oracle-stopping-time}
                For any $w>0$, $T_w = \inf \{n \geq 1: |C_n| \leq w \}$,  used to define the effective width, satisfies  
                \begin{align}
                    T_w \leq \tau_w^+ \vee \tau_w^- \defined \max \{ \tau_w^+, \tau_w^-\}. 
                \end{align}
                Hence, the expected value of $T_w$ satisfies 
                \begin{align}
                    \mathbb{E}[T_w] \leq \mathbb{E}[\tau_w^+ \vee \tau_w^-] \leq \mathbb{E}[\tau_w^+  + \tau_w^-] \leq 2 \max_{a \in \{+, -\} } \mathbb{E}[\tau_w^a].  \label{eq:Tw-upper-bound} 
                \end{align}
            \end{lemma}
            The proof of this statement is in~\Cref{proof:oracle-stopping-time}. 
            
            In the rest of this section, we  will focus on deriving an upper bound on the expected value of $\tau_w^+$, for any $w>0$. An exactly analogous argument also gives us an upper bound on the expected value of $\tau_w^-$. Together, these results give us an upper bound on $\mathbb{E}[T_w]$, that will eventually lead to a bound on the effective width.  

            \paragraph{Step 2: Upper bound on the oracle stopping time.} Let $m = \mu + w/2$, and introduce  $Z_i = \log(1 + \lambda^*_+(m)(m-X_i))$, and $S_n = \sum_{i=1}^n Z_i$, with $Z_0 = 0$. Note that $\mathbb{E}[Z_i] = \klinf(P, \mu+w/2)$ for $i \geq 1$, by the definition of $\lambda^*_+(\cdot)$. We now proceed as follows, for an arbitrary $N<\infty$: 
            \begin{align}
                \mathbb{E}\lb S_{\tau_w^+ \wedge N} \rb &= \mathbb{E}\lb \sum_{i=1}^{\tau_w^+ \wedge N} Z_i \rb = 0+ \sum_{i=1}^N \mathbb{E}\lb (S_i - S_{i-1}) \boldsymbol{1}_{\tau_w^+ \geq i} \rb  = 
                \sum_{i=1}^N \mathbb{E}\lb (S_i - S_{i-1}) \boldsymbol{1}_{\tau_w^+ \geq i} \rb \\
                & \stackrel{(i)}{=} \sum_{i=1}^N \mathbb{E}\lb \boldsymbol{1}_{\tau_w^+ \geq i} \mathbb{E}[S_i - S_{i-1}|\mc{F}_{i-1}] \rb  \stackrel{(ii)}{=}\mathbb{E}[\tau_{w}^+ \wedge N] \klinf(P, \mu+w/2), 
            \end{align}
            where $(i)$ follows from  the fact that $\{\tau_w^+ \geq i\} = \{\tau_w^+ \leq i-1\}^c$ is $\mc{F}_{i-1}$-measurable, since $\tau_w^+$ is a stopping time adapted to $(\mc{F}_i)_{i \geq 0}$. The equality $(ii)$ uses the optional stopping theorem for bounded stopping times. 
            Next, by~\Cref{assump:support}, $Z_i \leq C$ almost surely for some constant $C<\infty$~(depending upon the support of $P^*$ in~\Cref{assump:support}),  and thus we can conclude that 
            \begin{align}
                \mathbb{E}[\tau_w^+ \wedge N] &= \frac{\mathbb{E}\lb S_{\tau_w^+ \wedge N} \rb}{\klinf(P^*, \mu+w/2)}  \leq \frac{\mathbb{E}[S_{\tau_w^+ \wedge N -1} + C ]}{\klinf(P^*, \mu+w/2)}  \label{eq:overshoot}\\
                & \leq \frac{\mathbb{E}[\log(1/\alpha) + C + c \log(\tau_w^+ \wedge N)]}{\klinf(P^*, \mu+w/2)}   \\
                & \leq \frac{\log(1/\alpha) + C + c \log( \mathbb{E}[\tau_w^+ \wedge N])}{\klinf(P^*, \mu+w/2)},  
            \end{align}
            where the second inequality uses the fact that, by definition of $\tau_w^+$, the value $S_{\tau_w^+ \wedge N - 1}$ must be smaller than $\log(1/\alpha) + c \log(\tau_w^+)$, and the third inequality uses the concavity of $\log(\cdot)$, along with Jensen's inequality. 

            Finally, note that $N$ is arbitrary, and $\tau_w^+ \wedge N \convas \tau_w^+$ as $N \to \infty$. Hence, we have $\lim_{N \to \infty} \mathbb{E}[\tau_w^+ \wedge N] = \mathbb{E}[\tau_w^+]$ due to the monotone convergence theorem~(MCT), which implies the following: 
            \begin{align}
                \mathbb{E}[\tau_w^+] \leq \frac{\log(1/\alpha) + C + c \log( \mathbb{E}[\tau_w^+ ])}{\klinf(P^*, \mu+w/2)}.  \label{eq:tau-w-plus-1}
            \end{align}

            \paragraph{Step 3: Bound the oracle effective width.} Now, for any fixed $n \in \mathbb{N}$, let $w_n^+$ denote the smallest value of $w$ for which $\mathbb{E}[\tau_w^+]$ is no larger than $n/2$.  By~\eqref{eq:tau-w-plus-1}, it follows that 
            \begin{align}
                \frac{w_n^+}{2} \leq \inf \lbr w \geq 0: \frac{ \log(1/\alpha) + C + c\log(n)}{\klinf(P^*, \mu+w/2)} \leq \frac{n}{2} \rbr = \klinf(P^*, \cdot)^{-1}\lp \frac{ 2\log (n^c/\alpha) + 2C}{n} \rp - \mu. 
            \end{align}
            By repeating this argument of $\tau_w^-$, we get an upper bound on $w_n^-$ 
            \begin{align}
                 \frac{w_n^-}{2} \leq \inf \lbr w \geq 0: \frac{ \log(1/\alpha) + C + c\log(n)}{\klinfminus(P^*, \mu-w/2)} \leq \frac{n}{2} \rbr = \mu -  \klinfminus(P^*, \cdot)^{-1}\lp \frac{ 2\log (n^c/\alpha) + 2C}{n} \rp.                
            \end{align}
            Combining these two results, we get the required upper bound on the effective width, $w_e(n, P^*, \alpha) \leq \max \{w_n^+, w_n-\}$. That is, 
            \begin{align}
                &w_n(n, P^*, \alpha) \leq 2\max \lbr \klinf(P^*, \cdot)^{-1}\lp b(n,\alpha) \rp - \mu, \, \mu - \klinfminus(P^*, \cdot)^{-1}\lp b(n,\alpha) \rp \rbr,  \label{eq:eff-w-betting-cs-1} \\
                \text{where } &b(n,\alpha) = \frac{ 2\log (n^c/\alpha) + 2C}{n},
            \end{align}
            and $C$  is a  constant that depends on the support of $P^*$~(\Cref{assump:support}).

            \subsubsection{Proof of~Lemma~\ref{lemma:oracle-stopping-time}}
            \label{proof:oracle-stopping-time}
                
                To prove this statement, we first introduce the following two stopping times, defined for any $w > 0$: 
                \begin{align}
                    N_w^+ = \inf \{n \geq 1: C_n \cap (\mu+w/2, 1] = \emptyset \}, \quad \text{and} \quad N_w^- = \inf \{n \geq 1: C_n \cap [0, \mu-w/2) = \emptyset \}. 
                \end{align}
                In words, $N_w^+$~(resp. $N_w^-$) denotes the first time at which the betting CS has rejected all $m$ values larger than $\mu+w/2$~(resp. smaller than $\mu-w/2$). Due to the nested nature of CSs, if a point $m$ has been discarded by the CS at some time $N$, then it remains discarded for all $n \geq N$. This means, that at $n=\max\{N_w^+, N_w^-\}$, all values of $m$ outside the band $\mu \pm w/2$ have been rejected by the CS; or equivalently, the width of the CS is smaller than $w$. By definition, this means that $T_w \leq \max\{N_w^+, N_w^-\}$.

                We will prove the required statement, $N_w^+ \leq \tau_w^+$, through a series of intermediate inequalities . First, we need to introduce some new terms. For any $w>0$, we define $M_w^+$ to be the first time at which $\mu+w/2$ is discarded by the CS; that is, 
                \begin{align}
                    M_w^+ = \inf\{n \geq 1: \mu+w/2 \not \in C_n \}. 
                \end{align}
                It immediately follows from this definition that $N_w^+ = \sup_{w'>w} M_{w'}^+$. Next, we define another stopping time, $\widetilde{\tau}_w^+$, that marks the first time at which a lower bound on the wealth $(W_n(\mu+w/2))_{n \geq 1}$ exceeds the rejection threshold $1/\alpha$. More formally, we have 
                \begin{align}
                    &\widetilde{\tau}_w^+ = \inf \{ n \geq 1: \log(\widetilde{W}^+_n(\mu+w/2)) \geq \log(1/\alpha) + 2\log n \}, \quad \text{where} \\
                    & \log(\widetilde{W}^+_n(\mu+w/2)) \defined \sup_{\lambda \in \lb 0,  \frac{1}{1-\mu-w/2}\rb} \sum_{t=1}^n \log (1 + \lambda (\mu+w/2 - X_t)). 
                \end{align}
                Since $\log(\widetilde{W}^+_n(\mu+w/2))-2\log n$ is a lower bound on the wealth $W_n(\mu+w/2)$, it follows that $M_w^+$~(which denotes the first $1/\alpha$ crossing of the wealth process at $\mu+w/2$) upper bounded by $\widetilde{\tau}_w^+$ for all $w>0$. 
                To complete the proof, we will show the following chain of inequalities hold: 
                \begin{align}
                    N_w^+ \;=\; \sup_{w' > w} M_{w'}^+ \;\leq\; \sup_{w' > w} \widetilde{\tau}_{w'}^+ \; \stackrel{(i)}{\leq}\; \widetilde{\tau}_w^+ \;\stackrel{(ii)}{\leq}\; \tau_w^+. \label{eq:stopping-times-chain}
                \end{align}
                The first equality is simply the definition of $N_w^+$, while the second inequality follows from the fact that $\widetilde{\tau}_{w'}^+ \geq M_{w'}^+$ for all $w'>0$.  We prove the correctness of the two inequalities  $(i)$ and $(ii)$ below, which with conclude the proof.

                \paragraph{Proof of $\boldsymbol{(i)}$ in~\eqref{eq:stopping-times-chain}.} For $w' \geq w$, observe the following~(all the inequalities and equalities for the rest of this proof hold a.s.): 
                \begin{align}
                    \mu + w'/2 - X_t \geq  \mu + w/2 - X_t, \quad \text{for all } t \geq 1.
                \end{align}
                This implies that for all fixed values of $\lambda \in [0, 1/(1-\mu-w/2)]$, we have 
                \begin{align}
                    \sum_{t=1}^n \log \lp 1 + \lambda \lp \mu + w'/2 - X_t \rp \rp    \geq \sum_{t=1}^n \log \lp 1 + \lambda \lp \mu + w/2 - X_t \rp \rp.  \label{eq:tau-proof-1}
                \end{align}
                Now, taking the supremum over appropriate ranges of $\lambda$ leads to the following: 
                \begin{align}
                    \log \lp \widetilde{W}^+_n (\mu + w'/2)  \rp &= \sup_{\lambda \in \lb 0, \frac{1}{1 - \mu - w'/2} \rb} \sum_{t=1}^n \log \lp 1 + \lambda \lp \mu + w'/2 - X_t \rp \rp \\
                    & \geq \sup_{\lambda \in \lb 0, \frac{1}{1 - \mu - w/2} \rb} \sum_{t=1}^n \log \lp 1 + \lambda \lp \mu + w'/2 - X_t \rp \rp \\
                    & \geq \sup_{\lambda \in \lb 0, \frac{1}{1 - \mu - w/2} \rb} \sum_{t=1}^n \log \lp 1 + \lambda \lp \mu + w/2 - X_t \rp \rp \\
                    & = \log \lp \widetilde{W}_n^+(\mu + w/2) \rp. 
                \end{align}
                The first inequality above is due to the smaller domain of the supremum, while the second inequality is simply due to~\eqref{eq:tau-proof-1}. Since $\widetilde{\tau}_w^+$ is the first time $\log \widetilde{W}_n^{+}(\mu + w/2)$ crosses the boundary $\log(n^2/\alpha)$, this implies that $\widetilde{\tau}_{w'} \leq \widetilde{\tau}_{w}$.  Thus, we have proved the inequality $(i)$ in~\eqref{eq:stopping-times-chain}:
                \begin{align}
                    \sup_{w' \geq w} \widetilde{\tau}_{w'}^+ \leq \sup_{w' \geq w} \widetilde{\tau}_w^+ = \widetilde{\tau}_w^+. 
                \end{align}

                \paragraph{Proof of $\boldsymbol{(ii)}$ in~\eqref{eq:stopping-times-chain}.}  The final step is to bound $\widetilde{\tau}_w^+$ with the stopping time, $\tau_w^+$, defined in~\eqref{eq:tau-w-plus-def}. 
                \begin{align}
                    \tau_w^+ &= \inf\{ n \geq 1: \log (W_n^{*,+}(\mu+w/2)) \geq \log(n^2/\alpha)\} \\
                    & = \inf \lbr \{ n \geq 1: \sum_{t=1}^n \log \lp 1 + \lambda^*_+(X_t - m) \rp  \geq \log(n^2/\alpha) \rbr \} \\
                    & \geq \inf \lbr n \geq 1: \sup_{\lambda \in [-1/(1-\mu-w/2), 0]}\, \sum_{t=1}^n \log \lp 1 + \lambda^*_+(X_t - m) \rp  \geq \log(n^2/\alpha) \rbr \\ 
                    &= \inf\{ n \geq 1: \log (\widetilde{W}_n^{+}(\mu+w/2)) \geq \log(n^2/\alpha)\} \\
                    & = \widetilde{\tau}_w^+. 
                \end{align}
                Thus we have established the chain of inequalities stated in~\eqref{eq:stopping-times-chain}, required to show that $N_w^+ \leq \tau_w^+$. An exactly analogous argument implies that $N_w^- \leq \tau_w^-$, which completes the proof of~\Cref{lemma:oracle-stopping-time}.

\section{Proof of lower bounds for multivariate observations}
\label{proof:multivariate-lower-bound}

    \subsection{Proof of~Proposition~\ref{prop:multivariate-lower-bound-1}}
    \label{proof:multivariate-lower-bound-1}
        The proof of this statement is a simple generalization of the steps involved in proving~\Cref{prop:lower-bound-1} to deal with multivariate observations. 
        
        \paragraph{Step 1: Design a hypothesis test using $\boldsymbol{\mc{C}}$.}
        Given observations $X_1, X_2, \ldots \simiid P^*$ for $\mc{X} = \reals^d$, we consider the hypothesis testing problem with $H_0: P^* = P$ and $H_1: P^* = Q$, with $\|\mu_Q - \mu_P\| > w(n, P, \alpha)$. For this problem, we can define a test $\Psi:\mc{X}^n \to \{0,1\}$, as follows: 
        \begin{align}
            \Psi(X^n) = \begin{cases}
                1, & \text{ if } \mu_Q \in C_n = \mc{C}(X^n, \alpha), \\
                0, & \text{ otherwise}. 
            \end{cases}
        \end{align}
        As the method $\mc{C}$ constructs level-$(1-\alpha)$ CIs, the test $\Psi$ controls both type-I and type-II errors at level $\alpha$. 
    
        \paragraph{Step 2: Lower bound on $\boldsymbol{\dkl(P, Q)}$.} Now, by the chain rule for KL divergence along with the data processing inequality, we obtain the following, with $a(n, \alpha) = \log\lp (1-\alpha)^{1-\alpha} \alpha^{2\alpha-1} \rp/n$: 
        \begin{align}
             \dkl\lp P, Q \rp = \frac{1}{n} \dkl\lp P^{\otimes n}, Q^{\otimes n}  \rp \geq \frac{1}{n} \dkl \lp \mathbb{E}_{P^{\otimes n}}[1-\Psi],\; \mathbb{E}_{Q^{\otimes n}}[1-\Psi] \rp \geq a(n, \alpha). 
             \label{eq:multi-lower-bound-1-1}
        \end{align}
        
        \paragraph{Step 3: Connections to $\boldsymbol{\klinf}$.} Since the right-most term in~\eqref{eq:multi-lower-bound-1-1} is independent of $Q$, we can take an infimum over all $Q: \|\mu_Q - \mu_P\|_2 > w(n, P, \alpha)$, to get 
        \begin{align}
            \klinfmulti\lp P, w(n, P, \alpha), \mc{P} \rp \geq   \log\lp (1-\alpha)^{1-\alpha} \alpha^{2\alpha-1} \rp/n. \label{eq:multi-lower-bound-1-2}
        \end{align}
        The monotonicity of $\klinf$ then implies the required 
        \begin{align}
            w(n, P, \alpha) \geq \klinfmulti(P, \cdot, \mc{P}_0)^{-1}\lp a(n, \alpha)\rp. 
        \end{align}

    \subsection{Proof of~Proposition~\ref{prop:multivariate-lower-bound-2}}
    \label{proof:multivariate-lower-bound-2}
        The proof of this result follows by appropriately modifying the steps used in proving the analogous univariate result in~\Cref{prop:lower-bound-2}. We omit some of the details involved in proving this result, as they closely follow the proof of~\Cref{prop:lower-bound-2}.

        Given a stream of observations $X_1, X_2, \ldots \simiid P^*$,   we  set up  the following  hypothesis testing problem: 
        \begin{align}
            H_0: P^*=Q, \quad \text{versus} \quad H_1: P^*=P.
        \end{align}
         We assume that the means of the two distributions~($\mu_P$ and $\mu_Q$) satisfy the constraint $\|\mu_Q-\mu_P\|_2 = w_0 > w_e(n, P, \alpha)$. 
        
        \paragraph{Step 1: Sequential test.} We define a sequential test for this problem, that stops at a random time $T$, and makes a decision $\Psi \in \{0, 1\}$, as follows: 
        \begin{align}
            T = \inf \{n \geq 1: \sup_{x, x' \in C_n} \|x-x'\|_2 < w_0\}, \quad \text{and} \quad 
            \Psi = \boldsymbol{1}_{\mu_Q \not \in C_T}. 
        \end{align}
        Following the exact argument used in~\Cref{lemma:lower-bound-2-1}, we can show that this test controls both the type-I and type-II errors at level $\alpha$. 
        
        \paragraph{Step 2: Lower bound on KL divergence.} Next, by an application of the optional stopping theorem, and the data-processing inequality, we  show that the KL divergence between $P$ and $Q$ must satisfy: 
            \begin{align}
                \dkl(P, Q) \geq \frac{ \log \lp (1-\alpha)^{1-\alpha} \alpha^{2\alpha-1} \rp}{\mathbb{E}_{H_1}[T]}. \label{eq:multi-lower-bound-2-1}
            \end{align}           
        The steps involved to obtain this follow the same argument as in the proof of~\Cref{lemma:lower-bound-2-2}.         
        \paragraph{Step 3: Lower bound on width.} Now, recall that $\|\mu_Q - \mu_P\|_2 = 2w_0$, where $w_0>w_e(n, P, \alpha)$ by construction. Then, we have for any $w \in (w_e(n, P, \alpha), w_0)$: 
        \begin{align}
            \mathbb{E}_{H_1}[T] \;\leq\; \mathbb{E}[T_{w}(P, \alpha)] \;\leq\;   n,  
        \end{align}
        where $T_w = \inf \{n \geq 1: \sup_{x, x' \in \mc{C}_n} \|x-x'\|_2  \leq w\}$. The first inequality then uses the fact that $w < w_0$, and the second inequality follows from the definition of $w_e(n, P, \alpha)$ and the fact that $w>w_e(n, P, \alpha)$ 
        Plugging the inequality into~\eqref{eq:multi-lower-bound-2-1}, and taking the infimum over all $Q \in \mc{P}_0: \|\mu_Q - \mu_P\|_2 \geq w_0$, we get 
        \begin{align}
            \klinfminus(P, w_0, \mc{P}_0) \geq a(n, \alpha). 
        \end{align}
        The required result then follows by taking the inverse information projection. 

\section{Details of the \wor case}
\label{appendix:wor-sampling}
    We recall some existing results about constructing CIs in the framework of uniform sampling without replacement in~\Cref{appendix:wor-ci-background}, and then present the proof of~\Cref{prop:wor-limiting-width} in~\Cref{proof:wor-limiting-width}. 
    \subsection{Background on \wor CIs}
    \label{appendix:wor-ci-background}

        \citet{hoeffding1963probability} showed the following relation between the expected value of a convex function of random variables drawn with, and without replacement. 
        \begin{fact}
            \label{fact:hoeffding-relation}
            Let $Y_1, \ldots, Y_n$ denote \iid $\text{Uniform}(\Xsample_M)$ random variables drawn with replacement, and $X_1,  \ldots, X_n$ denote their \wor analogs. 
            Then, for any continuous and convex function $f: \reals \to \reals$, we have 
            \begin{align}
                \mathbb{E}\lb f \lp \sum_{i=1}^n X_i \rp \rb  \leq \mathbb{E}\lb f \lp \sum_{i=1}^n Y_i \rp \rb. 
            \end{align}
        \end{fact}
        The above relation implies an extension of the usual (with replacement) Hoeffding's inequality to  the \wor case. 
        \begin{fact}
            \label{fact:wor-hoeffding-1}
            For a fixed sample size $n$, we have the following for any $\epsilon>0$. 
            \begin{align}
                \mathbb{P}\lp \lv \frac{1}{n} \sum_{i=1}^n X_i - \mu\rv \geq \epsilon \rp \leq 2 \exp \lp -2n\epsilon^2 \rp. 
            \end{align}
        \end{fact}
        However, the above result does not capture the fact that the uncertainty about the mean, $\mu = \frac{1}{M} \sum_{i=1}^M x_i$, rapidly decays as the sample size $n$ approaches $M$, the total number of items. This was addressed by~\citet{serfling1974probability}, who obtained an improved version of Hoeffding's inequality, with the term $n$ replaced by $\frac{n}{1 - (n-1)/M}$. 
        \citet{bardenet2015concentration} obtained a slight improvement of Serfling's result, replacing the $(n-1)$ with $n$, and then used it to obtain variants of Hoeffding-Serfling, Bernstein-Serfling, and Empirical-Bernstein-Serfling inequalities. 
        
        We now recall the Bernstein-Serfling confidence interval~(CI) obtained by~\citet[Corollary~3.6]{bardenet2015concentration} for uniform sampling without replacement.
        \begin{fact}[\citet{bardenet2015concentration}]
        \label{fact:wor-bernstein}
            Let $\Xsample_M$ denote numbers $\{x_1, \ldots, x_M\}$ lying in $[0,1]$. Suppose  $X_1, X_2, \ldots$ are drawn uniformly from $\Xsample_M$ without replacement. Introduce the terms $\mu = \frac{1}{M} \sum_{i=1}^n x_i$, and $\sigma^2 = \frac{1}{M} \sum_{i=1}^M (x_i-\mu)^2$. Then, we have the following with probability at least $1-\alpha$ for some $\alpha \in (0,1]$, for all $t \geq n$: 
            \begin{align}
                &\mu \in \lb \frac{1}{n} \sum_{i=1}^n X_i - \frac{w_n^{(BM)}}{2},\;  \frac{1}{n} \sum_{i=1}^n X_i + \frac{w_n^{(BM)}}{2} \rb, \\
                \quad \text{where} \quad &\frac{w_n^{(BM)}}{2} = \sigma \sqrt{ \frac{2(1-n/M)(1+1/n) \log(2/\alpha)}{n} } + \lp \frac{4}{3} + \sqrt{\lp \frac{M}{n+1}-1 \rp \lp 1- \frac{n}{M} \rp} \rp\frac{  \log(2/\alpha)}{n}. 
            \end{align}
        \end{fact}
        \begin{remark}
            Note that the above expression involves the $\log(2/\alpha)$ term, instead of $\log(4/\alpha)$, because we have stated only the `second half' of the result, valid for $n \geq M/2$. This will be used in a direct comparison with the confidence sequences constructed by~\citet{waudby2020confidence}.  
        \end{remark}

    \subsubsection{\wor CIs proposed by \citet{waudby2023estimating}}
    \label{sub2sec:wor-ci-WSR}
        \paragraph{\wor \prpieb CI.} To describe this CI, we need to introduce the following terms: 
        \begin{align}
            \lambda_{t,n} = \sqrt{ \frac{2 \log(2/\alpha)}{n V_{t-1}}}, 
            \quad V_t  = \frac{\frac{1}{4} + \sum_{i=1}^t (X_i - \muhat_{i-1})^2}{t},
            \quad \text{and} \quad 
            \muhat_t = \frac{1}{t} \sum_{i=1}^t X_i. \label{eq:wor-lambda-t-n}
        \end{align}
        For this predictable sequence $\{\lambda_t: t \in [M]\}$,  define the  mean estimate, $\mutilde_t$,  as 
        \begin{align}
            \label{eq:weighted-mean}
            \mutilde_t \defined \lp { \sum_{i=1}^t \lambda_{i,n}\lp X_i + \frac{1}{M-i+1}\sum_{j=1}^{i-1} X_j \rp}\rp\bigg/\lp { \sum_{i=1}^t \lambda_{i,n} \lp 1 + \frac{1}{M-i+1} \rp }\rp. 
        \end{align}
        Recall that we have assumed that all the values $x_i \in [0,1]$, and that $X_i$ are drawn uniformly without replacement from $\Xsample_M = \{x_1, \ldots, x_M\}$. With $\mu_M \defined \frac{1}{M} \sum_{i=1}^M x_i$, the \wor \prpieb CI is defined as: 
        \begin{align}
            &C_n^{\prpieb} = \lb \mutilde_n \pm \frac{w_n^{\prpieb)}}{2} \rb    
            \quad \text{where} \quad  
            \frac{w_n^{(\prpieb)}}{2} = \frac{\log(2/\alpha)+ 4 \sum_{i=1}^t (X_i - \muhat_{i-1})^2\psi_E(\lambda_{i,n})}{\sum_{i=1}^t \lambda_{i,n} \lp 1 + \frac{i-1}{M-(i-1)} \rp}. 
        \end{align}

        \paragraph{\wor betting CI.} The \wor betting CI is constructed similar to the with-replacement case. The main difference is in the definition of the wealth process $\{W_t(m): t \geq 1\}$, which is centered with a time-varying conditional expectation. In particular, let $X_1, X_2, \ldots$ denote the sequence of observations, drawn uniformly \wor from a set of numbers $\Xsample_M = \{x_1, \ldots, x_M\}$, and introduce the terms 
            \begin{align}
                \mu_t = \frac{M\mu - \sum_{i=1}^{t-1} X_i}{M - t+1}, \quad \text{and} \quad 
                m_t = \frac{Mm - \sum_{i=1}^{t-1} X_i}{M - t+1}. 
            \end{align}
            Define the wealth at time $n$, 
            \begin{align}
                W_n(m) = \prod_{t=1}^n \big( 1 + \lambda_t(m) \lp X_t - m_t \rp \big), 
            \end{align}
            with the bets $\{\lambda_t(m): t \geq 1\}$ selected according to any predictable betting strategy, such as the mixture method of~\Cref{def:mixture-method}. The version of the betting CI analyzed in~\Cref{prop:wor-limiting-width} is constructed by setting $\lambda_t(m) = \lambda_{t,n}$ for all $m \in [0,1]$, with $\lambda_{t,n}$ as defined in~\eqref{eq:wor-lambda-t-n}. 

    \subsection{Proof of \wor betting CI limiting width~(Proposition~\ref{prop:wor-limiting-width})}
    \label{proof:wor-limiting-width}
        We prove this result in the following steps: 
            \begin{itemize}
                \item We begin by obtaining a high probability concentration result for the empirical mean and variance about their true values in~\Cref{lemma:eb-cs-width-1}, for a fixed value of $n$ and $M$. 
                As a consequence of the above, we show in~\Cref{lemma:eb-cs-width-2} that $|V_n-\sigma_M^2| \convas 0$, and  $|\muhat_n-\mu_M| \convas 0$, under the assumption that $n = \lceil \rho M\rceil $, as $M \to \infty$

                \item Then, in~\Cref{lemma:eb-cs-width-3}and~\Cref{lemma:eb-cs-width-4}, we evaluate the numerator and denominator in the limiting width of the \wor \prpieb CI. Together these two results imply that 
                 $\gamma_1^{\prpieb}(\rho)$ is  almost surely equal to $  \frac{2}{\rho}\log \lp \frac{1}{1 - \rho} \rp \frac{\sqrt{2 \log(2/\alpha)}}{\sigma}$.
            \end{itemize}

        \paragraph{Step 1: Concentration of empirical estimates.}  In the next two lemmas, we show that the difference between empirical estimates of the mean, $\muhat_t = \frac{1}{t} \sum_{i=1}^t X_i$, and the variance $V_t = \frac{1}{t+1} \big(  1/4 + \sum_{i=1}^t (X_i - \muhat_{i-1})^2 \big)$ and their true values~($\mu_M$ and $\sigma^2_M$) converge to zero almost surely as $M\to \infty$. 
            \begin{lemma}
                \label{lemma:eb-cs-width-1}
                For any fixed $1 \leq n \leq M$, the event $\mc{E}_n$ defined below, occurs with probability at least $1-2/n^2$.  
                \begin{align}
                    &\mc{E}_{n,1} = \lbr \forall t \in [n]: |\muhat_t - \mu_M| \leq \mc{O}\lp \sqrt{\log n/t}\rp \rbr, \quad
                    \mc{E}_{n,2} = \lbr \forall t \in [n]: |V_t - \sigma_M^2| \leq \mc{O}\lp \sqrt{\log n/ t}\rp \rbr,  \\
                    \quad \text{and} \quad &\mc{E}_n = \mc{E}_{n,1} \cap \mc{E}_{n,2}. 
                \end{align}
                Recall that $\mu_M$ and $\sigma_M^2$ denote the mean and variance computed using all the $M$ values in $\Xsample_M$. 
            \end{lemma}
            \begin{proof}
                This result follows by applying the usual, fixed sample size, Hoeffding-Serfling inequality~(e.g., \Cref{fact:wor-hoeffding-1}) for the mean and standard deviation estimates. Both of these results use the fact that the observations $(X_i)$ lie in the interval $[0,1]$. Since we are interested in the asymptotic limiting width, we use $\mc{O}(\cdot)$ in the definition of the events, instead of the exact constants, to simplify the notation.  
            \end{proof}
    
            As a consequence of the above result, we can easily conclude the following. 
            \begin{lemma}
                \label{lemma:eb-cs-width-2}
                Suppose $\rho > 0$, and $n= \lceil \rho M \rceil$ for all $M \geq 1$. Then, we have $\mathbb{P}\lp \mc{E}_n^c \text{i.o.} \rp = 0$ as $M \to \infty$. This in turn implies that $\lim_{M \to \infty} | V_n - \sigma_M^2| \to 0$, and $|\muhat_n - \mu_M| \to 0$, almost surely. 
           \end{lemma}
           \begin{proof}
               The first statement follows by a standard application of the (first) Borel-Cantelli lemma. In particular, note that by~\Cref{lemma:eb-cs-width-1}, we have $\mathbb{P}\lp \mc{E}_n^c \rp \leq \frac{1}{n^2}$, which implies that $\sum_{M \geq 1/\rho}^{\infty} \mathbb{P}\lp \mc{E}_{\lceil \rho M \rceil}^c \rp = \mc{O}\lp \sum_{n=1}^{\infty} 1/n^2 \rp <\infty$. Hence, we can conclude that $\mathbb{P}\lp \mc{E}_{\lceil \rho M \rceil}^c \text{i.o.} \rp = 0$, by the first Borel-Cantelli lemma. 
    
               Since $\lim_{M \to \infty} \sqrt{\log \lceil \rho M \rceil / \lceil \rho M \rceil} = 0$, we can thus conclude that $|\muhat_{\lceil \rho M \rceil} - \mu_M| \stackrel{a.s.}{\rightarrow}  0$ and $|V_{\lceil \rho M \rceil}-\sigma^2_M| \stackrel{a.s.}{\rightarrow}0$, due to the definition of the events $\mc{E}_{\lceil \rho M \rceil,1}$ and $\mc{E}_{\lceil \rho M \rceil,2}$ respectively.
           \end{proof}
        \paragraph{Step 2: Limiting width of \wor \prpieb CI.}
            Our next two lemmas compute the limits of the $\lambda_{t,n}$ dependent terms in numerator and denominator respectively of the width of the \wor \prpieb CI of~\citet{waudby2023estimating}. 
            \begin{lemma}
                \label{lemma:eb-cs-width-3}
                The following is true: 
                \begin{align}
                    \lim_{M \to \infty} 4 \sum_{i=1}^n (X_i - \muhat_{i-1})^2 \psi_E(\lambda_{i,n}) \stackrel{a.s.}{=}  \log(2/\alpha).   
                \end{align}
            \end{lemma}

            The proof of this statement is in~\Cref{proof:eb-cs-width-3}. Next, we obtain the limiting value of the normalized denominator  of $w_n$. 
            \begin{lemma}
                \label{lemma:eb-cs-width-4}
                Suppose $\lim_{M \to \infty} n/M = \rho$.  Then, we have the following 
                \begin{align}
                    \frac{1}{\sqrt{n}} \sum_{i=1}^n \lambda_{i, n} \lp 1 + \frac{i-1}{M-(i-1)} \rp \convas    \frac{1}{\rho}\log \lp \frac{1}{1 - \rho} \rp \frac{\sqrt{2 \log(2/\alpha)}}{\sigma}.
                \end{align}
            \end{lemma}
            The proof of this statement is in~\Cref{proof:eb-cs-width-4}.
            Combining these two results, we conclude that the limiting width of the \wor \prpieb CI is equal to $2\rho \sqrt{2 \log(2/\alpha)}/\log(1/(1-\rho))$. 

            \paragraph{Step 3: Obtain the bound on betting CI limiting width.} Our final step is to show that the limiting width of the \wor betting CI  is no larger than that of the \wor \prpieb CI. For this, we employ the general strategy we developed for proving~\Cref{prop:betting-CI-vs-EB-CI} in~\Cref{proof:betting-CI-vs-EB-CI}, with some modifications to deal with \wor sampling. To avoid repetition, we omit some of the details in the steps that proceed almost exactly as the analogous steps in the proof of~\Cref{prop:betting-CI-vs-EB-CI}. 

            Recall that, as in the with replacement case, the wealth $W_n(m)$ is the average of the two components: 
            \begin{align}
                W_n(m) = \frac{1}{2} \lp W_n^+(m) + W_n^-(m) \rp, \quad \text{where} \quad W_n^{\pm}(m) = \prod_{i=1}^n \lp 1 \pm \lambda_{t,n}(m)(X_t - m_t) \rp. 
            \end{align}
            This implies that the betting CI  satisfies the inclusion: 
            \begin{align}
                \CIbet \subset \cap_{a \in \{-, +\} } \lbr m \in [0,1]: \log(W_n^a(m)) < \log(2/\alpha) \rbr. 
            \end{align}
            Next, we observe that the wealth $W^+_n(m)$ can be lower bounded by an application of Fan's inequality~\eqref{eq:fan-inequality-0}: 
            \begin{align}
                \log(W^+_n(m)) &\geq  \sum_{t=1}^n \lambda_{t,n}(X_t-m_t) - 4\sum_{t=1}^t\psi_E(\lambda_{t,n}) (X_t-m_t)^2 \\
                & \geq  \sum_{t=1}^n \lambda_{t,n}(X_t-m_t) - 4\sum_{t=1}^t\psi_E(\lambda_{t,n}). 
            \end{align}
            Repeating this argument for $W_n^-(m)$,  we can define a larger CI that contains the betting CI: 
            \begin{align}
                \CIbet \subset \CIfan \defined \lb \mutilde_n \pm \frac{\widetilde{w}_n^{\dagger}}{2}\rb, \quad \text{with} \quad  \frac{\widetilde{w}_n^{\dagger}}{2} \defined \frac{ \log(2/\alpha) + \sum_{t=1}^n \psi_E(\lambda_{t,n})}{\sum_{t=1}^n \lambda_{t,n}(1+ (t-1)/(M-t+1)}. 
            \end{align}
            Then, following the exact argument as~\Cref{lemma:betting-CI-1}, we can show that \begin{align}
                \limsup_{M \to \infty, n = \lceil \rho M \rceil} \sqrt{n} \times \widetilde{w}_n^{\dagger} = 2c_2 = \lp \frac{2\rho}{\log(1/(1-\rho))} \rp \sqrt{\frac{\log(2/\alpha)}{2}}\lp \sigma_M + \frac{1}{\sigma_M} \rp < \infty. \label{eq:wor-fan-limit} 
            \end{align} 
            Thus, as before, we can define a larger CI, denoted by $\CItilde$, that contains the betting CI. Furthermore, we can also define an  random time~($N$), after which the width of $\CIfan$ is smaller than $4c_2/\sqrt{n}$. 
            \begin{align}
                \CItilde = \lb \mutilde_n \pm 2c_2/\sqrt{n} \rb, \quad \text{and} \quad 
                N = \sup\{M \geq 1: |\CIfan| > 4c_2/\sqrt{n}, \text{ for }\; n=\lceil \rho M \rceil \}.  
                \label{eq:wor-CI-tilde}
            \end{align}

            To proceed with the proof, we need two technical results. The first result shows that under the high probability event $\mc{E}_n$, the modified mean estimate, $\mutilde_n$, is close to the true mean $\mu_M$. 
            \begin{lemma}
                \label{lemma:wor-lemma-2}
                Under the event $\mc{E}_n$ defined in~\Cref{lemma:eb-cs-width-1}, we have the following: 
                \begin{align}
                    |\mutilde_n - \mu_M| = \mc{O}\lp \sqrt{\frac{\log n}{n}} \rp.  
                \end{align}
                Recall that $\mutilde_n$ was defined in~\eqref{eq:weighted-mean}, and $\mu_M = \frac{1}{M} \sum_{i=1}^M x_i$, is the true mean of the $M$ items. 
            \end{lemma}
            The proof of this statement is in~\Cref{proof:wor-lemma-2}. We now use this to obtain the next technical statement, that says that the term $(X_t - m_t)^2$ can be sufficiently well approximated with $(X_t - \muhat_{t-1})^2$ for all $m \in \CItilde$, and for all $M \geq N$.
            \begin{lemma}
                \label{lemma:wor-lemma-3}
                 The random variable $N$ defined above in~\eqref{eq:wor-CI-tilde} is finite almost surely. Furthermore, with the high probability event $\mc{E}_n$ as defined in~\Cref{lemma:eb-cs-width-1}, for any $m \in \CItilde$ and $2 \leq t \leq n$,   we have 
                    \begin{align}
                        \boldsymbol{1}_{M \geq N} \boldsymbol{1}_{\mc{E}_n} (X_t-m_t)^2   \; \stackrel{a.s.}{\leq} \;  (X_t-\mu_M)^2 + \mc{O}\lp \sqrt{\frac{ \log n}{t} } \rp.  \label{eq:wor-betting-ci-proof-6}
                    \end{align}               
            \end{lemma}
            The proof of this result is in~\Cref{proof:wor-lemma-3}. 
        By the definition of $N$, under the event $\{N \leq M\}$, the confidence interval $\CItilde$ contains $\CIfan$, which by construction, contains $\CIbet$. Hence, we can write the following inclusion for $\CIbet$  under $\{N \leq M\}$: 
        \begin{align}
            \CIbet & \subset \lbr m \in \CItilde: \log\lp W_n^+(m) \rp < \log(2/\alpha) \rbr \cap \lbr m \in \CItilde: \log\lp W_n^-(m) \rp < \log(2/\alpha) \rbr \label{eq:wor-proof-1}. 
        \end{align}
        Let $l^*$ and $u^*$ denote the (random) smallest and largest values of $m$ respectively, that  satisfy the conditions in the right side of~\eqref{eq:wor-proof-1} with equality.   Then, we have the following: 
        \begin{align}
            \boldsymbol{1}_{M \geq N} \times |\CIbet| \; \stackrel{a.s.}{\leq} \; u^* - l^* =  \frac{2 \log(2/\alpha) + 4 \sum_{t=1}^n \psi_E(\lambda_{t,n}) \lp (X_t-u^*_t)^2 + (X_t-l^*_t)^2 \rp }{\sum_{t=1}^n \lambda_{t,n}\lp 1 + \frac{t-1}{N-t+1} \rp}. \label{eq:wor-proof-2}
        \end{align}
        As for all $m \in [0,1]$, we define $l^*_t = \frac{Nl^* - \sum_{i=1}^{t-1} X_i}{N-t+1}$ and $u^*_t = \frac{Nu^* - \sum_{i=1}^{t-1} X_i}{N-t+1}$ above. 
 
                Next, note  that since $N$ is known to be finite almost surely (\Cref{lemma:wor-lemma-3}), we have 
                \begin{align}
                   \gammabet(\rho) \defined \limsup_{M \to \infty, n = \lceil \rho M\rceil} \sqrt{n} \times |\CIbet| \eqas \limsup_{M \to \infty} \sqrt{n} \times |\CIbet| \times \boldsymbol{1}_{M \geq N}. 
                \end{align}
                Combining this observation with~\eqref{eq:wor-proof-2}, we get: 
                \begin{align}
                    \frac{\gammabet(\rho)}{2} \leq \limsup_{M \to \infty}\; \sqrt{n} \times \frac{ \log(2/\alpha) + 2 \sum_{t=1}^n \psi_E(\lambda_{t,n}) \lp (X_t-u^*_t)^2 + (X_t-l^*_t)^2 \rp }{\sum_{t=1}^n \lambda_{t,n}(1+ (t-1)/(M-t+1)} \times \boldsymbol{1}_{M \geq N}. 
                \end{align}
                Thus, to complete the proof, we need to show the following, with $U_t \defined (X_t-u^*_t)^2 + (X_t-l^*_t)^2$:  
                \begin{align}
                    \lim_{M \to \infty, \, n=\rhoM} \lp 2 \sum_{t=1}^n \psi_E(\lambda_{t,n}) \,U_t \rp \times \boldsymbol{1}_{M \geq N} \; \eqas \;  \sqrt{\log(2/\alpha)}. 
                \end{align}
                Since we already know that $\lim_{n \to \infty} 4 \sum_{t=1}^n \psi_E(\lambda_{t,n})(X_t-\muhat_{t-1})^2 \eqas \sqrt{\log(2/\alpha)}$, it suffices to prove 
                 \begin{align}
                    \lim_{M \to \infty,\, n=\rhoM} \lp 2 \sum_{t=1}^n \psi_E(\lambda_{t,n}) \,U_t \rp \times \boldsymbol{1}_{M \geq N} \eqas
                    \lim_{M \to \infty, \, n=\rhoM} 4 \sum_{t=1}^n \psi_E(\lambda_{t,n})(X_t-\muhat_{t-1})^2.  
                \end{align}       
                Since $\mathbb{P}\lp \{ \mc{E}_n\; i.o. \} \rp = 0$, we have 
                \begin{align}
                    \lim_{M \to \infty, \, n=\rhoM} \lp 2 \sum_{t=1}^n \psi_E(\lambda_{t,n}) \,U_t  \rp \times \boldsymbol{1}_{M \geq N} \boldsymbol{1}_{\mc{E}_n} \eqas \lim_{M \to \infty, \, n=\rhoM} \lp 2 \sum_{t=1}^n \psi_E(\lambda_{t,n}) \,U_t \rp \times \boldsymbol{1}_{M \geq N}. 
                \end{align}
                The final step is provided by~\Cref{lemma:wor-lemma-3}, which gives us 
                \begin{align}
                    U_t \times \boldsymbol{1}_{\mc{E}_n} \times \boldsymbol{1}_{M \geq N} \stackrel{a.s.}{\leq} 2(X_t - \mu_M)^2 + \mc{O}\lp \sqrt{ \frac{\log n}{t}} \rp.  
                \end{align}
                Plugging this into the previous equation, we obtain 
                \begin{align}
                    \lim_{M \to \infty, \, n=\rhoM} \lp 2 \sum_{t=1}^n \psi_E(\lambda_{t,n}) \,U_t  \rp \times \boldsymbol{1}_{M \geq N} \boldsymbol{1}_{\mc{E}_n}&\stackrel{a.s.}{\leq} \lim_{M \to \infty, \, n=\rhoM} 4 \sum_{t=1}^n \psi_E(\lambda_{t,n})\lp (X_t - \mu_M)^2 + \mc{O}\lp \sqrt{ \frac{\log n}{t}} \rp \rp \\
                    & \leq \sqrt{\log(2/\alpha)} + \mc{O}\lp \lim_{M \to \infty, \, n=\rhoM} \sum_{t=1}^n \psi_E(\lambda_{t,n})\sqrt{\frac{\log(n)}{t}}   \rp \\
                    & = \sqrt{\log(2/\alpha)}. 
                \end{align}
                The last equality follows from the fact that $\psi_E(\lambda_{t,n}) \approx \lambda_{t,n}^2/8 = \log(2/\alpha)/(4n V_t) \approx \log(2/\alpha)/(4n\sigma_M^2)$, which implies that the term inside the $\mc{O}\lp \cdot \rp$ converges to $0$.  This completes the proof. \hfill \qedsymbol

            \subsubsection{Proof of~\Cref{lemma:eb-cs-width-3}}
            \label{proof:eb-cs-width-3}
                    To simplify the presentation in this proof, we use $A \asymp B$ to indicate that $|A-B| \convas 0$ as $M \to \infty$.  
                    
                    We begin by noting that due to~\Cref{assump:eb-cs-limiting-width}, and~\Cref{lemma:eb-cs-width-2}, we have $V_n \convas \sigma^2>0$. As a result, $\lambda_{n, n} = \sqrt{ 2\log(2/\alpha)/(V_{n-1} n)} \convas 0$, which implies that $|\psi_E(\lambda_{n, n}) - \psi_H(\lambda_{n, n})| \convas 0$, where $\psi_H(\lambda) = \frac{\lambda^2}{8}$. Thus, we can conclude the following:
                    \begin{align}
                        4 \sum_{i=1}^n \lp X_i - \muhat_{i-1} \rp^2 \psi_E(\lambda_{i, n}) & \asymp 4 \sum_{i=1}^n \lp X_i - \muhat_{i-1} \rp^2   \frac{\lambda_{i, n}^2}{8}  \\
                        & \stackrel{(i)}{\asymp} 4 \sum_{i=1}^n \lp X_i - \muhat_{i-1} \rp^2 \frac{2 \log(2/\alpha)}{8 n V_i} \\
                        & \stackrel{(ii)}{\asymp} 4 \sum_{i=1}^n \lp X_i - \muhat_{i-1} \rp^2 \frac{2 \log(2/\alpha)}{8 n \sigma_M^2} 
                         \stackrel{(iii)}{\asymp}  4 \sum_{i=1}^n \lp X_i - \muhat_{i-1} \rp^2 \frac{2 \log(2/\alpha)}{8 n \sigma^2}. 
                    \end{align}
                    In the above display, $(i)$ uses the fact that $\lambda_{i, n}^2 = 2\log(2/\alpha)/n V_{i-1}$, and the fact that $\sum_{i=1}^n |\psi_E(\lambda_{i, n}) - \lambda_{i, n}^2/8| \convas 0$. Next, the equivalence~$(ii)$ uses the fact that $|V_i - \sigma_M^2| = \mc{O}\lp \sqrt{ \log n / i} \rp$ all but finitely often(~\Cref{lemma:eb-cs-width-1}), while $(iii)$ follows by~\Cref{assump:eb-cs-limiting-width}.  

                    Next, we continue with the above chain of inequalities to obtain 
                    \begin{align}
                        4 \sum_{i=1}^n \lp X_i - \muhat_{i-1} \rp^2 \psi_E(\lambda_{i, n}) 
                        & \asymp \frac{\log(2/\alpha)}{ \sigma^2}  \frac{1}{n}\sum_{i=1}^n \lp X_i - \muhat_{i-1} \rp^2  \\
                        & \stackrel{(iv)}{\asymp} \frac{\log(2/\alpha)}{ \sigma^2}  V_n 
                       \stackrel{(v)}{\asymp}  \frac{\log(2/\alpha)}{ \sigma^2}  \sigma_M^2 \\
                         & \stackrel{(vi)}{\asymp} \log(2/\alpha). 
                    \end{align}
                    The equivalence $(iv)$ simply follows from the definition of $V_n = \frac{1}{n+1}\lp  \frac{1}{4}  + \sum_{i=1}^n (X_i - \muhat_{i-1})^2 \rp$, while $(v)$ uses~\Cref{lemma:eb-cs-width-1} and $(vi)$ follows from~\Cref{assump:eb-cs-limiting-width}.  
                
            \subsubsection{Proof of~\Cref{lemma:eb-cs-width-4}}
            \label{proof:eb-cs-width-4}
                    We proceed as follows: 
                    \begin{align}
                        A_{M,n} \defined \frac{1}{\sqrt{n}} \sum_{i=1}^n \lambda_{i, n} \lp 1 + \frac{i-1}{M-(i-1)} \rp &=   \frac{1}{\sqrt{n}} \sum_{i=1}^n \lambda_{i, n} \lp  \frac{M}{M-(i-1)} \rp.  
                    \end{align}
                    Next, note that we have the following, with $\lambda_n^* \defined \sqrt{2 \log(2/\alpha)/n \sigma^2}$: 
                    \begin{align}
                      \lambda_{i, n} = \sqrt{ \frac{2\log(2/\alpha)}{n V_{i-1}} } 
                      = \sqrt{ \frac{2 \log(2/\alpha)}{n \sigma^2 \lp 1 - \frac{\sigma^2-V_{i-1}}{\sigma^2}\rp} } 
                      \asymp \lambda_n^* \sqrt{\lp 1 + \frac{\sigma^2 - V_{i-1}}{\sigma^2} \rp }. 
                    \end{align}
                    The last $\asymp$ follows by noting that the function $f(x) = 1/x$ is concave, hence, it can be bounded (over a bounded interval) from both above and below using linear approximations. That is, $f(1+x) \geq 1 - x$ for all $x \in (-1, \infty)$, and $f(1+x) \leq 1 - 2x$ for $x \in [-0.5, 0]$, and $f(1+x) \leq 1 - (2/3)x$, for $x \in [0, 0.5]$. Under the event $\mc{E}_{n,2}$, we know that the term $\sigma^2 - V_{i-1}$ is $\mc{O}\lp \sqrt{\log n/ i}\rp$. Hence, we have the following: 
                    \begin{align}
                        \frac{1}{\sqrt{n}} \sum_{i=1}^n \lambda_{i, n} \frac{M}{M-(i-1)} &\lesssim 
                        \frac{M}{\sqrt{n}} \sum_{i=1}^n \lambda^*_n \frac{1}{M-(i-1)} \lp 1 + \lp \frac{ \log n}{i} \rp^{1/4} \rp \nonumber \\
                        & \leq \frac{M}{\sqrt{n}} \lambda^*_n \lp S_M - S_{M-n} \rp + \frac{M \lambda^*_n}{\sqrt{n}} \lp \sum_{i=1}^n \frac{1}{(M-i+1)^2} \rp^{1/2} \lp \sqrt{\log n} \sum_{i=1}^n \frac{1}{\sqrt{i}} \rp^{1/2} \label{eq:eb-cs-proof-1}
                    \end{align}
                    where $S_k \defined \sum_{i=1}^k 1/i$. Now, since $S_k = \log k + 0.5772 + \frac{1}{2k} - \mc{O}(1/k^2)$, we can conclude that $S_M - S_{M-n} \asymp \log M - \log(M-n)$, under the assumption that $n = \rho M$ for some $\rho>0$. 
    
                    Next, we note that $\sum_{i=1}^n 1/(M-i+1)^2 \asymp \frac{n}{M(M-n)}$, and $\sum_{i=1}^n 1/\sqrt{i} \asymp \sqrt{n}$.  Combining these facts, we observe that the second term in~\eqref{eq:eb-cs-proof-1} is $\mc{O}\lp (\log n/n)^{1/4} \rp$, which converges to $0$ with $M \to \infty$.  
                    Hence, we get the following: 
                    \begin{align}
                        A_{M,n} \asymp \frac{M}{\sqrt{n}} \log \lp \frac{M \lambda^*_n}{M-n} \rp = \frac{M \sqrt{2 \log(n/\alpha)}}{\sigma n} \log \lp \frac{1}{1 - n/M} \rp, 
                    \end{align}
                    which implies that 
                    \begin{align}
                        A_{M,n} \convas  \frac{1}{\rho}\log \lp \frac{1}{1 - \rho} \rp \frac{\sqrt{2 \log(2/\alpha)}}{\sigma}. 
                    \end{align}
                
            \subsubsection{Proof of~\Cref{lemma:wor-lemma-2}}
            \label{proof:wor-lemma-2}
                Under the event $\mc{E}_n$, we know that $\muhat_i$ and $V_i$ concentrate around the true mean~($\mu \equiv \mu_M$) and variance~($\sigma^2 \equiv \sigma_M^2$) respectively. As a result, we have the following (again, with $\lambda^*_n = \sqrt{2 \log(2/\alpha)/n \sigma^2})$: 
                \begin{align}
                    &\lambda_{i, n} = \lambda^*_n\lp 1 + r_i \rp = \sqrt{\frac{2 \log(2/\alpha)}{n \sigma^2}}\lp 1 + r_i \rp, \quad \text{and} \quad \muhat_i = \mu + s_i, \quad \text{where} \quad r_i, s_i = \mc{O}\lp \sqrt{\log n/i} \rp. 
                \end{align}
                Using this, we rewrite the weighted mean estimate as follows: 
                \begin{align}
                    \mutilde_n = \frac{\sum_{i=1}^n \lambda^*_n\lp X_i + \frac{\muhat_{i-1} (i-1)}{N-i+1} \rp\lp 1 + r_i\rp }{ \sum_{i=1}^n\lambda^*_n\lp 1 + \frac{i-1}{N-i+1} \rp \lp 1 + r_i \rp} 
                    = \frac{\sum_{i=1}^n \lp X_i + \frac{\muhat_{i-1} (i-1)}{N-i+1} \rp\lp 1 + r_i\rp }{ \sum_{i=1}^n\lp 1 + \frac{i-1}{N-i+1} \rp \lp 1 + r_i \rp}
                    = \frac{\numerator}{\denominator}. 
                \end{align}
                We now evaluate the \numerator and the \denominator separately. First, we introduce the terms
                \begin{align}
                    q_i = 1 + \frac{i-1}{N-i+1}, \quad \text{and} \quad Q_n = \sum_{i=1}^n q_i, 
                \end{align}
                and note that  $1 \leq q_i \leq 1/(1-\rho) = \Theta(1)$, which implies that $n \leq Q_n \leq n/(1-\rho)$. Using this, we obtain the following bound on the \numerator: 
                \begin{align}
                    \numerator &= \sum_{i=1}^n X_i + \sum_{i=1}^n \muhat_{i-1}\frac{i-1}{N-i+1} + \mc{O}\lp \sum_{i=1}^n \lp 1 + \frac{i-1}{N-i+1} r_i\rp \rp \\
                    & = n\mu  + \sum_{i=1}^n \frac{\mu (i-1)}{N-i+1} + \mc{O}\lp \sum_{i=1}^n q_i (r_i + s_i) \rp \\
                    & = Q_n\lp \mu +  \frac{1}{Q_n}\mc{O}\lp \sum_{i=1}^n (r_i + s_i) \rp \rp. 
                \end{align}
                Similarly, the \denominator satisfies 
                \begin{align}
                    \denominator = Q_n + \sum_{i=1}^n q_i r_i  = Q_n \lp 1 + \frac{1}{Q_n} \lp \sum_{i=1}^n r_i \rp \rp .
                \end{align}
                On combining these two observations, we get 
                \begin{align}
                    \mutilde_n &= \frac{ \mu +   \frac{1}{Q_n}\mc{O}\lp \sum_{i=1}^n (r_i + s_i) \rp}{1 + \frac{1}{Q_n} \lp \sum_{i=1}^n r_i \rp} = \lp \mu +  \frac{1}{Q_n}\mc{O}\lp \sum_{i=1}^n (r_i + s_i) \rp \rp \lp 1 + \mc{O}\lp\frac{1}{Q_n}  \sum_{i=1}^n r_i \rp \rp  \\
                    & = \mu +  \mc{O}\lp \frac{1}{Q_n}\sum_{i=1}^n (r_i + s_i) \rp = \mu + \mc{O}\lp \frac{1}{n} \sum_{i=1}^n \sqrt{\frac{\log n}{i}} \rp  = \mu + \mc{O}\lp \sqrt{\frac{\log n}{n}} \rp. 
                \end{align}
        \subsubsection{Proof of~\Cref{lemma:wor-lemma-3}}
        \label{proof:wor-lemma-3}
        \paragraph{$\boldsymbol{N}$ is finite almost surely.}
            We know that $\lim_{M \to \infty} \sqrt{n} \,\widthfan = \lim_{M \to \infty} \sqrt{n} |\CIfan| = 2c_2$ almost surely. Or in other words, the set $\Omega = \{\omega: \lim_{M \to \infty} \sqrt{n} \widthfan(\omega) = 2c_2\}$ has probability $1$, which also implies that the set $\{\omega: \exists N(\omega) < \infty, \text{ s.t. } \sqrt{n} \widthfan(\omega) \leq 4c_2 \text{ for all } M \geq N(w)\}$ also has probability $1$. Since the latter is exactly the set on which $N$ is finite, this proves the first part of the lemma. 
            
        \paragraph{Bound on $\boldsymbol{(X_t - m_t)^2}$.} Let $\mu \equiv \mu_M$ denote the true mean of the $M$ items, and define $\mu_t = (N\mu - \sum_{i=1}^{t-1}X_i)/(N-t+1)$. Then, we have the following:
        \begin{align}
            (X_t - m_t)^2 &= (X_t - \mu_t)^2 + (\mu_t - m_t)^2 + 2(X_t - \mu_t)(\mu_t - m_t) \\
            &\leq (X_t - \mu_t)^2 + 3|\mu_t - m_t|  \\
            & \leq (X_t - \mu_t)^2 + 3 \frac{1}{1-\rho}|\mu - m| \\
            & \leq (X_t - \mu_t)^2 + 3 \frac{1}{1-\rho}\lp |\mu - \mutilde_n| + |\mutilde_n - m|\rp.  \label{eq:wor-proof-3}
        \end{align}
        The first inequality uses the fact that both $(\mu_t - m_t)^2$ and $(X_t - \mu_t)(\mu_t- m_t)$ are upper bounded by $|\mu_t - m_t|$.  The second inequality uses the fact that $|\mu_t - m_t| = N|\mu-m|/(N-t+1)$ and $t \leq n = \rho M + 1$, while the last inequality is a direct consequence of the triangle inequality for the absolute value function.  

        Next, we subtract and add a $\mu$ in $(X_t - \mu_t)^2$, to get an upper bound
        \begin{align}
            (X_t - \mu_t)^2 &\leq (X_t - \mu)^2 + 3|\mu_t - \mu| = (X_t - \mu)^2 + 3 \frac{t-1}{N-t+1}|\muhat_{t-1} - \mu| \\
            & \leq (X_t - \mu)^2 +  \frac{3\rho}{1-\rho} |\muhat_{t-1} - \mu|. \label{eq:wor-proof-4}
        \end{align}
        Combining~\eqref{eq:wor-proof-3} and~\eqref{eq:wor-proof-4}, we get 
        \begin{align}
            \boldsymbol{1}_{\mc{E}_n} \boldsymbol{1}_{M \geq N} (X_t - m_t)^2 &\leq \boldsymbol{1}_{\mc{E}_n} \boldsymbol{1}_{M \geq N}  \bigg( (X_t - \mu)^2 + \mc{O}\lp |\mu - \mutilde_n| + |\mutilde_n - m| + |\muhat_{t-1}-\mu| \rp \bigg)  \\
            &\leq (X_t - \mu)^2 + \boldsymbol{1}_{\mc{E}_n} \boldsymbol{1}_{M \geq N}  \bigg( \mc{O}\lp |\mu - \mutilde_n| + |\mutilde_n - m| + |\muhat_{t-1}-\mu| \rp \bigg). 
        \end{align}
        Finally, we observe that all the three terms inside $\mc{O}(\cdot)$ are small under the event $\mc{E}_n \cap \{M \geq N\}$: 
        \begin{itemize}
            \item From~\Cref{lemma:wor-lemma-2}, we know that $|\mutilde_n - \mu| = \mc{O}\lp \sqrt{ \log n /n} \rp$ under the event $\mc{E}_n$. 
            
            \item Since $m \in \CItilde = [\mutilde_n \pm 2c_2/\sqrt{n}]$, it follows that $|\mutilde_n - m| \leq 4c_2/\sqrt{n}$. The event $\{M \geq N\}$ implies that $\CIfan \subset \CItilde$, which in turn implies that $\CIbet \subset \CItilde$, since $\CIfan$ contains $\CIbet$ by construction. 
            
            \item Finally, under the event $\mc{E}_n$, we know that $|\muhat_{t-1} - \mu| = \mc{O}\lp \sqrt{ \log n /t} \rp$, for all $t \leq n$. 
        \end{itemize}
        Thus, these three observations imply that         
        \begin{align}
            \boldsymbol{1}_{\mc{E}_n} \boldsymbol{1}_{M \geq N} (X_t - m_t)^2 &\leq (X_t - \mu)^2 + \mc{O}\lp \sqrt{ \frac{ \log n}{t} } \rp, 
        \end{align}
        as required.

    \section{Second-order limiting width analysis~(Proposition~\ref{prop:second-order-mpeb})}
    \label{proof:second-order-mpeb}

        \paragraph{Second-order width of~\mpeb CI.}
        Let us denote the term $n \lp w_n^{(\mpeb)} - \gamma_1^{(\mpeb)}/\sqrt{n} \rp$ with $S_n$. To ease notation, we will drop the \mpeb from the superscripts for this proof. Then, observe that 
        \begin{align}
            S_n &= 2\sqrt{n}\lp \sigmahat_n - \sigma \rp \sqrt{2 \log(4/\alpha)} + \frac{14 \log(4/\alpha)}{3}\lp 1 + \frac{1}{n-1} \rp  \\
            & \defined c_1 \sqrt{n}\lp \sigmahat_n - \sigma \rp + c_2 \lp 1 + \frac{1}{n-1} \rp. 
        \end{align}
        Since $c_2(1+1/(n-1))$ converges to $c_2$ as $n \to \infty$, we can focus on the first term above, and refer to it by $T_n$. By~\citet[Example 3.2]{vaart2000asymptotic}, we know that 
        \begin{align}
            \sqrt{n} \lp \sigmahat_n^2 - \sigma^2 \rp \convdist N\lp 0, \mu_4 - \sigma^4 \rp, 
        \end{align}
        where $\mu_4 \defined \mathbb{E}[(X-\mathbb{E}[X])^4]$. Then, by the delta-method~\citep[Theorem 3.1]{vaart2000asymptotic}, we have 
        \begin{align}
            \frac{T_n}{c_1} = \sqrt{n} \lp \sigmahat_n - \sigma \rp = \sqrt{n}\lp \sqrt{\sigmahat_n^2} - \sqrt{\sigma^2} \rp  \convdist N\lp 0, \frac{\mu_4 - \sigma^4}{4} \rp. 
        \end{align}
        Thus, we have proved that the statistic $S_n$ satisfies 
        \begin{align}
            S_n = T_n + c_2 + o_P(1), \quad \text{with} \quad T_n \convdist N\lp 0, 2\log(4/\alpha)(\mu_4 - \sigma^4) \rp, \quad \text{and} \quad c_2 = \frac{14 \log(4/\alpha)}{3}.  
        \end{align}
        
        \paragraph{Second order width of~\prpieb CI.} To analyze the statistic involved in the definition of the second-order limiting width, we introduce the following notation~(dropping \prpieb from the superscripts) 
        \begin{align}
            w_n \equiv w_n^{(\prpieb)} = 2\frac{ \log(2/\alpha) + 4\sum_{t=1}^n \psi_E(\lambda_{t,n}) \lp X_t - \muhat_{t-1} \rp^2 }{ \sqrt{n} \,\times\, \frac{1}{\sqrt{n}} \sum_{t=1}^n \lambda_{t,n}  } = \frac{A_n}{\sqrt{n} B_n}. 
        \end{align}
        From~\Cref{fact:limiting-computations}, we know that $A_n \convas 2 \log(2/\alpha)$ and $B_n \convas b \defined \sqrt{2 \log(2/\alpha)}/\sigma$. Now, the second-order width statistic can be rewritten as  follows, with $\Delta_n = (B_n - b)/b$: 
        \begin{align}
            S_n &= n \lp \frac{A_n}{\sqrt{n}b \lp 1 + \Delta_n  \rp} - \frac{\gamma_1}{\sqrt{n}} \rp = \frac{A_n \sqrt{n}}{b} \lp 1 - \Delta_n + \mc{O}\lp  \Delta_n^2   \rp \rp - \sqrt{n} \gamma_1  \\
            & = \sqrt{n} \lp \frac{A_n - b \gamma_1}{b} \rp  - \sqrt{n} \frac{A_n \Delta_n}{b} + \mc{O}\lp \frac{A_n \sqrt{n} \Delta_n^2}{b} \rp \label{eq:prpl-second-order-proof-1}
        \end{align} 
        An application of Borel-Cantelli Lemma implies that  $\{|\Delta_n| \leq \mc{O}\lp \sqrt{\log n / n} \rp \text{ i.o.}\}$ happens with probability one. Thus, the $\mc{O}\lp \sqrt{n} \Delta_n^2 \rp$ terms converges to $0$ almost surely, and hence also in probability.  We now expand the first term in~\eqref{eq:prpl-second-order-proof-1}: 
        \begin{align}
            \frac{\sqrt{n}}{2} \lp \frac{A_n - b \gamma_1}{b} \rp   &= \sqrt{n} \lp \frac{\log(2/\alpha) + 4 \sum_{t=1}^n \psi_E(\lambda_{t,n})(X_t-\muhat_{t-1})^2 - 2 \log(2/\alpha)}{b} \rp  \\ 
            &= \sqrt{n} \lp \frac{4 \sum_{t=1}^n \psi_E(\lambda_{t,n})(X_t-\muhat_{t-1})^2 -  \log(2/\alpha)}{b} \rp  = \sqrt{n} \frac{C_n}{b}
        \end{align}
        Now, we note that $4 \psi_E(\lambda) = \frac{\lambda^2}{2} + \frac{\lambda^3}{3} + \mc{O}\lp \lambda^4 \rp$, which implies that 
        \begin{align}
            4 \psi_E(\lambda_{t,n}) = \frac{\log(2/\alpha)}{n V_{t-1}} + \frac{1}{3}\lp \frac{2 \log(2/\alpha)}{n V_{t-1}}\rp^{3/2} +  \mc{O}\lp \frac{1}{n^2 V_{t-1}^2} \rp 
        \end{align} 
        This implies that 
        \begin{align}
            \sqrt{n} C_n &= \frac{\log(2/\alpha)}{\sqrt{n}} \sum_{t=1}^n \lp \frac{(X_t - \muhat_{t-1})^2}{V_{t-1}} - 1 \rp + \frac{(2\log(/2\alpha))^{3/2}}{3n} \sum_{t=1}^n \frac{(X_t-\muhat_{t-1})^2}{V_{t-1}^{3/2}} + o_P(1) \\
            & \defined G + H + o_P(1).  \label{eq:prpl-second-order-proof-2}
        \end{align}

        We now show that the term $H$ in~\eqref{eq:prpl-second-order-proof-2} is equal to $2b \log(2/\alpha)/3 + o_P(1)$. 
        \begin{align}
            H &= \frac{(2\log(/2\alpha))^{3/2}}{3n \sigma^3} \sum_{t=1}^n (X_t-\muhat_{t-1})^2 \lp 1 - \frac{3(V_{t-1} - \sigma^2)}{2 \sigma^2} + \mc{O} \lp  \lp \frac{V_{t-1} - \sigma^2}{\sigma^2}\rp^2 \rp \rp \\
            & \stackrel{(i)}{=} \frac{(2\log(/2\alpha))^{3/2}}{3\sigma} + \frac{(2\log(/2\alpha))^{3/2}}{3n \sigma^3} \sum_{t=1}^n \lp  (X_t- \muhat_{t-1})^2  - \sigma^2\rp + o_P(1) \\
             & \stackrel{(ii)}{=} \frac{(2\log(/2\alpha))}{3} \times \frac{ \sqrt{2 \log(2/\alpha)}}{\sigma} + o_P(1) \\
             & = \frac{(2\log(/2\alpha))}{3} \times b + o_P(1). 
        \end{align}
        In this display, $(i)$ uses the fact that $V_t \convas \sigma^2$, and $(ii)$ uses the fact that $\frac{1}{n} \sum_{t=1}^n (X_t - \muhat_{t-1})^2 \convas \sigma^2$. Thus, we have proved that $H/b = 2 \log(2/\alpha) + o_P(1)$.

        Next, let $\sigmahat_{t}^2 = \frac{1}{t-1} \sum_{i=1}^t (X_i - \muhat_t)^2$ denote the unbiased empirical estimate of the variance, and observe that $V_t/\sigmahat_t^2 \stackrel{p}{\to} 1$. This implies that the term $G$ in~\eqref{eq:prpl-second-order-proof-2} satisfies
        \begin{align}
            \frac{G}{\log(2/\alpha)} &= \frac{1}{\sqrt{n}} \sum_{t=1}^n \lp \frac{(X_t - \muhat_{t-1})^2}{\sigma^2 (1+\rho_{t-1}) } - 1 \rp + o_P(1), \quad \text{where} \quad \rho_{t-1} = \frac{\sigmahat^2_{t-1} - \sigma^2}{\sigma^2} \label{eq:prpl-second-order-proof-3}
        \end{align}
        is a zero-mean random variable. 
        Now, observe that 
        \begin{align}
            \frac{1}{\sqrt{n}}\sum_{t=1}^n \frac{(X_t - \muhat_{t-1})^2}{1 + \rho_{t-1}} &= \frac{1}{\sqrt{n}}\sum_{t=1}^n \lp X_t - \muhat_{t-1} \rp^2 \lp 1 - \rho_{t-1} + \mc{O}\lp \rho_{t-1}^2 \rp \rp \\ 
            &= \frac{1}{\sqrt{n}}\sum_{t=1}^n (X_t - \muhat_{t-1})^2 ( 1 + \rho_{t-1})  + o_P(1), 
        \end{align}
        where the second equality uses the fact that $(1/\sqrt{n})\sum_{t=1}^n \mathbb{E}[(X_t - \muhat_{t-1})^2 \rho_{t-1}^2] \leq (1/\sqrt{n})\sum_{t=1}^n \mathbb{E}[\rho_{t-1}^2] = \mc{O}(\log n / \sqrt{n}) \to 0$. Next, we add and subtract $\mu$ in $(X_t - \muhat_{t-1})^2$, to get 
        \begin{align}
            \frac{1}{\sqrt{n}}\sum_{t=1}^n (X_t - \muhat_{t-1})^2(1 + \rho_{t-1}) &= \frac{1}{\sqrt{n}}\sum_{t=1}^n \lp (X_t - \mu)^2 + 2(X_t - \mu)(\mu - \muhat_{t-1}) + (\muhat_{t-1}-\mu)^2 \rp (1 + \rho_{t-1}) \\
            & =  \frac{1}{\sqrt{n}}\sum_{t=1}^n \lp (X_t - \mu)^2 + 2(X_t - \mu)(\mu - \muhat_{t-1})  \rp (1 + \rho_{t-1}) + o_P(1). 
        \end{align} 
        The second equality uses the fact that $(1/\sqrt{n})\sum_{t=1}^n \mathbb{E}[(\muhat_{t-1}-\mu)^2] = \mc{O}(\log n/\sqrt{n}) \to 0$. Note that the mean of each term in the summation is equal to $\sigma^2$. Thus, plugging this back into~\eqref{eq:prpl-second-order-proof-3}, we get that 
        \begin{align}
            \frac{G}{b} &= \frac{\log(2/\alpha)}{\sigma^2 b \sqrt{n}} \sum_{t=1}^n \lp \frac{(X_t-\muhat_{t-1})^2}{1 + \rho_{t-1}} - \sigma^2 \rp \\
            &= \frac{\log(2/\alpha)}{\sigma^2 b \sqrt{n}}\sum_{t=1}^n \lp \lp (X_t - \mu)^2 + 2(X_t - \mu)(\mu - \muhat_{t-1})  \rp (1 + \rho_{t-1}) - \sigma^2 \rp + o_P(1) \\
            & \defined T_{n,1} + o_P(1), 
        \end{align}
        where $T_{n,1}$ is a zero-mean random variable. Thus, the first term in~\eqref{eq:prpl-second-order-proof-1} satisfies 
        \begin{align}
            \frac{\sqrt{n}}{2} \lp \frac{A_n - b \gamma_1}{b} \rp = \frac{2 \log(2/\alpha)}{3} +  T_{n,1} + o_P(1). \label{eq:prpl-second-order-proof-4}
        \end{align}

        It remains to characterize the second term in~\eqref{eq:prpl-second-order-proof-1} to complete the proof. In particular, following the same arguments as above, we can show that 
        \begin{align}
            \sqrt{n}\frac{A_n}{b^2} \lp \frac{1}{\sqrt{n}} \sum_{t=1}^n \lambda_{t,n} - b\rp &=  \frac{4 \log(2/\alpha)}{b} \lp \frac{\sqrt{2 \log(2/\alpha)}}{b \sigma} \frac{1}{\sqrt{n}} \sum_{t=1}^n \rho_{t-1}   \rp  + o_P(1) \\
            & = T_{n,2} + o_P(1),  \label{eq:prpl-second-order-proof-5}
        \end{align}
        where $T_{n,2}$ is also a zero-mean random variable. Thus,~\eqref{eq:prpl-second-order-proof-4} and~\eqref{eq:prpl-second-order-proof-5} together imply that the term $S_n$ introduced in in~\eqref{eq:prpl-second-order-proof-1} satisfies 
        \begin{align}
            S_n = \frac{4\log(2/\alpha)}{3} + T_n + o_P(1), 
        \end{align}
        where $T_n = 2T_{n,1} + T_{n,2}$ is a zero-mean random variable. This completes the proof.

\end{appendix}

\end{document}